\documentclass[11pt,reqno]{amsart} 

\usepackage[utf8]{inputenc}
\usepackage[a4paper, margin=2.7cm]{geometry}
 \usepackage{changepage}
 
\usepackage{amsmath}
\usepackage{amsfonts}
\usepackage{amssymb}
\usepackage{amsthm}
\usepackage{mathtools}
\usepackage{paralist}
\usepackage{booktabs}

\usepackage{natbib}

\usepackage{color}
\usepackage[colorlinks=true,linkcolor=blue,citecolor=blue,pdfborder={0 0 0}]{hyperref}

\usepackage{dsfont}
\usepackage{bm}

\usepackage{graphicx}
\usepackage{enumerate}

\newtheorem{theorem}{Theorem}[section]
\newtheorem{proposition}[theorem]{Proposition}
\newtheorem{lemma}[theorem]{Lemma}
\newtheorem{corollary}[theorem]{Corollary}

\theoremstyle{definition}

\theoremstyle{remark}
\newtheorem{remark}[theorem]{Remark}

\allowdisplaybreaks[4]

\numberwithin{equation}{section}

\newcommand\blfootnote[1]{%
  \begingroup
  \renewcommand\thefootnote{}\footnote{#1}%
  \addtocounter{footnote}{-1}%
  \endgroup
}

\newcommand{\R}{\mathbb{R}}

\newcommand{\N}{\mathbb{N}}
\newcommand{\Prob}{\mathbb{P}}    

\newcommand{\Cb}{\mathbb {C}}

\newcommand{\Nc}{\mathcal{N}}

\newcommand{\eps}{\varepsilon}

\newcommand{\diff}{{\,\mathrm{d}}}
\newcommand{\weak}{\rightsquigarrow}
\newcommand{\Exp}{\operatorname{E}}
\newcommand{\Var}{\operatorname{Var}}
\newcommand{\Cov}{\operatorname{Cov}}

\newcommand{\E}{\op{\mathbb{E}}}

\newcommand{\p}{\mathbb{P}}

\makeatletter
\def\z@first#1#2{#1}
\def\z@second#1#2{#2}
\def\z@zp@selectchar#1#2{
	\IfStrEqCase{#2}{%
		{p}{#1{(}{)}}%
		{P}{#1{)}{(}}%
		{c}{#1{[}{]}}%
		{C}{#1{]}{[}}%
		{a}{#1{\{}{\}}}%
		{A}{#1{\}}{\{}}%
		{i}{#1{[}{]}\!#1{[}{]}}%
		{I}{#1{]}{[}\!#1{]}{[}}%
		{t}{#1{<}{>}}%
		{T}{#1{>}{<}}%
		{b}{#1{|}{|}}%
		{n}{#1{\|}{\|}}%
		{v}{#1{.}{.}}%
	}[#1{(}{)}]%
}
\def\z@zp#1#2\fin#3{
	\z@zp@selectchar{\left\z@first}{#1}#3
	\zifempty{#2}%
	{\z@zp@selectchar{\right\z@second}{#1}}%
	{\z@zp@selectchar{\right\z@second}{#2}}%
}
\newcommand{\zp}[2][p]{\zifempty{#1}{\left(#2\right)}{\z@zp#1\fin{#2}}}
\makeatother

\usepackage{xstring}
\def\zifempty#1#2#3{\def\foo{#1}\ifx\foo\empty\relax#2\else#3\fi}

\newcommand{\ds}{\displaystyle}

\newcommand{\op}{\operatorname} 

\newcommand{\ff}{\frac{1}}

\begin{document}

\title[Testing independence]{Testing for independence in high dimensions based on empirical copulas}

\author{Axel B\"ucher}
\author{Cambyse Pakzad}

\address{Heinrich-Heine-Universit\"at D\"usseldorf, Mathematisches Institut, Universit\"atsstr.~1, 40225 D\"usseldorf, Germany.}
\email{axel.buecher@hhu.de}
\email{cambyse.pakzad@hhu.de}

\date{\today}

\begin{abstract} 
Testing for pairwise independence for the case where the number of variables may be of the same size or even larger than the sample size has received increasing attention in the recent years. We contribute to this branch of the literature by considering tests that allow to detect higher-order dependencies. The proposed methods are based on connecting the problem to copulas and making use of the Moebius transformation of the empirical copula process; an approach that has already been used successfully for the case where the number of variables is fixed. Based on a martingale central limit theorem, it is shown that respective test statistics converge to the standard normal distribution, allowing for straightforward definition of critical values. The results are illustrated by a Monte Carlo simulation study.

\end{abstract}

\keywords{Empirical copula process; high dimensional statistics; higher order dependence; Moebius transform; rank based inference}

\maketitle


\section{Introduction}

Suppose $\ds{\bm X_1, \dots, \bm X_n}$, $\ds{\bm X_i=(X_{i1}, \dots, X_{id})}$ is an i.i.d.\ sample of $d$-variate observations with joint cumulative distribution function (cdf) $F$ and continuous marginal cdf's $F_1, \dots, F_d$. A generic random variable with c.d.f.\ $F$ will be denoted by $\ds{\bm X=(\bm X_1, \dots, \bm X_d)^\top}$. We are interested in testing for the hypothesis
\begin{align} \label{eq:indep}
H: X_1, \dots, X_d \text{ are mutually independent}
\end{align}
in a high dimensional asymptotic regime, where $d=d(n)$ is allowed to be larger than $n$ and grows to infinity with increasing sample size $n$. 

The problem has recently attracted increasing attention. Motivated by the Gaussian case, where mutual independence is equivalent to pairwise independence, much work has been devoted to testing whether covariance or (rank) correlation matrices are equal to the identity matrix, see, e.g., \cite{LedWol02, Sch05, CaiJia11, CheZhaZho10,JiaQi15, LeuDrt18, HanCheLiu17, Yao18, HanWu20, Drt20}. The approaches that have motivated the present work may be categorized into two types: maximum-type statistics or $L^2$-type statistics. The former aim at detecting possibly sparse alternatives where only a few pairs of coordinates of $\bm X$ are highly dependent (see, e.g., \citealp{HanCheLiu17, Drt20}). The latter aim at detecting possibly dense alternatives, where many of the pairs of coordinates of $\bm X$ are only weakly dependent, but the overall non-independence signal is large \citep{LeuDrt18, Yao18}. Quite remarkably, to the best of our knowledge, the testing problem has not yet been approached by copula methods, despite the longstanding success of respective methods in the `fixed $d$' case (see \citealp{Gen19} and the discussion below). 

The purpose of this paper is two-fold: first of all, we take up the last point from the previous paragraph and aim at transferring a successful methodological approach from the literature on statistics for copulas (which has only been studied for the fixed $d$ case) to the current high-dimensional setting.
Here, the underlying theoretical connection to copulas \citep{Nel06} is quite obvious: by Sklar's theorem \citep{Skl59}, the continuity assumption on the marginals implies that the hypothesis in \eqref{eq:indep} is equivalent to 
\begin{align} \label{eq:Hall}
H: C = \Pi_d,
\end{align}
where $C$ is the unique copula associated with $\bm X$ and where $\Pi_d$ denotes the $d$-dimensional independence copula defined as 
\[
\Pi_d(\bm u)=\textstyle \prod_{j=1}^d u_j, \qquad \bm u=(u_1, \dots, u_d)^\top \in [0,1]^d.
\] 
Given the simplicity of the hypothesis, it may in hindsight be regarded quite natural that essentially the entire field of statistics for copulas has emerged from a series of papers on independence testing in the `fixed $d$' case from around 1980, see \cite{Deh79, Deh81a, Deh81b}. The approach that we will take up in this paper is closely connected to \cite{Deh81a} and a thorough subsequent analysis in \cite{GenRem04}; see also \cite{GenQueRem07,KojHol09, Gen19} for further contributions. In particular, our test statistics will rely on the Moebius transform of the empirical copula process. The latter process has been extensively studied in the fixed $d$ case, see, e.g., \cite{Seg12}.

The second objective of the present paper is motivated by the fact that the tests mentioned in the second paragraph (e.g., \citealp{LeuDrt18}, which is closest in spirit to our approach) are only consistent against (either sparse or dense) alternatives which involve some form of pairwise dependence. For instance, situations where all pairs of $\bm X$ are independent, but some triples are dependent cannot be detected. In other words, the proposed tests should rather be regarded as tests for  the weaker hypothesis
\[
H_2:  X_1, \dots, X_d \text{ are pairwise independent},
\]
or, equivalently, that all bivariate margins of $C$ are  equal to the bivariate independence copula, i.e.,  
\begin{align} \label{eq:H2}
H_2: C_A = \Pi_2 \text{ for all }  A \in I_d(2), 
\end{align}
where $I_d(k)$ 
denotes the set of all $A \subset \{1, \dots, d\}$ of cardinality $k$ and where $C_A$ denotes the $|A|$-dimensional marginal copula of $C$ belonging to the sub-vector $\bm X_A=(X_j)_{j\in A}$.  The hypothesis in \eqref{eq:H2} may easily be extended to $k\in\{3, \dots, d\}$:
\begin{align} \label{eq:Hk}
H_k: C_A = \Pi_{k} \text{ for all }  A \in I_d(k),
\end{align}
i.e., $\bm X$ is $k$-wise dependent.
With a slight abuse of notation, $H$ in \eqref{eq:Hall} may be written as $H=\bigcap_{k=2}^d H_k$ with $H_2 \subset H_3 \subset \dots \subset H_d$. 

In view of the latter chain of subset relations, one might be tempted to define a single test statistic that measures the discrepancy from $C$ to $\Pi_d$; for instance, the Kolmogorov distance between the empirical copula and the independence copula, whose finite-sample distribution may easily be approximated by simulation (for any $n,d$). However, doing so is not helpful in terms of power properties due to the potentially little signal to noise ratio for large $d$. In fact, one might  argue that typical alternatives from practice should involve some form of lower order (pairs, triples, ...) dependencies, and as such it seems reasonable to come up with test statistics that are designed to detect such alternatives and, hence, build up upon measures of the respective lower-order discrepancies (for instance, on pairwise dependence measures as in \citealp{LeuDrt18}, among others). In the `fixed $d$' case, this is exactly the purpose of the test statistics based on the Moebius transform of the empirical copula process, see \cite{GenRem04}. 

In contrast to the `fixed $d$' case however, numerical problems arise when $d$ is even moderately large (say, $d\ge 20$). Those arise from the fact that the cardinality of $I_d(k)$ is $\binom{d}{k}$, which, for $d$ fixed and even, is maximal at $k/2$ and grows exponentially in $d$ (for instance, $|I_{20}(10)|=184\, 756$). As a consequence, in the high-dimensional regime, we propose to compromise between testing for $H_2$ only and testing for the full null $H=H_d$ by testing for the intermediate hypothesis $H_m$
with some fixed and finite $m$. Here, $m$ should be chosen by the statistician based on external knowledge of the problem at hand (e.g., guided by the question of which alternatives one might expect or be interested to detect, or by the growth conditions connecting $d$, $n$ and $m$ derived in our main result, Theorem~\ref{theo:main}) and, possibly,  on the availability of computational resources. In particular, it is worthwhile to mention that, given that the null hypothesis is a singleton, exact critical values may in principle be obtained by simulation for any combination of $(n,d)$; note that this argument applies to any distribution-free test of independence. It is mostly for computational reasons (and, of course,  for simplicity) that tests based on asymptotic critical values provide a valid alternative.


The remaining parts of this paper are organized as follows. In Section~\ref{sec:bb}, we introduce the main building blocks for our tests statistics. We further explain their connection with the empirical copula process and the Moebius transfrom and derive some basic properties regarding their first and second order moments. In Section~\ref{sec:asy}, we aggregate the building blocks by suitable summation, and formulate a respective limit theorem. It is further discussed how the result immediately suggests asymptotic tests for independence. The general proof outline is discussed in Section~\ref{sec:proofmain}, where we also formulate important intermediate results. The finite-sample properties of the obtained tests are investigated in Section~\ref{sec:sims} based on a large scale Monte Carlo simulation study. Finally, the detailed proofs are collected in Section~\ref{sec:proofs}, with some 
straightforward calculations and the handling of a specific extension to higher dimensions postponed to a supplementary material.

All convergences are for $n$ to infinity, if not mentioned otherwise. Weak convergence of random variables and probability distributions is denoted by `$\weak$'.

\section{Building blocks for the  test statistics} \label{sec:bb}

Throughout this section, fix $m\in \N$, and let $n$ be sufficiently large such that $d=d(n) \ge m$.
Our test statistic for testing $H_m$ in \eqref{eq:Hk} will be based on combining test statistics for each individual $H_k$ with $k\in\{2, \dots, m\}$. For the latter, we follow the approach in \cite{GenRem04} based on the Moebius transform of the empirical copula process. For the ease of reading, we start by recapitulating the most important derivations from that paper.

The basic underlying ingredient is the \textit{empirical copula}, defined as
\[
\hat C_n(\bm u) = \frac1n \sum_{i=1}^n 1_{\{\hat {\bm U_i} \le \bm u\}}, \qquad \bm u=(u_1, \dots, u_d)^\top \in [0,1]^d,
\]
where the inequality is understood componentwise and where $\hat{\bm U_i} = (\hat U_{i1}, \dots, \hat U_{id})^\top$ denotes observable pseudo-observations from $C$ defined as
\begin{align*} 
\hat U_{ip} =  \frac{R_{ip}}{n+1},\quad R_{ip}=\sum_{j=1}^{n} 1_{\{ X_{j p} \le X_{ip} \}},\quad p=1,\ldots,d.
\end{align*}
Note that $R_{ip}$ is the (max-)rank of $X_{ip}$ among $X_{1p}, \dots, X_{np}$.
Further recall that, in the `fixed $d$' case, $\hat C_n$ is a consistent non-parametric estimator for the unknown copula $C$ \citep{Seg12}.
The rescaled estimation error under the hypothesis of mutual independence, that is
\begin{align*} 
\Cb_n(\bm u) = \sqrt n \{ \hat C_n (\bm u) - \Pi(\bm u)\},
\end{align*}
is commonly referred to as the \textit{empirical copula process}, provided that $C=\Pi$. For $A\subset \{1, \dots, d\}$ with $|A|>1$, the $A$-margin of $\Cb_n$ will be denoted by
\[
\Cb_{n,A}(\bm u) 
= 
\Cb_n(\bm u^A)
=
\sqrt n\{ \hat C_{n,A} (\bm u) - \Pi_A(\bm u)\}.
\]
Here, for a c.d.f.\ $G$ on $[0,1]^d$, $G_A$ is  defined as
\[
G_A((u_j)_{j\in A}) = G(\bm u^A),  \qquad (u_j)_{j\in A} \in [0,1]^{|A|},
\]
where $\bm u^A \in [0,1]^d$ has $p$th component $u_p^A =  u_p1_{\{p\in A\}} + 1_{\{p\notin A\}}$. 
Occasionally, we also use the notation $\bm u^A$ for vectors $\bm u \in[0,1]^d$, which should note yield any confusion.
Finally, we also write
$G_A(\bm u) = G(\bm u^A)$ for $\bm u \in [0,1]^d$, despite the fact that $G_A$ is a function on $[0,1]^{|A|}$.

Next, the Moebius transformation of $\Cb_n$ is defined, for $\bm u=(u_1, \dots, u_d)^\top \in [0,1]^d$, as:
\begin{align*}
\Cb^{\rm M}_{n,A}(\bm u) 
&= 
\sum_{B \subset A} (-1)^{|A \setminus B|} \Cb_n(\bm u^{B}) \prod_{j \in A \setminus B} u_j  =
\frac1{\sqrt n} \sum_{i=1}^n \prod_{p \in A}\left(  1_{\{\hat U_{ip} \le u_p\}} - u_p\right) .
\end{align*}
 In the `fixed $d$' case, we have the  remarkable property that the processes $\Cb^{\rm M}_{n,A}(\cdot)$ with $A \subset \{1, \dots, d\}$ of cardinality larger than $1$ are asymptotically independent (and Gaussian). This eventually allows to obtain limit results for test statistics that suitably aggregate over various sets $A$, even in the high dimensional regime, see Section~\ref{sec:asy}. 

Following \cite{GenRem04}, we measure the non-independence of $\bm X_A$ based on Cram\'er-von-Mises statistics, i.e.,
\[
\bar S^{\rm M}_{n,A} =\int_{[0,1]^{|A|}} \{  \Cb_{n,A}^{\rm M}(\bm u)  \}^2 \diff \Pi_{A}((u_j)_{j\in A})
\]  
(clearly, other functionals may be used as well, as for instance the supremum norm of $\Cb_{n,A}^{\rm M}$); a statistic that is stochastically bounded under $H$ and diverges to infinity in probability under (fixed) alternatives.
In the interest of improved efficiency, \cite{GenRem04}, p.~347,  propose to  use the previous definition with a version of the empirical copula process that is centered under the null hypothesis $C=\Pi$, namely
\[
\Cb_{n,A}^{\rm M} = \frac1{\sqrt n} \sum_{i=1}^n \prod_{p \in A} \left(  1_{\{\hat U_{ip} \le u_p\}} - U_n(u_p) \right) ,
\]
where  $U_n=U_n(t)= \min\{ \lfloor (n+1)t \rfloor/n,1\}$ denotes the cdf of a random variable that is uniformly distributed on $\{1/(n+1), \dots, n/(n+1)\}$. The respective Cram\'er-von-Mises statistics may then be calculated explicitly, and one obtains:
\begin{align} \label{eq:snam}
S_{n,A}^{\rm M} = \int_{[0,1]^{|A|}} \{   
\Cb_{n,A}^{\rm M}(\bm u)  \}^2 \diff \Pi_{A}((u_j)_{j\in A})
=
\frac1n \sum_{i, j=1}^n   \prod_{p \in A}  I_{i,j}^{(p)},
\end{align}
where
\begin{align*} 
I_{i,j}^{(p)}  =  \frac{2n+1}{6n} + \frac{R_{ip}(R_{ip}-1)}{2n(n+1)} + \frac{R_{jp}(R_{jp}-1)}{2n(n+1)}  - 
\frac{\max(R_{ip}, R_{jp})}{n+1}.
\end{align*}
Note that $S_{n,A}^{\rm M}$ may be written as $\frac2n \sum_{i<j}^n   \prod_{p \in A}  I_{i,j}^{(p)} + \frac1n\sum_{i=1}^n  I_{i,i}^{(p)}$, with the first sum being a rescaled u-statistics. Similar u-statistics have been treated in \cite{LeuDrt18}; however, unlike in that paper, the ranks $R_{ip}$ cannot be replaced by their unobservable counterparts $n F_p(X_{ip})$ without changing the finite-sample and asymptotic behavior, which complicates the derivation of asymptotic theory.

Now, large values of $S_{n,A}^{\rm M}$ provide evidence against mutual independence of $\bm X_A$. For the purpose of aggregating over various index sets $A$, it is helpful to calculate expectation and variance of $S_{n,A}^{\rm M}$. For $A\subset\{2, \dots, d\}$ such that $|A|=k \in\{2, \dots d\}$, let 
\begin{align} \label{eq:musn2}
\mu_n(k) := \E_{H_k}\left[S^{\rm M}_{n,A}\right], \qquad \sigma_n^2(k) \coloneqq \Var_{H_k}(S_{n,A}^{\rm M}),
\end{align}
where $\E_{H_k}$ denotes expectation under $H_k$.

\begin{lemma}\label{lem:snvar}
Irrespective of the dependence of $\bm X$ we have, for any $p\in\{1, \dots, d\}$ and $i,j\in\{1, \dots, n\}$,
\begin{align} \label{eq:expi}
\E[I_{i,j}^{(p)}]
=
\left( \frac16 - \dfrac1{6n}\right) 1_{\{i=j\}} - \frac1{6n} 1_{\{i\neq j\}}.
\end{align}
Moreover, if $H_k$ is met,
\begin{align} \label{eq:munk}
\mu_n(k) = \left(\ff{6}-\ff{6n} \right)^{k} +\left(n-1 \right)\left( \frac{-1}{6n}\right)^{k},
\end{align}
and
\begin{align} \label{eq:varn}
\sigma_n^2(k) = \frac2{90^k}  \{ 1+O(n^{-1}) \}.
\end{align}
Finally, for $k\in\{2,3\}$, we have
\begin{align*}
\sigma_n^2(2)
&=
\frac{(n-2)^2(n-1)(8n+1)}{32400n^2(n+1)^2}   
=
\frac2{90^2} - \frac{11}{6480 n} + \frac{161}{32400n^2} + O(n^{-3}), \\
\sigma_n^2(3)
&=
\frac{(n-2)(n-1)(16n^5-96n^4+359n^3-269n^2-963n-370)}{5832000 \cdot n^4(n+1)^3}  \\
&=
\frac2{90^3} - \frac{1}{30375 n} + \frac{1207}{5832000 n^2} + O(n^{-3}).
\end{align*}
\end{lemma}

The proof is given in Section~\ref{subsec:moml1}. It is worthwhile to mention that the leading term in the asymptotic expression for $\sigma_n^2(k)$, i.e., $2 \cdot 90^{-k}$, can easily seen to be equal to the variance of the (fixed~$d$) limiting distribution of $S_{n,A}^{\rm M}$ stated in Proposition~4.1 in \cite{GenRem04}.

\section{$L_2$-type test statistics and asymptotic results} \label{sec:asy}

The basic building blocks $S_{n,A}^{\rm M}$ from \eqref{eq:snam} may be aggregated in various ways over index sets $A \subset\{1,\dots, d\}$. Throughout, following \cite{LeuDrt18} and \cite{Yao18}, we opt for $L^2$-type aggregation, and leave aggregation based on maxima for future research (which would involve a completely different theoretical approach, see \citealp{HanCheLiu17}). More precisely,  for some given $k\in\{2, \dots, m\}$, we consider the following aggregation over all sets $A\subset\{1,\dots, d\}$ with $|A|=k$:
\begin{align} \label{eq:tnell}
T_n(k) =\sum\limits_{\substack{A \subset \{1, \dots, d\} \\ |A|=k }} S_{n,A}^{\rm M} .
\end{align}
Note that $T_n(k)$ is related to the linear combination  rule in Section 4.3.1 in \cite{GenQueRem07}. 

Now, as stated in Section~\ref{sec:bb}, the results in \cite{GenRem04} imply that the $\binom{d}{k}$ summands in $T_n(k)$ are asymptotically independent in a `fixed d' scenario. Recalling $\mu_n(k)$ and $\sigma_n^2(k)$ from \eqref{eq:musn2} (with explicit formulas provided in Lemma~\ref{lem:snvar}),  this motivates the introduction of the following scaling sequences: 
\begin{align*}
\nu_n(k) &= \binom{d}k \cdot\mu_n(k) 
=  \binom{d}k\cdot \left\{\left(\ff{6}-\ff{6n} \right)^{k} +\left(n-1 \right)\left( \frac{-1}{6n}\right)^{k} \right\},
\end{align*}
and
\begin{align} \label{eq:deltan}
\bar \delta_n(k) =  \sqrt{\sigma_n^2(k)\cdot \binom{d}k },
\qquad
\delta_n(k) =  \sqrt{\frac{2}{90^k}\cdot \binom{d}k}.
\end{align}
The following theorem is the main result of this paper, which immediately gives rise to consistent asymptotic tests. 

\begin{theorem} \label{theo:main}
Under the null hypothesis $H_5$, if $d=d_n\to\infty$, we have
\[
\frac{T_n(2) - \nu_n(2)}{\delta_n(2)} \weak  \Nc(0,1).
\]
Moreover, under the null hypothesis $H_{4m-3}$ and  if $d=d_n\to\infty$ such that $d=o\left( n^{\frac{1}{m-1}}\right) $, we have
\[
\left( \frac{T_n(2)-\nu_n(2)}{\delta_n(2)}, \ldots, \frac{T_n(m)-\nu_n(m)}{\delta_n(m)}\right)  \weak  \Nc(0,1)^{\otimes (m-1)}.
\]
Finally, the same results are true if $\delta_n(m)$ is replaced by $
\bar \delta_n(m).
$
\end{theorem}

\begin{corollary}\label{cor:test} 
Under the conditions of Theorem~\ref{theo:main}, we have
\[
\bar T_n(m) = \frac1{\sqrt{m-1}} \sum_{k=2}^m \frac{T_n(k)- \nu_n(k)}{\bar \delta_n(k)} \weak \Nc(0,1),
\]
which implies that the test which rejects $H_m$  iff $\bar T_n(m) > u_{1-\alpha}$, the $1-\alpha$-quantile of the standard normal distribution, has asymptotic level $\alpha$. The same result is true if $\bar \delta_n(k)$ is replaced by $
\delta_n(k)$, which does however yield worse finite-sample performance.
\end{corollary}

Remarkably, the result of Theorem~\ref{theo:main} for $m=2$ does not require any condition on $d$ at all, which is akin to the pairwise independence tests from \cite{LeuDrt18}; see their Theorem 4.1. For detecting higher order dependencies $m\ge 3$ however, the condition $d_n=o(n^{1/(m-1)})$ becomes more and more restrictive. This finding will later be confirmed by the simulation study in Section~\ref{sec:sims}. It is also worthwhile to mention that some of the intermediate results from the proofs straightforwardly extent to the case $\{T_n(k_n)-\nu_n(k_n)\}/\bar \delta_n(k_n)$ with $k_n\to\infty$ at a sufficiently small rate. However, in view of the following remark, we do not pursue a rigorous extension any further.

\begin{remark}[Computational Cost] \label{rem:cost}
The computational cost of calculating ranks for a sample of size $n$ is $\Theta(n\log(n))$ (\citealp{Cor09}, Section II). Given the ranks, calculating $S_{n,A}$ requires $\Theta(k n^2)$ computations for a single set $|A| \subset\{1, \dots, d\}$ with $|A|=k$. Hence, calculating $T_n(k)$ for $k$ fixed requires $\Theta(k n^2 d^k)$ computations, again given the ranks. Overall, calculating $\bar T_n(m)$ has a computational cost of $\Theta(d n \log(n)  + \sum_{k=2}^m k n^2 d^k) = \Theta(m n^2 d^m)$.	
\end{remark}


\section{Proof of Theorem~\ref{theo:main}} \label{sec:proofmain}

In this section, we illustrate how the assertion  of Theorem~\ref{theo:main} can be obtained from a sequence of intermediate results. Those intermediate results are summarized in suitable propositions that will be proven in Section~\ref{sec:proofs} below.

By Slutsky's lemma and in view of \eqref{eq:varn} from Lemma~\ref{lem:snvar}, it is sufficient to prove the results of Theorem~\ref{theo:main}  for the expressions that involve $\delta_n(\ell)$. The proof is decomposed into three steps:

\medskip
\noindent
\textbf{Step 1: Reduction to centred summands}.
Versions of $S_{n,A}^{\rm M}$ in \eqref{eq:snam} and $T_n(k)$ in \eqref{eq:tnell} that are based on centred summands may be defined as follows:
\begin{align} \label{eq:snamt}
\tilde T_n (k) :=\sum\limits_{\substack{A \subset \{1, \dots, d\} \\ |A|=k }} \tilde S_{n,A}^{\rm M},  \qquad 
\tilde S_{n,A}^{\rm M} 
:=
\frac1n \sum_{i, j=1}^n   \prod_{p \in A} \tilde  I_{i,j}^{(p)},
\end{align}
where, in view of \eqref{eq:expi},
\begin{align} 
\tilde I_{i,j}^{(p)}  
&= 
I_{i,j}^{(p)} -  \E[ I_{i,j}^{(p)} ]  
\label{eq:bii}
=
I_{i,j}^{(p)} - \Big(\frac16 - \dfrac1{6n}\Big)  1_{\{i=j\}} + \frac1{6n} 1_{\{i\neq j\}} \\ 
\nonumber
&=
\begin{cases}
 \dfrac{n+1}{3n} + \dfrac{R_{ip}(R_{ip}-1)}{2n(n+1)} + \dfrac{R_{jp}(R_{jp}-1)}{2n(n+1)}  - 
\dfrac{\max(R_{ip}, R_{jp})}{n+1} &,\quad i\ne j, \\
\dfrac{n+2}{6n} + \dfrac{R_{ip}^2}{n(n+1)} - \dfrac{R_{ip}}{n} & , \quad i=j.
\end{cases}
\end{align}

\begin{proposition}\label{prop:tt}
We have
\[
\frac{T_n(2)-\nu_n(2)}{\delta_n(2)}= \frac{\tilde T_n(2)}{\delta_n(2)}.
\]
Moreover, for any fixed $k\in\N_{\ge 3}$, if $H_{4k-7}$ is met and if $d=o\left( n^{\frac{1}{k-1}}\right) $, we have, as $n\to\infty$,
\[
\frac{T_n(k)-\nu_n(k)}{\delta_n(k)}= \frac{\tilde T_n(k)}{\delta_n(k)} + o_\Prob(1).
\]
\end{proposition}

The proof is given in Section~\ref{sec:ps1}.
The result implies that, under the given conditions on $d$, we may deduce the weak convergence result in Theorem~\ref{theo:main} from respective weak convergence results on the tilde versions. Note that Proposition~\ref{prop:tt} is the only result requiring the growth condition on $d=d_n$.

\medskip
\noindent
\textbf{Step 2: Negligibility of summands $\bm {i=j}$.} The double sum in the definition of $\tilde S_{n,A}^{\rm M}$ in \eqref{eq:snamt} may be split into
$
\tilde S_{n,A}^{\rm M}  = \tilde M_{n,A} + \tilde N_{n,A},
$
where
\begin{align} \label{eq:mnnn}
\tilde M_{n,A} = \frac2n \sum_{i<j} \prod_{p \in A} \tilde I_{i,j}^{(p)}, \qquad
\tilde N_{n,A} = \frac1n \sum_{i=1}^n \prod_{p \in A} \tilde I_{i,i}^{(p)}.
\end{align}
As a consequence, we may write $\tilde T_n(k) = \tilde M_n(k) +\tilde N_n(k) $, where
\begin{align} \label{eq:mnnn2}
\tilde M_n(k) =\sum\limits_{\substack{A \subset \{1, \dots, d\} \\ |A|=k }} \tilde M_{n,A}, \qquad
\tilde N_n(k) =\sum\limits_{\substack{A \subset \{1, \dots, d\} \\ |A|=k }}\tilde N_{n,A}.
\end{align}
The proof of the next result is given in Section~\ref{sec:ps2}.

\begin{proposition}\label{prop:nn}
For fixed $k\in\N_{\ge 2}$, if $H_{2k}$ is met, we have 
$\frac{\tilde N_n(k)}{\delta_n(k)}  =O_\Prob(n^{-\frac{1}{2}})
$.
As a consequence,
\[
 \frac{\tilde T_n(k)}{\delta_n(k)} = \frac{\tilde M_n(k)}{{\delta_n(k)}} + o_\Prob(1).
\]
\end{proposition}

\medskip
\noindent
\textbf{Step 3: Asymptotic normality of ${(\tilde M_n(2), \dots, \tilde M_n(m))}$.} The $\tilde M$-terms may be identified as martingales, and a martingale central limit theorem may be applied to deduce the following result, whose proof is given in Sections~\ref{sec:ps3a} and \ref{sec:ps3b}. 

\begin{proposition}
\label{prop:mm} 
Fix $m\in\N_{\ge 2}$, and suppose that $H_{4m-3}$ is met. Then
\[
\left( \frac{\tilde M_n(2)}{\delta_n(2)}, \dots, \frac{\tilde M_n(m)}{\delta_n(m)}\right)  \weak  \Nc_{m-1}(0,I_{m-1}),
\]
with $I_{m-1}$ the $(m-1)$-dimensional unit matrix.
\end{proposition}

Finally, Theorem~\ref{theo:main} is a mere consequence of Propositions~\ref{prop:tt}, \ref{prop:nn} and \ref{prop:mm}. \qed

\begin{remark}
The proof of Proposition~\ref{prop:mm} is based on a direct application of a martingale central limit theorem to the rank-based statistics $\tilde M_n(k)$. This is in contrast to many other asymptotic results on the empirical copula process and functionals thereof, which often rely on first reducing the problem to the case of `known marginals' (i.e., $\hat{\bm U}_i$ gets replaced by $\bm U_i$ with $U_{ip}=F_p(X_{ip})$) and then working with the independent sample $\bm U_1, \dots, \bm U_d$. An important and powerful intermediate result when following this approach is the Stute representation going back \cite{Stu84}; see Proposition 4.2 in \cite{Seg12} for a formulation under feasible, non-restrictive smoothness conditions.  The authors' attempts of generalizing the result of Proposition 4.2 in \cite{Seg12}  to the case $d=d_n\to\infty$ (irrespective of whether $H_0$ is met or not) were only partly successful in that we needed the restrictive growth condition $d_n=o(n^{-\frac{1}{6}}(\log n)^{-\frac{1}{2}})$; the error in Equation (4.1) in \cite{Seg12} was then shown to be of the order $O(d_n^{\frac{3}{2}} n^{-\frac{1}{4}} (\log n)^{\frac{3}{4}})$ almost surely. In view of this restriction, we did not pursue this approach any further.
\end{remark}

\section{Finite-sample results} \label{sec:sims}

A large-scale simulation study was performed to investigate the level and power properties in finite-sample situations. Special attention is paid to models for which we have pairwise independence, but higher order dependence. Note that the high-dimensional independence  tests proposed in  \cite{LeuDrt18} do not have any power against such alternatives by construction. 

The design of the simulation study is akin to \cite{LeuDrt18}. In particular, several data generating processes are considered for each combination of 
$
n\in \left\{16,32,64,128\right\}$ and  $d\in \left\{4,8,16,32,64,128,256\right\}.
$
Note that we can restrict attention to any arbitrary marginal distribution functions. The following models are considered:

\begin{compactenum}
\item \textbf{Mutual $d$-variate Independence.} $X_1, \dots, X_d$ are independent, i.e., $H=H_d$ from \eqref{eq:Hall} is met.
\item  \textbf{Gaussian copula with constant pairwise dependence.}   
Let $\bm X\sim \Nc_d(\bm 0, \bm \Sigma)$ with $\bm \Sigma=(\sigma_{pq})_{p,q=1}^d$ and
$
\sigma_{pq}=  \bm 1_{\{p=q\}} + \rho  \bm 1_{\{p\neq q\}},
$
where the correlation parameter $\rho=\rho_d >0$ is chosen in such a way that the sum over all pairwise values of Kendall's tau, i.e., 
\[
\|\bm \tau\|_2^2 := \sum_{1\le p<q \le d} (\tau_{pq})^2  = \frac{d(d-1)}{2} \tau_{12}^2
\]
is constant in $d$ and takes values in $\{0.1, 0.3, 0.7\}$.  Recall that $\tau_{pq}=\frac2\pi \arcsin(\sigma_{pq})$ (\citealp{Fan02}, Theorem 3.1), whence
$
\rho = \rho_d=\sin\Big( \frac\pi2 \sqrt{\frac{2 \| \bm \tau\|^2_2}{d(d-1)}} \Big)=O(d^{-1}).
$
The model is taken from \cite{LeuDrt18} and allows for a comparison with their results.

{\footnotesize
\begin{table}[!t]
\centering
\begin{tabular}{ c|c|| rrrrrrr || rrrrrrr } 
\hline \hline
\multicolumn{2}{l}{} & \multicolumn{7}{c}{\textit{Finite variance scaling $\bar \delta_n$}} & \multicolumn{7}{c}{\textit{Asymptotic variance scaling $\delta_n$}} \\ \hline \hline
Test & $n \setminus  d$ &\multicolumn{1}{c}{$4$} & \multicolumn{1}{c}{$8$} & \multicolumn{1}{c}{$16$} & \multicolumn{1}{c}{$32$}& \multicolumn{1}{c}{$64$}& \multicolumn{1}{c}{$128$}& \multicolumn{1}{c||}{$256$} 
&\multicolumn{1}{c}{$4$} & \multicolumn{1}{c}{$8$} & \multicolumn{1}{c}{$16$} & \multicolumn{1}{c}{$32$}& \multicolumn{1}{c}{$64$}& \multicolumn{1}{c}{$128$}& \multicolumn{1}{c}{$256$} \\
\hline
 \addlinespace[.2cm]
  \multicolumn{9}{l}{\quad \textit{Model 1: mutual independence}} \\ \hline
$\mathcal S_{2}$ &  & 7.2 & 4.8 & 6.6 & 6.2 & 5.2 & 6.6 & 6.8 & 4.8 & 1.6 & 2.8 & 2.6 & 3.2 & 2.8 & 3.0 \\ 
  $\mathcal S_{3}$ &  & 6.8 & 7.2 & 15.4 & 19.6 & 28.0 & 34.2 & 37.4 & 2.2 & 3.6 & 6.2 & 9.4 & 19.8 & 29.0 & 31.4 \\ 
  $\mathcal T_{3}$ & 16 & 3.4 & 5.8 & 6.2 & 10.4 & 17.4 & 26.2 & 29.8 & 1.6 & 1.2 & 1.6 & 6.0 & 10.0 & 18.6 & 23.8 \\ 
  $\mathcal S_{4}$ &  &  &  &  &  &  &  &  & 2.2 & 18.2 & 29.4 & 34.4 & 43.4 & 46.0 & 42.4 \\ 
  $\mathcal T_{4}$ &  &  &  &  &  &  &  &  & 2.4 & 11.4 & 24.2 & 30.6 & 39.8 & 45.4 & 42.0 \\  \hline
  $\mathcal S_{2}$ &  & 4.8 & 7.2 & 3.4 & 6.4 & 5.0 & 6.0 & 5.2 & 4.0 & 5.4 & 3.0 & 5.0 & 3.2 & 4.2 & 4.0 \\ 
  $\mathcal S_{3}$ &  & 9.0 & 5.8 & 8.2 & 16.6 & 20.6 & 23.4 & 31.6 & 5.0 & 4.2 & 5.0 & 10.8 & 15.6 & 19.4 & 27.8 \\ 
  $\mathcal T_{3}$ & 32 & 4.8 & 5.0 & 4.2 & 8.2 & 12.0 & 14.2 & 22.6 & 4.0 & 3.4 & 2.6 & 6.2 & 8.0 & 9.8 & 19.0 \\ 
  $\mathcal S_{4}$ &  &  &  &  &  &  &  &  & 6.6 & 13.4 & 26.6 & 40.2 & 44.8 & 42.6 & 46.6 \\ 
  $\mathcal T_{4}$ &  &  &  &  &  &  &  &  & 4.4 & 9.6 & 18.0 & 36.0 & 42.4 & 40.8 & 45.8 \\  \hline
  $\mathcal S_{2}$ &  & 7.2 & 6.2 & 7.6 & 5.2 & 5.2 & 4.8 & 5.6 & 6.0 & 5.8 & 6.6 & 4.6 & 4.4 & 3.2 & 4.6 \\ 
  $\mathcal S_{3}$ &  & 5.8 & 6.4 & 7.6 & 8.6 & 11.8 & 15.4 & 25.2 & 5.0 & 4.2 & 6.2 & 6.2 & 9.6 & 13.6 & 24.0 \\ 
  $\mathcal T_{3}$ & 64 & 5.2 & 5.2 & 6.0 & 4.6 & 6.2 & 9.8 & 16.2 & 5.0 & 4.2 & 4.4 & 3.8 & 4.8 & 8.6 & 14.8 \\ 
  $\mathcal S_{4}$ &  &  &  &  &  &  &  &  & 4.8 & 9.0 & 25.6 & 33.4 & 41.8 & 45.0 & 47.2 \\ 
  $\mathcal T_{4}$ &  &  &  &  &  &  &  &  & 5.4 & 7.2 & 17.8 & 29.8 & 39.2 & 43.2 & 46.4 \\  \hline
  $\mathcal S_{2}$ &  & 8.0 & 4.8 & 6.6 & 5.0 & 6.0 & 5.8 & 5.2 & 6.8 & 4.0 & 5.8 & 4.8 & 5.2 & 5.8 & 4.4 \\ 
  $\mathcal S_{3}$ &  & 7.0 & 5.6 & 4.6 & 9.0 & 7.2 & 11.6 & 16.2 & 6.6 & 4.8 & 4.2 & 7.8 & 6.0 & 10.8 & 15.0 \\ 
  $\mathcal T_{3}$ & 128 & 6.8 & 4.4 & 5.2 & 6.0 & 4.8 & 6.8 & 10.6 & 6.2 & 3.6 & 4.2 & 5.6 & 4.6 & 6.0 & 9.2 \\ 
  $\mathcal S_{4}$ &  &  &  &  &  &  &  &  & 6.6 & 10.4 & 22.0 & 30.0 & 36.6 & 43.2 & 46.8 \\ 
  $\mathcal T_{4}$ &  &  &  &  &  &  &  &  & 5.8 & 6.8 & 13.6 & 25.2 & 33.2 & 40.2 & 44.6 \\ 
  \hline \hline
\end{tabular} \medskip
\caption{Empirical rejections probabilities in \% for the independence model (Model 1).}	\label{tab:m1}
\vspace{-.5cm}
\end{table}
}

\item \textbf{Inductive Model.} 
The following model exhibits pairwise independence but not triplewise independence: let $X_1, X_2$ be i.i.d. standard uniform. We inductively construct, for $k\in\{3, \dots, d\}$,
\[
X_k = \begin{cases}
	X_{k-2} + X_{k-1} &,\quad X_{k-2} + X_{k-1} \le 1 \\
	X_{k-2} + X_{k-1} -1 &,\quad X_{k-2} + X_{k-1} > 1.
\end{cases}
\]
Note that dependence only arises for subvectors $\bm X_A$ for which $A$ contains some set $\{\ell, \ell+1, \ell+2\}$. In particular, out of the $\binom{d}{3}$ triplets, only $d-2$ are not independent, which is a proportion of $O(d^{-2})$. The tests' power should hence be decreasing in $d$.
\item \textbf{Geisser-Mantel Model}. 
The following model is inspired by \cite{GeiMan62} and is defined for values $d$ such that $d=p(p-1)/2$ for some $p\ge3$. Since not every value $d\in\{4,8,16,32,64,128,256\}$ is of that form, we instead consider $d\in\{3,6,10,28,55,120,231\}$ with respective value $p\in\{3, 4, 5, 8, 11, 16, 21\}$. 
Within the simulation results, those values are still identified with $\{4,8,16,32,64,128,256\}$. The model has one parameter $m\in\N$, and is defined by the following algorithm. First, simulate a sample of size $p+m$ from $\Nc_p(0, \bm I_p)$. Then, calculate the empirical correlation matrix of that sample, and store the pairwise correlation coefficients in a vector $\bm X$ of length $d$. The coordinates of $\bm X$ are pairwise independent for any $m$, but exhibit significant $k$-variate dependence for any $k\ge 3$. The smaller $m$, the larger the dependence. Heuristically, the dependence stems from the fact that the correlation matrix is positive-definite.
We implement this model for the choice of $m = \left\lfloor p/2\right\rfloor.$

\item \textbf{Truncated Romano-Siegel Model.}
This model is a truncated version of the Romano-Siegel model in Section 4.2 of \cite{GenRem04}, and is defined as follows. Let $Z_1,Z_2,Z_3$ iid standard normal random variables. We construct: $$X_1 = |Z_1| \cdot \textrm{sign}(Z_2Z_3),\quad X_2 = Z_2, \quad X_3 = Z_3.$$
The resulting triplet $\left( X_1,X_2,X_3\right) $ exhibits pairwise independence but not mutual independence. By generating $Z_4,Z_5,Z_6$ iid standard normal random variables independently of $(Z_1,Z_2,Z_3)$, we duplicate the method to construct $X_4,X_5,X_6$. We repeat this process until we obtain vectors of length $d\in\{3,6,15,30,63,126,255\}$, which correspond to the largest dimension smaller or equal to $d\in\{4,8,16,32,64,128,256\}$ that is divisible by 3. Similar as for the inductive model, out of the $\binom{d}{3}$ triplets, only $d/3$ are not independent, which is a proportion of $O(d^{-2})$. The tests' power should hence be decreasing in $d$.
\end{compactenum}

{\footnotesize
\begin{table}[thb!]
\centering
\begin{adjustwidth}{-.2cm}{}
\begin{tabular}{ c|c|| rrrrrrr || rrrrrrr } 
\hline \hline
\multicolumn{2}{l}{} & \multicolumn{7}{c}{\textit{Finite variance scaling $\bar \delta_n$}} & \multicolumn{7}{c}{\textit{Asymptotic variance scaling $\delta_n$}} \\ \hline \hline
Test & $n \setminus  d$ &\multicolumn{1}{c}{$4$} & \multicolumn{1}{c}{$8$} & \multicolumn{1}{c}{$16$} & \multicolumn{1}{c}{$32$}& \multicolumn{1}{c}{$64$}& \multicolumn{1}{c}{$128$}& \multicolumn{1}{c||}{$256$} 
&\multicolumn{1}{c}{$4$} & \multicolumn{1}{c}{$8$} & \multicolumn{1}{c}{$16$} & \multicolumn{1}{c}{$32$}& \multicolumn{1}{c}{$64$}& \multicolumn{1}{c}{$128$}& \multicolumn{1}{c}{$256$} \\
\hline
 \addlinespace[.2cm]
  \multicolumn{9}{l}{\quad \textit{Model 2-1: Gaussian $\| \bm \tau\|_2^2=0.1$}} \\ \hline
  $\mathcal S_{2}$ &  & 24.0 & 12.8 & 9.6 & 7.0 & 4.4 & 7.0 & 4.6 & 18.0 & 8.0 & 5.4 & 4.0 & 2.6 & 4.0 & 1.4 \\ 
  $\mathcal S_{3}$ & 16 & 8.4 & 11.0 & 13.0 & 24.4 & 30.8 & 32.6 & 38.6 & 4.0 & 5.0 & 6.8 & 16.8 & 20.2 & 25.8 & 31.8 \\ 
  $\mathcal T_{3}$ &  & 17.2 & 9.0 & 8.6 & 16.4 & 17.0 & 24.4 & 31.6 & 11.2 & 4.0 & 3.6 & 8.2 & 12.4 & 17.8 & 24.6 \\  \hline
  $\mathcal S_{2}$ &  & 41.6 & 22.8 & 10.6 & 7.4 & 5.2 & 6.6 & 3.6 & 36.4 & 18.0 & 7.2 & 6.4 & 3.8 & 5.0 & 2.8 \\ 
  $\mathcal S_{3}$ & 32 & 6.8 & 4.4 & 9.8 & 12.0 & 19.0 & 24.2 & 30.2 & 3.8 & 2.4 & 6.0 & 8.0 & 14.2 & 20.8 & 26.2 \\ 
  $\mathcal T_{3}$ &  & 28.4 & 11.4 & 9.2 & 7.0 & 13.0 & 15.4 & 22.2 & 24.6 & 9.6 & 6.2 & 5.0 & 9.0 & 12.8 & 19.0 \\  \hline
  $\mathcal S_{2}$ &  & 71.0 & 40.4 & 22.2 & 10.6 & 9.2 & 6.8 & 5.0 & 69.0 & 37.8 & 19.4 & 9.4 & 8.0 & 4.8 & 4.4 \\ 
  $\mathcal S_{3}$ & 64 & 6.8 & 7.6 & 6.2 & 10.4 & 11.6 & 15.6 & 23.8 & 5.8 & 6.0 & 4.2 & 8.6 & 9.8 & 13.8 & 22.2 \\ 
  $\mathcal T_{3}$ &  & 57.4 & 28.8 & 11.6 & 10.0 & 10.8 & 10.8 & 17.4 & 56.6 & 26.8 & 9.6 & 8.0 & 8.8 & 9.8 & 16.0 \\  \hline
  $\mathcal S_{2}$ &  & 95.4 & 74.0 & 42.8 & 19.8 & 10.6 & 9.4 & 6.0 & 95.2 & 73.0 & 41.8 & 18.6 & 9.4 & 8.8 & 6.0 \\ 
  $\mathcal S_{3}$ & 128 & 9.2 & 6.8 & 8.2 & 7.2 & 9.8 & 11.8 & 20.2 & 8.8 & 6.0 & 7.6 & 6.6 & 9.2 & 11.2 & 19.0 \\ 
  $\mathcal T_{3}$ &  & 90.2 & 56.0 & 32.4 & 14.4 & 9.2 & 10.4 & 14.0 & 90.0 & 55.4 & 30.6 & 13.8 & 8.8 & 9.4 & 13.4 \\  \hline
\addlinespace[.2cm]
  \multicolumn{9}{l}{\quad \textit{Model 2-1: Gaussian $\| \bm \tau\|_2^2=0.3$}} \\ \hline
  $\mathcal S_{2}$ &  & 53.0 & 25.6 & 17.0 & 9.8 & 6.0 & 7.6 & 5.0 & 44.8 & 19.2 & 9.6 & 5.6 & 3.2 & 3.6 & 2.6 \\ 
  $\mathcal S_{3}$ & 16 & 8.8 & 11.6 & 14.6 & 19.8 & 26.2 & 34.8 & 37.0 & 3.8 & 4.8 & 6.8 & 12.0 & 17.4 & 29.2 & 32.4 \\ 
  $\mathcal T_{3}$ &  & 37.8 & 17.2 & 9.6 & 11.8 & 17.0 & 26.6 & 31.4 & 30.8 & 10.6 & 4.6 & 6.4 & 9.0 & 18.6 & 26.2 \\  \hline
  $\mathcal S_{2}$ &  & 81.8 & 52.4 & 29.0 & 17.0 & 9.4 & 6.4 & 6.4 & 79.4 & 47.0 & 24.6 & 13.8 & 6.4 & 4.2 & 4.8 \\ 
  $\mathcal S_{3}$ & 32 & 7.6 & 6.6 & 9.8 & 14.0 & 20.6 & 25.0 & 34.2 & 5.0 & 3.6 & 5.8 & 10.2 & 14.2 & 21.4 & 31.0 \\ 
  $\mathcal T_{3}$ &  & 74.6 & 37.0 & 20.2 & 13.6 & 12.8 & 18.2 & 27.0 & 72.0 & 31.4 & 15.4 & 9.8 & 8.0 & 15.4 & 22.6 \\  \hline
  $\mathcal S_{2}$ &  & 98.2 & 84.4 & 61.6 & 30.0 & 15.4 & 10.0 & 6.0 & 98.2 & 83.8 & 58.8 & 27.0 & 13.6 & 9.0 & 5.6 \\ 
  $\mathcal S_{3}$ & 64 & 14.4 & 9.6 & 6.0 & 8.0 & 14.0 & 17.2 & 25.0 & 12.0 & 7.8 & 5.8 & 6.2 & 11.2 & 15.6 & 23.8 \\ 
  $\mathcal T_{3}$ &  & 96.2 & 71.6 & 40.2 & 19.8 & 15.0 & 13.6 & 19.0 & 95.8 & 70.4 & 37.6 & 16.8 & 12.0 & 11.2 & 17.2 \\  \hline
  $\mathcal S_{2}$ &  & 100.0 & 99.6 & 92.2 & 65.6 & 32.4 & 16.2 & 9.6 & 100.0 & 99.6 & 91.8 & 64.8 & 31.6 & 15.8 & 9.0 \\ 
  $\mathcal S_{3}$ & 128 & 24.0 & 7.2 & 5.2 & 6.8 & 8.2 & 12.8 & 17.0 & 22.8 & 7.2 & 5.0 & 6.2 & 7.2 & 11.4 & 16.2 \\ 
  $\mathcal T_{3}$ &  & 100.0 & 97.8 & 79.6 & 44.4 & 19.8 & 14.0 & 14.8 & 100.0 & 97.8 & 79.2 & 42.8 & 18.0 & 12.8 & 14.0 \\  \hline
\addlinespace[.2cm]
  \multicolumn{9}{l}{\quad \textit{Model 2-1: Gaussian $\| \bm \tau\|_2^2=0.7$}} \\ \hline
  $\mathcal S_{2}$ &  & 84.6 & 56.6 & 29.2 & 15.6 & 11.2 & 8.0 & 8.0 & 80.4 & 47.6 & 22.0 & 11.2 & 6.4 & 3.6 & 3.0 \\ 
  $\mathcal S_{3}$ & 16 & 9.4 & 10.8 & 15.6 & 21.8 & 23.0 & 32.2 & 35.8 & 4.6 & 5.8 & 8.6 & 16.2 & 16.2 & 27.0 & 33.0 \\ 
  $\mathcal T_{3}$ &  & 74.4 & 40.2 & 20.2 & 19.2 & 16.8 & 25.2 & 32.8 & 68.4 & 29.8 & 12.2 & 10.8 & 10.8 & 18.8 & 28.0 \\  \hline
  $\mathcal S_{2}$ &  & 99.2 & 89.0 & 62.8 & 31.0 & 17.2 & 11.2 & 7.6 & 99.0 & 87.0 & 57.8 & 26.4 & 14.4 & 9.2 & 5.6 \\ 
  $\mathcal S_{3}$ & 32 & 18.8 & 10.2 & 8.2 & 8.8 & 18.2 & 22.0 & 31.0 & 14.4 & 7.8 & 5.6 & 5.8 & 14.4 & 18.0 & 27.6 \\ 
  $\mathcal T_{3}$ &  & 98.8 & 79.6 & 44.0 & 17.8 & 17.6 & 16.4 & 25.4 & 98.4 & 76.6 & 39.6 & 14.4 & 11.4 & 12.4 & 23.2 \\  \hline
  $\mathcal S_{2}$ &  & 100.0 & 99.6 & 92.8 & 67.0 & 33.6 & 18.4 & 10.0 & 100.0 & 99.6 & 92.0 & 64.2 & 31.0 & 16.4 & 7.8 \\ 
  $\mathcal S_{3}$ & 64 & 46.6 & 11.8 & 7.2 & 9.0 & 11.0 & 16.0 & 22.6 & 42.8 & 10.8 & 6.0 & 7.8 & 9.4 & 14.2 & 20.4 \\ 
  $\mathcal T_{3}$ &  & 100.0 & 98.8 & 80.8 & 44.4 & 22.0 & 15.2 & 18.4 & 100.0 & 98.6 & 80.4 & 41.6 & 18.6 & 13.2 & 16.0 \\  \hline
  $\mathcal S_{2}$ &  & 100.0 & 100.0 & 99.8 & 97.2 & 72.8 & 36.6 & 17.6 & 100.0 & 100.0 & 99.8 & 97.2 & 71.6 & 35.4 & 17.2 \\ 
  $\mathcal S_{3}$ & 128 & 86.8 & 24.2 & 9.2 & 5.8 & 5.0 & 11.4 & 16.6 & 86.2 & 23.0 & 8.2 & 5.0 & 4.4 & 10.6 & 15.6 \\ 
  $\mathcal T_{3}$ &  & 100.0 & 100.0 & 100.0 & 88.4 & 50.2 & 25.4 & 17.4 & 100.0 & 100.0 & 100.0 & 88.4 & 49.6 & 23.8 & 17.0 \\ 
  \hline \hline
\end{tabular} \medskip
\end{adjustwidth}
\caption{Empirical rejections probabilities in \% for the Gaussian Model (Model 2). }	\label{tab:m2}
\vspace{-.8cm}
\end{table}
}

{\footnotesize
\begin{table}[t!]
\centering
\begin{adjustwidth}{-.9cm}{}
\begin{tabular}{ c|c|| rrrrrrr || rrrrrrr } 
\hline \hline
\multicolumn{2}{l}{} & \multicolumn{7}{c}{\textit{Finite variance scaling $\bar \delta_n$}} & \multicolumn{7}{c}{\textit{Asymptotic variance scaling $\delta_n$}} \\ \hline \hline
Test & $n \setminus  d$ &\multicolumn{1}{c}{$4$} & \multicolumn{1}{c}{$8$} & \multicolumn{1}{c}{$16$} & \multicolumn{1}{c}{$32$}& \multicolumn{1}{c}{$64$}& \multicolumn{1}{c}{$128$}& \multicolumn{1}{c||}{$256$} 
&\multicolumn{1}{c}{$4$} & \multicolumn{1}{c}{$8$} & \multicolumn{1}{c}{$16$} & \multicolumn{1}{c}{$32$}& \multicolumn{1}{c}{$64$}& \multicolumn{1}{c}{$128$}& \multicolumn{1}{c}{$256$} \\
\hline
 \addlinespace[.2cm]
  \multicolumn{9}{l}{\quad \textit{Model 3: Inductive Model}} \\ \hline
  $\mathcal S_{2}$ &  & 8.2 & 11.6 & 9.4 & 8.0 & 6.0 & 5.2 & 5.6 & 5.2 & 6.8 & 6.6 & 3.6 & 3.2 & 3.2 & 2.4 \\ 
  $\mathcal S_{3}$ & 16 & 48.0 & 46.6 & 40.8 & 38.4 & 37.8 & 35.0 & 41.0 & 29.4 & 30.6 & 27.0 & 26.6 & 29.2 & 28.8 & 35.0 \\ 
  $\mathcal T_{3}$ &  & 31.4 & 36.8 & 27.6 & 27.2 & 26.4 & 28.2 & 34.6 & 17.6 & 21.8 & 16.8 & 16.8 & 15.6 & 19.8 & 27.8 \\  \hline
  $\mathcal S_{2}$ &  & 7.2 & 10.6 & 8.6 & 7.0 & 4.8 & 6.6 & 4.8 & 5.8 & 9.4 & 7.2 & 5.6 & 3.4 & 5.0 & 3.6 \\ 
  $\mathcal S_{3}$ & 32 & 100.0 & 95.6 & 86.8 & 69.8 & 55.4 & 43.2 & 41.6 & 98.4 & 93.6 & 80.0 & 63.6 & 48.2 & 38.8 & 38.0 \\ 
  $\mathcal T_{3}$ &  & 90.6 & 80.2 & 63.4 & 51.8 & 42.0 & 34.2 & 32.0 & 81.6 & 73.2 & 56.6 & 45.0 & 33.4 & 28.8 & 27.8 \\  \hline
  $\mathcal S_{2}$ &  & 8.8 & 11.4 & 10.6 & 7.2 & 5.2 & 6.0 & 7.6 & 7.8 & 10.4 & 9.2 & 5.8 & 4.0 & 5.4 & 6.2 \\ 
  $\mathcal S_{3}$ & 64 & 100.0 & 100.0 & 100.0 & 100.0 & 94.8 & 75.8 & 55.4 & 100.0 & 100.0 & 100.0 & 99.8 & 92.6 & 72.6 & 52.6 \\ 
  $\mathcal T_{3}$ &  & 100.0 & 100.0 & 99.2 & 96.8 & 82.0 & 59.4 & 45.2 & 100.0 & 100.0 & 98.6 & 95.0 & 79.0 & 54.8 & 43.0 \\  \hline
  $\mathcal S_{2}$ &  & 12.4 & 10.0 & 12.0 & 7.8 & 4.6 & 5.0 & 6.8 & 12.4 & 9.6 & 11.0 & 7.4 & 4.2 & 4.8 & 5.6 \\ 
  $\mathcal S_{3}$ & 128 & 100.0 & 100.0 & 100.0 & 100.0 & 100.0 & 100.0 & 96.0 & 100.0 & 100.0 & 100.0 & 100.0 & 100.0 & 100.0 & 95.8 \\ 
  $\mathcal T_{3}$ &  & 100.0 & 100.0 & 100.0 & 100.0 & 100.0 & 99.0 & 89.0 & 100.0 & 100.0 & 100.0 & 100.0 & 100.0 & 98.6 & 87.0 \\  \hline
\addlinespace[.2cm]
  \multicolumn{9}{l}{\quad \textit{Model 4: Geisser-Mantel Model}} \\ \hline
  $\mathcal S_{2}$ &  & 3.6 & 2.6 & 6.6 & 4.6 & 6.4 & 4.4 & 6.6 & 1.4 & 0.4 & 3.0 & 1.8 & 3.4 & 1.8 & 3.8 \\ 
  $\mathcal S_{3}$ & 16 & 21.0 & 16.2 & 59.2 & 67.8 & 78.8 & 78.6 & 26.4 & 8.8 & 3.8 & 45.6 & 59.0 & 75.2 & 75.6 & 10.4 \\ 
  $\mathcal T_{3}$ &  & 12.8 & 9.4 & 48.8 & 62.4 & 74.6 & 75.0 & 15.6 & 4.2 & 2.8 & 30.6 & 47.0 & 67.2 & 69.4 & 3.8 \\  \hline
  $\mathcal S_{2}$ &  & 3.6 & 1.8 & 6.4 & 5.0 & 9.4 & 6.4 & 6.4 & 2.8 & 1.0 & 4.8 & 3.4 & 6.2 & 3.8 & 5.2 \\ 
  $\mathcal S_{3}$ & 32 & 56.6 & 40.4 & 79.6 & 84.4 & 90.0 & 90.6 & 66.2 & 42.4 & 24.6 & 72.0 & 81.2 & 88.4 & 89.8 & 56.6 \\ 
  $\mathcal T_{3}$ &  & 37.0 & 23.0 & 69.6 & 79.4 & 86.2 & 89.2 & 51.0 & 25.6 & 12.2 & 59.2 & 71.8 & 83.4 & 87.6 & 40.4 \\  \hline
  $\mathcal S_{2}$ &  & 5.0 & 2.0 & 6.6 & 4.0 & 6.6 & 6.6 & 5.0 & 4.4 & 1.8 & 5.6 & 3.6 & 5.0 & 5.8 & 4.2 \\ 
  $\mathcal S_{3}$ & 64 & 93.0 & 89.4 & 97.4 & 97.8 & 98.6 & 99.2 & 93.2 & 90.6 & 83.2 & 96.2 & 97.4 & 98.4 & 99.2 & 92.6 \\ 
  $\mathcal T_{3}$ &  & 78.2 & 63.4 & 89.8 & 93.8 & 96.2 & 98.6 & 89.0 & 73.0 & 56.8 & 86.8 & 91.8 & 95.8 & 97.8 & 86.6 \\  \hline
  $\mathcal S_{2}$ &  & 7.0 & 1.0 & 6.6 & 5.0 & 6.6 & 6.4 & 8.6 & 7.0 & 0.8 & 5.6 & 4.4 & 6.2 & 6.2 & 8.2 \\ 
  $\mathcal S_{3}$ & 128 & 100.0 & 100.0 & 100.0 & 100.0 & 100.0 & 100.0 & 99.8 & 100.0 & 100.0 & 100.0 & 100.0 & 100.0 & 100.0 & 99.8 \\ 
  $\mathcal T_{3}$ &  & 99.6 & 99.4 & 100.0 & 99.8 & 100.0 & 100.0 & 99.8 & 99.4 & 99.2 & 100.0 & 99.6 & 100.0 & 99.8 & 99.8 \\  \hline
\addlinespace[.2cm]
  \multicolumn{9}{l}{\quad \textit{Model 5: Truncated Romano Siegel Model}} \\ \hline
  $\mathcal S_{2}$ &  & 3.6 & 6.8 & 5.2 & 4.6 & 4.0 & 7.0 & 2.8 & 1.6 & 3.2 & 2.6 & 2.4 & 2.4 & 3.4 & 1.2 \\ 
  $\mathcal S_{3}$ & 16 & 16.6 & 23.4 & 26.4 & 29.4 & 32.4 & 33.8 & 53.4 & 5.6 & 11.0 & 15.4 & 22.8 & 25.8 & 28.0 & 25.4 \\ 
  $\mathcal T_{3}$ &  & 12.6 & 13.2 & 15.0 & 20.0 & 24.2 & 27.6 & 22.4 & 3.6 & 6.6 & 8.6 & 12.8 & 16.4 & 21.6 & 4.4 \\  \hline
  $\mathcal S_{2}$ &  & 2.2 & 4.8 & 5.0 & 3.6 & 4.2 & 5.4 & 2.4 & 2.0 & 3.6 & 2.8 & 2.4 & 2.8 & 4.2 & 1.4 \\ 
  $\mathcal S_{3}$ & 32 & 82.4 & 57.6 & 39.8 & 34.2 & 36.2 & 33.6 & 98.8 & 66.8 & 44.0 & 30.2 & 28.0 & 30.6 & 30.2 & 98.0 \\ 
  $\mathcal T_{3}$ &  & 45.8 & 32.6 & 22.2 & 20.4 & 24.2 & 24.4 & 98.2 & 30.2 & 22.6 & 15.4 & 15.0 & 17.4 & 21.4 & 95.0 \\  \hline
  $\mathcal S_{2}$ &  & 2.2 & 5.2 & 3.6 & 3.8 & 4.6 & 6.8 & 4.2 & 1.6 & 4.2 & 2.4 & 3.4 & 3.8 & 6.0 & 3.8 \\ 
  $\mathcal S_{3}$ & 64 & 100.0 & 100.0 & 84.8 & 63.8 & 46.2 & 37.2 & 100.0 & 100.0 & 100.0 & 80.8 & 58.4 & 43.0 & 34.6 & 100.0 \\ 
  $\mathcal T_{3}$ &  & 100.0 & 90.4 & 60.6 & 41.6 & 30.8 & 26.6 & 100.0 & 100.0 & 86.8 & 55.0 & 38.0 & 26.4 & 24.2 & 100.0 \\  \hline
  $\mathcal S_{2}$ &  & 0.8 & 3.6 & 4.2 & 4.0 & 5.0 & 4.8 & 3.4 & 0.6 & 2.8 & 4.0 & 3.6 & 4.4 & 4.2 & 3.0 \\ 
  $\mathcal S_{3}$ & 128 & 100.0 & 100.0 & 100.0 & 99.2 & 85.6 & 62.4 & 100.0 & 100.0 & 100.0 & 100.0 & 98.8 & 84.8 & 61.0 & 100.0 \\ 
  $\mathcal T_{3}$ &  & 100.0 & 100.0 & 99.2 & 91.8 & 66.8 & 46.4 & 100.0 & 100.0 & 100.0 & 98.8 & 90.4 & 64.0 & 43.2 & 100.0 \\ 
  \hline \hline
\end{tabular} \medskip
\end{adjustwidth}
\caption{Empirical rejections probabilities in \% for models involving pairwise independence but not mutual independence (Model 3 - 5). }	\label{tab:m345}
\vspace{-.5cm}
\end{table}
}

We study the performance of the following eight tests, with intended level $\alpha$ fixed to $\alpha=5\%$: 
first, for $k\in\{2,3,4\}$ and recalling the asymptotic variance $\delta_n(k)$ from \eqref{eq:deltan}, we consider the test which rejects $H$ if $\delta_n^{-1}(k)\left(  T_n(k)-\nu_n(k)\right)> 1.645 = u_{0.95}$, which we denote by $\mathcal S_k(\delta)$. Likewise, we consider the analogue based on the explicit finite-sample variance $\bar \delta_n(k)$, which we denote by $\mathcal S_{k}(\bar \delta)$ and which we only apply for $k\in\{2,3\}$ since we were not able to compute $\bar \delta_n(4)$ explicitly. Moreover, we also consider the tests from Corollary~\ref{cor:test} with $m\in\{3,4\}$, which we denote by $\mathcal T_m(\delta)$ or $\mathcal T_m(\bar\delta)$ (the latter only for $m=3$) depending on whether $\delta_n(k)$ or $\bar \delta_n(k)$  has been used for scaling the underlying test statistic.

For each of the above models and combinations of $n$ and $d$ (with occasional modifications for $d$ as explained in the model descriptions above), we simulate 500 random i.i.d.\ samples. The performance of each test is assessed by  calculating empirical rejection probabilities, which are presented in Tables~\ref{tab:m1}-\ref{tab:m345}.

Table~\ref{tab:m1} contains the results for mutual independence, i.e., Model 1. It can be seen that the tests $\mathcal S_2(\delta)$ and $\mathcal S_2(\bar \delta)$ accurately hold their intended level, with small level exceedances up to 7.4\% only for small values of $n$. The tests $\mathcal S_3(\delta)$ and $\mathcal S_3(\bar \delta)$, however, do not hold their level for large values of $d$ (relative to $n$); for instance, for $n=128$, the level approximation is acceptable only for $d\le 64$. This observation is consistent with the theoretical results: asymptotic level approximation is only guaranteed if $d=o(n^{\frac{1}{2}})$. Next, expectable, the tests $\mathcal T_3$ are somewhat in between $\mathcal S_2$ and $\mathcal S_3$ and provide acceptable level approximations for slightly larger values of $d$ compared to $\mathcal S_3$ alone. Comparing the $\delta$ and $\bar \delta$ versions, the former are slightly more conservative, with the difference between the two vanishing for increasing sample size $n$. Regarding the tests $\mathcal S_4$ and $\mathcal T_4$, which theoretically require $d=o(n^{\frac{1}{3}})$, we observe a reasonable level approximation only for $d=4$, which appears reasonable since $128^{\frac{1}{3}}\approx 5$. For this reason, we exclude these two tests  from the power study in Tables~\ref{tab:m2}-\ref{tab:m345}.

Table~\ref{tab:m2} contains results for the Gaussian models. For the $\mathcal S_2$ tests, we observe that the power is decreasing in $d$ and increasing $n$. This behavior may be explained by the fact that the overall pairwise dependence signal as measured by $\| \bm \tau \|_2^2$ is constant in $d$ and $n$, whence a decrease in $d$ and an increase in $n$ essentially increases the signal-to-noise ratio. Overall, the results are in fact akin to the results in Table~3 in  \cite{LeuDrt18}, which concern the exact same models for a number of related test statistics (some of which are inconsistent in general, for instance against alternatives which are pairwise dependent, but exhibit a non-zero population coefficient of Kendall's tau). Moving to higher order dependencies, we observe rather little power for small values of $n$ and $d$. For sufficiently large $n$, the power as a function of $d$ is u-shaped, which might be a result of a superposition of the signal-to-noise ratio decrease and the increase visible under the null hypothesis in Table~\ref{tab:m1} (as a result of the fact that the condition $d=o(n^{\frac{1}{2}})$ is not met). Finally, the results for tests $\mathcal T_3$ are again an average  of the results for $\mathcal S_2$ and $\mathcal S_3$, as expected.

Results for Models 3-5 concerning pairwise independence but higher order dependence are presented in Table~\ref{tab:m345}. Concerning the $\mathcal S_2$ tests, we observe no power at all, which is exactly the expected behavior. Note that one would obtain very similar results when applying the tests from \cite{LeuDrt18}, which are also designed to detect pairwise dependencies only. Concerning $\mathcal S_3$, we observe decent power behavior in all models under consideration. For the inductive model and the truncated Romano Siegel model, we observe a decreasing in power in $d$ (for all sufficiently large $n$), which was expectable in view of the discussion above: in both models, the proportion of dependent triples is of the order $O(d^{-2})$.

\section{Remaining Proofs and Auxiliary Results} \label{sec:proofs}

\subsection{Preliminaries and Notations}\label{sect:notation}
Throughout the proofs, sets $\ds{A\subset \{1, \dots, d\}:=S_d}$ of cardinality $k \in \{2, \dots, d\}$ will occasionally be identified with ordered integer vectors $\ds{\mathbf p_k \in \mathcal{P}(d,k)}$, where
\begin{align} \label{eq:pdk}
\mathcal{P}(d,k)=\left\{\mathbf p_k=(p_1, \dots, p_k) \in \{1, \dots, d\}^k:  p_1<\cdots<p_k\right\}.
\end{align}
Note that, for an arbitrary function $f:\{1, \dots, d\}^k \to \R$, we have the equivalent summation notations
\begin{align*} \sum\limits_{\mathbf{p}_k\in \mathcal{P}(d,k) }f(\mathbf{p}_k)& =\sum\limits_{1\le p_1<\cdots<p_k\le d}f(\mathbf{p}_k)= \sum_{p_{k}=1}^{d}\sum_{p_{k-1}=1}^{p_k-1}\cdots \sum_{p_1=1}^{p_2-1}f(\mathbf{p}_k). 
\end{align*}

For a vector $\bm x\in \R^d$ and a set $A\subset \{1, \dots, d\}$, we write $\bm x_A=(x_j)_{j \in A}$. If $A=\{\ell, \ell+1, \dots, \ell'\}$, we also write $\bm x_{\ell:\ell'}=x_{\{\ell, \ell+1, \dots, \ell'\}}$. Further, for an integer vector $\mathbf i$, we write  $|\mathbf i|$ for the number of distinct coordinates of $\mathbf i$.  For $m\in\N$ and an integer vector $\mathbf i =(i_1, \dots, i_{2m})$ of dimension $2m$, we define
\begin{align} \label{eq:plainphi}
\varphi_m(\mathbf i) = \E\left[ \tilde I^{(1)}_{i_{1},i_{2}} \tilde I^{(1)}_{i_{3},i_{4}} \cdots \tilde I^{(1)}_{i_{2m-1},i_{2m}}\right].
\end{align}

Throughout the proofs, we will repeatedly need integer vectors $\mathbf i$ of dimension 4. Of particular importance are the integer vectors taken from the following set
\begin{align} \label{eq:jn}
\mathcal J=\mathcal J_n
=
\left\{ \mathbf{i}= (i_1, i_2, i_3, i_4) \in \{1, \dots, n\}^4: i_1 < i_2, i_3 < i_4\right\}.
\end{align}
Note that $\mathcal J = \cup_{\ell=2}^4\mathcal{I}_\ell$, where
\begin{align} \label{eq:iell}
\mathcal{I}_\ell := \mathcal{I}_{\ell,n}  
:=
\left\{ \mathbf{i}=(i_1, i_2, i_3, i_4) \in \mathcal J: |\mathbf i|=\ell\right\},\qquad \ell\in\{2,3,4\},
\end{align}
and that
 \begin{align}\label{cardinal}
 	\zp[b]{\mathcal{J}} = \zp[b]{\mathcal{I}_4}=\frac{n^4}{4}\left( 1+O\left( n^{-1}\right) \right) ,\quad \zp[b]{\mathcal{I}_3} =O(n^3),\quad 
 	\zp[b]{\mathcal{I}_2} = \frac{n(n-1)}{2} .
 \end{align}

For real constants $a_1, \dots, a_q, b_1, \dots, b_q$ recall the multinomial identity
\begin{align}\label{multinomial}
	\prod_{i=1}^q (a_i+b_i) = \sum_{A \subset S_q}\prod_{j \in A} a_j \prod_{j \notin A}b_j,
\end{align} 
where the empty product is defined as $1$.
For two non-negative real sequences $(a_n)_n$ and $(b_n)_n$ we write $a_n\propto b_n$ if $a_n=O(b_n)$ and $b_n=O(a_n)$.

\subsection{Proof of Lemma~\ref{lem:snvar}}
\label{subsec:moml1}

We set $b_{i,j} :=
\E[I_{i,j}^{(p)}]$. We start by proving \eqref{eq:expi} for  fixed $p\in\{1, \dots, d\}$,  which follows from
	\begin{align*}
		\E\left[R_{ip}\right]
		&= 
		\sum_{k=1}^n k \p\left[R_{ip}=k\right]
		=
		\frac1n \sum_{k=1}^n k = \frac{n+1}2, \\
		\E\left[R_{ip}\left(R_{ip}-1 \right) \right]
		& = 
		\sum_{k=1}^n k(k-1) \p\left[R_{ip}=k\right]
		=
		\frac1n \sum_{k=1}^n k(k-1) 
		=
		\frac{1}3 (n-1) (n+1), \\
		\E\left[\max(R_{ip}, R_{jp}) \right]
		&=
		\sum_{k, \ell=1}^n \max(k,\ell) \p \left[R_{ip}=k, R_{jp}=\ell\right]  \\
		&= 
		\frac1{n(n-1)} \sum_{k, \ell=1, k \ne \ell }^n  \max(k,\ell) 
		= \frac{2(n+1)}{3} \qquad(i \ne j),
	\end{align*}
and some simple calculations. Next, the assertion in \eqref{eq:munk} readily follow from \eqref{eq:expi} and
	\[
	\E_{H_k} [S_{n,A}^{\rm M} ] = \E[I_{i,j}^{(p)}]^k + (n-1) \E[I_{i,i}^{(p)}]^{k}.
	\]
	Finally, regarding \eqref{eq:varn} we have
$
\Var(S_{n,A}^{\rm M})
=
\frac{1}{n^2} \sum_{\mathbf i \in \mathcal{J}} \Cov(I_{i_1,i_2}^{(p)}, I_{i_3,i_4}^{(p)}),
$
where $\mathcal J$ has been defined in \eqref{eq:jn} and where
\begin{align*}
c_{\mathbf{i}} := \Cov(I_{i_1,i_2}^{(p)}, I_{i_3,i_4}^{(p)}) 
&=
\Cov\Big( 
\prod_{p \in A} I_{i_1,i_2}^{(p)}, \prod_{p \in A} I_{i_3,i_4}^{(p)}\Big) \\
&=
\prod_{p\in A} \E[ I_{i_1,i_2}^{(p)}  I_{i_3,i_4}^{(p)} ] - \prod_{p \in A}  \E[ I_{i_1,i_2}^{(p)}]  \E[ I_{i_3,i_4}^{(p)}] \\
&=
\E[ I_{i_1,i_2}^{(p)}  I_{i_3,i_4}^{(p)} ]^{k} -  \E[ I_{i_1,i_2}^{(p)}]^k  \E[ I_{i_3,i_4}^{(p)}]^k 
=: \tilde c_{\mathbf i}^k - b_{(i_1, i_2)}^k b_{(i_3,i_4)}^k.
\end{align*}
The latter expression is independent of $p$, but does depend on the number and position of equal indizes among $\ds{\mathbf{i}=(i_1, \ldots, i_4)}$. In general, one of the following 4 cases must occur:
\begin{compactitem}
\item $|\mathbf i|=4$, this happens for $n(n-1)(n-2)(n-3)$ quadruples.
\item $|\mathbf i|=3$, this is the case for $6n(n-1)(n-2)$ quadruples.
\item $|\mathbf i|=2$, this is the case for $4n(n-1)$ quadruples for which three indizes are equal and different from the remaining one, and for $3n(n-1)$ quadruples for which there are two pairs of equal indizes. 
\item $|\mathbf i|=1$, this is the case for $n$ quadruples.
\end{compactitem}
As a consequence,
\begin{align*}
\Var(S_{n,A}^{\rm M})
&=  \frac1{n^2} \Big\{ n(n-1)(n-2)(n-3) c_{(1,2,3,4)}   \\
& \hspace{2cm} + n(n-1)(n-2) \{  2 c_{(1,1,2,3)}  + 4 c_{(1,2,2,3)} \} \\
& \hspace{2cm} + n(n-1) \{ 4 c_{(1,1,1,2)} + c_{(1,1,2,2)} + 2 c_{(1,2,1,2)}\} + 
n c_{(1,1,1,1)} \Big\}.
\end{align*}
The assertion then follows from Equations \eqref{eq:ccc1}--\eqref{eq:ccc2} in Appendix~\ref{subsec:ausmom} after some tedious calculations, which have been checked with a computer algebra system.
\qed

\subsection{Results on higher order moments}
\label{subsec:momh}

Lemma~\ref{lem:snvar}  implies that $\tilde I_{i,j}^{(p)}$ defined in \eqref{eq:bii} is centered, which may be written as
$
\varphi_1\left( i,j\right)= 0$.
The following three lemmas provide explicit formulas for second order moments and bounds on third and fourth order moments, respectively. Recall that
$
\mathcal J=\mathcal J_n=\{ \mathbf{i}\in \{1, \dots, n\}^4: i_1 < i_2, i_3 < i_4\},
$
and that $\ds{|\mathbf i|}$ denotes the number of distinct coordinates of $\mathbf i$. 
Observe that for any $\mathbf{i} \in \mathcal{J}$, we have $2\le |\mathbf{i} |\le 4$.

\begin{lemma}[Second order moments] 
\label{lem:2m}
For any $p\in\{1, \dots, d\}$ and $\ds{\mathbf{i} =(i_1,\ldots,i_4)\in \mathcal J}$, we have
\[
\varphi_2(\mathbf{i}) = \Cov(\tilde I_{i_1, i_2}^{(p)}, \tilde I_{i_3, i_4}^{(p)})
=
\begin{cases}
\frac{1}{45n^2} & , \text{ if } |\mathbf{i}|=4, \\
- \frac{2n^2-3n-4}{180n^2(n+1)}=-\frac{1}{90n}+O(n^{-2}) & , \text{ if } |\mathbf{i}|=3, \\
\frac{(n-2)(n^2-n-1)}{90n^2(n+1)}
=\ff{90}+O(n^{-1}) & , \text{ if } |\mathbf{i}|=2.
\end{cases}
\]
As a consequence, for $\mu \in \{2,3,4\}$,
\begin{align}\label{eq:vfi2}
	\sup_{\mathbf i \in \mathcal J: |\mathbf i|=\mu} |\varphi_2\left(\mathbf{i}\right)| &= O\left( n^{2-\mu}\right).
\end{align}
Moreover, for $i,j \in \{1, \dots, n\}$,
\[
\varphi_2((i,i,j,j)) = \Cov(\tilde I_{i, i}^{(p)}, \tilde I_{j,j}^{(p)})
=
\begin{cases}
\frac{(n-2)(n-1)(n+2)}{180n^2(n+1)} = \frac{1}{180} + O(n^{-1}) & , \text{ if } i=j, \\
 -\frac{(n-2)(n+2)}{180n^2(n+1)}  = -\frac{1}{180n}  + O(n^{-2})& , \text{ if } i \ne j.
\end{cases}
\]
\end{lemma}

For proving a Lyapunov condition, we will require a bound on fourth order moments:
\begin{lemma}[Fourth order moments]
\label{lem:4m}
For any $\ds{\mathbf{i}=(i_1,\ldots,i_8)\in \mathcal{J}^2}$ such that $\ds{5\le \zp[b]{\mathbf{i}}\le 8}$, 
\begin{align}
\label{moment_goal}
\varphi_4\left( \mathbf{i} \right)&= \E\left[\tilde I^{(1)}_{i_1,i_2}\tilde I^{(1)}_{i_3,i_4}\tilde I^{(1)}_{i_5,i_6}\tilde I^{(1)}_{i_7,i_8}\right] = O\left( n^{4-\zp[b]{\mathbf{i}} }\right)
\end{align}
uniformly in $\mathbf i$ (in the sense of \eqref{eq:vfi2}).
\end{lemma}

Finally, a particular third order moment is needed as well:
\begin{lemma}[Third order moment]\label{lem:3m} 
For $\ds{\mathbf{i}_0 =\left( 1,2,1,2,3,4\right) }$, one has: 
\[
\varphi_3\left(\mathbf{i}_0 \right)= \E \Big[\left(\tilde  I^{(1)}_{1,2}\right)^2 \tilde I^{(1)}_{3,4} \Big] =O\left( n^{-1}\right) .
\]
\end{lemma}

\begin{proof}[Proof of Lemma~\ref{lem:2m}]
All formulas can be deduced from Equations~\eqref{eq:bb1}--\eqref{eq:bb16} in Section~\ref{subsec:ausmom} after some tedious calculations that have been checked with computer algebra systems. Exemplary calculations can be found in Section~\ref{sec:2m}.
\end{proof}

\begin{proof}[Proof of Lemma~\ref{lem:4m}]
We define $\ds{\psi : S_{8}\mapsto S_4 }$ by: 
\[
\psi(x ) = \frac{x}{2}1_{\{x\text{ even}\}}  +\frac{x+1}{2}1_{\{x\text{ odd}\}}.
\]

First of all, note that, for any $1\le i\neq j\le n$ and $1\le p\le d$:
\begin{align}
\label{itilde}
\tilde I^{(1)}_{i,j} &=  \ff{6n} + \ff{n+1}\sum_{\ell=1}^n \Xi_{i,\ell}\cdot \Xi_{j,\ell},\quad\text{ where }~ \Xi_{i,\ell}=  1_{\{ R_{i1}\leq \ell \}} - \frac{\ell}{n}.
\end{align}
Let $\mathbf i =\mathbf{i}_{1:8} \in \mathcal{J}^2$. Then, together with the multinomial formula from (\ref{multinomial}),
\begin{align}
\varphi_{4}\left(\mathbf{i} \right)
&=  \nonumber
\E\Big[\prod_{s=1}^4 \Big\{ \ff{6n} + \ff{n+1}\sum_{\ell=1}^n\Big(  \prod_{j \in \psi^{-1}(\{s\})} \Xi_{i_j,\ell}\Big)  \Big\} \Big] \\
&=  \nonumber
\frac{1}{(6n)^4} + \sum_{A \subset S_4, A \ne \emptyset}\ff{\left( 6n\right)^{\zp[b]{A^c}}(n+1)^{\zp[b]{A}} }\E\Big[\prod_{s\in A} \sum_{\ell=1}^n \prod_{j \in \psi^{-1}(\{s\})}\Xi_{i_j,\ell} \Big] \\
&=   \label{eq:phi1i} 
\frac{1}{(6n)^4} +\sum_{A \subset S_4, A \ne \emptyset}\ff{\left( 6n\right)^{\zp[b]{A^c}}(n+1)^{\zp[b]{A}} }  \sum\limits_{ \bm \ell = (\ell_s)_{s \in A} \in (1:n)^{|A|}} E_{n,\bm \ell} (A),
\end{align}
where, for nonempty $\ds{A\subset S_4}$ and $\ell =(\ell_s)_{s \in A} \in \{1, \dots, n\}^{|A|}$,  
\begin{align} \label{eq:enlab}
E_{n,\bm \ell} (A)= \E\Big[ \prod\limits_{{j \in\psi^{-1}(A) }}\Xi_{i_j,\ell_{\psi(j)}}\Big]
\end{align}
(note that we notationally suppress the dependence on $\mathbf i$). Note that $|\Xi_{i_j,\ell_{\psi(j)}}| \le 1$ and hence $|E_{n,\bm \ell} (A)| \le 1$.  As a consequence, for proving Lemma~\ref{lem:4m}, it is sufficient to restrict the sum over $\bm \ell$ in \eqref{eq:phi1i} to the case where at least one coordinate of $\bm \ell$ is larger than 7. Hence, the proof of Lemma~\ref{lem:4m} is finished once we show that 
 \begin{align}\label{eq:summand_rate}
	E_{n,\bm \ell} (A)&=O\left( n^{\kappa_{A,\mathbf{i}}}\right),\quad \kappa_{A,\mathbf{i}}:= 8-\zp[b]{\mathbf{i}}-\zp[b]{A} 
\end{align}
for all $\varnothing \ne A \subset S_4$, all $\bm \ell =(\ell_s)_{s\in A} \in \{1, \dots, n \}^{|A|}$ with at least one coordinate larger than $7$ and all $\mathbf i \in \{1, \dots, n \}^8$ with $\ds{\zp[b]{\mathbf{i}}=5,\ldots,8}$.  Clearly, it is sufficient to consider the case where the components of $\bm \ell$ are non-decreasing (by symmetry). 

The proof of \eqref{eq:summand_rate} is based on careful case-by-case study. First, since $|E_{n,\bm \ell} (A)| \le 1,$ nothing has to be shown for the case $\kappa_{A,i}\ge 0$.  The other cases are treated in the subsequent Sections~\ref{sect:ka-1} and \ref{sect:ka-2} and in Appendix~\ref{sect:ka-3} and  \ref{sect:ka-4}. 
\end{proof}

\subsubsection{Proof of \eqref{eq:summand_rate} for $\kappa_{A,i}= -1$.}\label{sect:ka-1}
Throughout, we will write $R_i$ instead of $R_{ip}$ for simplicity. For $D \subset \{1, \dots, n\}$, let
\[
\ds{\mathcal{G}_{D} = \sigma\left(R_{i}: i \in D \right)}.
\]
Some straightforward arguments imply that, for $i\in\{1, \dots, n\}$ and $D \subset\{1, \dots, n\}\setminus \{i\}$, we have
\begin{align} \label{eq:condcdf}
\p(R_i \le k \mid \mathcal G_D) 
&=
\begin{cases} \dfrac{k-Z_D(k)}{n-|D|} &, Z_D(k) \le k, \\
0&, \text{else},
\end{cases}
\end{align}
where $Z_D(k) = \sum_{j\in D} 1_{\{R_j \le k\}}$. Fix some arbitrary index $h \in \psi^{-1}(A)$ with $\ell_{\psi(h)}  \ge 7$ and let $B_h =  \psi^{-1}(A) \setminus \{h\}$. Then, by iterated expectation,
\begin{align}  \label{eq:eit}
E_{n,\bm \ell} (A)
&= 
\E\Big[ \prod\limits_{{j \in \psi^{-1}(A)}}\Xi_{j,\ell_{\psi(j)}}\Big]
=
\E\Big[ \E\Big[\Xi_{h,\ell_{\psi(h)}} | \mathcal G_{B_h} \Big] \prod\limits_{{j \in B_h}}\Xi_{j,\ell_{\psi(j)}}\Big]
\end{align}
In view of \eqref{eq:condcdf}, since $\ell_{\psi(h)}  \ge 7$ by assumption,  we have $\ell_{\psi(h)} \ge  |B_h| = 2|A|-1$ and  may hence write
\begin{align}  \label{eq:innercond}
\E\left[\Xi_{h, \ell_{\psi(h)}}\big|\mathcal{G}_{B_h}  \right] 
&=
\frac{\ell_{\psi(h)} - Z_{B_h}(\ell_{\psi(h)})}{n-2|A|+1}-\frac{\ell_{\psi(h)}}{n} 
=
\frac{(2|A|-1)\ell_{\psi(h)} - n  Z_{B_h}(\ell_{\psi(h)} )}{n(n-2|A|+1)}.
\end{align}
As a consequence, in view of the fact that $\Xi_{j,\ell_{\psi(j)}}$ is bounded by 1 and that $Z_{B_h}(\ell_{\psi(h)} )  \le 2|A|-1$, we immediately obtain that the integrand in the outer expectation in \eqref{eq:eit} is uniformly bounded by $2(2|A|-1)/(n-2|A|+1)=O(n^{-1})$.   \qed

\subsubsection{Proof of \eqref{eq:summand_rate} for $\kappa_{A,i}= -2$.}\label{sect:ka-2} 

We only give the proof for the case $(|\mathbf{i}|,|A|)=(8,2)$; the other two cases $(|\mathbf{i}|,|A|)=(7,3)$ and $(|\mathbf{i}|,|A|)=(6,4)$ may be treated similarly by a careful selection of the conditioning variables in the subsequent arguments.

When $\ds{\zp[b]{\mathbf{i}} = 8}$, we can assume, by equidistribution, that $\mathbf i=(1, 2, \dots, 8)$; i.e., $i_j=j$ for $j\in S_8$. Likewise, we may assume that $\ell_1 \le \ell_2$. We need to show that 
\begin{align} \label{eq:goal8}
	E_{n,\bm \ell} (A)& = \E\Big[ \prod\limits_{{j \in\psi^{-1}(A) }}\Xi_{j,\ell_{\psi(j)}}\Big]  = O(n^{-\zp[b]{A}}) . 
\end{align}
for every $A \subset \{1, \dots, 4\}$ with $|A|=2$. 

We start from \eqref{eq:eit} and \eqref{eq:innercond} and invoke another application of iterated expectation. More precisely, choose $h' \in \psi^{-1}(A)$ such that $\psi(h')=\psi(h)$ and write
\[
E_{n,\bm \ell} (A)
=
\E\Big[  \E\Big[  
\frac{(2|A|-1)\ell_{\psi(h)} - n  Z_{B_h}(\ell_{\psi(h)} )}{n(n-2|A|+1)}
\cdot \Xi_{h',\ell_{\psi(h')}} \Big| \mathcal{G}_{B_{h,h'}} \Big]
\cdot \prod\limits_{j \in \mathcal{G}_{B_{h,h'}} }\Xi_{j,\ell_{\psi(j)}} \Big],
\]
where $B_{h,h'}=\psi^{-1}(A) \setminus \{h,h'\}$ (recall that $\ell_{\psi(h)} \ge 7 \ge |B_h|$). 

Next, let $|A|=2$, which is generically represented by the particular case $A=\{1,2\}$ with $\psi^{-1}(A)=\{1, 2,3,4\}$. Further, for simplicity,  we restrict attention to $h=1$ and $h'=2$.   Then, observing that $\ell \equiv \ell_{\psi(1)}=\ell_{\psi(2)}$, 
 the previous display simplifies to
\begin{align}\label{eq:cui}
E_{n,\bm \ell} (A)
=
\E\Big[  \E\Big[  
\frac{3\ell - n  Z_{\{2,3,4\}}(\ell)}{n(n-3)}
\cdot \Xi_{2,\ell} \Big| \mathcal{G}_{\{3,4\}} \Big]\cdot \prod\limits_{j \in \{3,4\} }\Xi_{j,\ell_{\psi(j)}} \Big].
\end{align}
We start by expanding the expression inside the conditional expectation, which may be written as
\begin{align*}  
&\phantom{{}={}}
\frac1{n^2(n-3)} \Big\{3 n  \ell 1_{\{R_{2} \le \ell \}}- 3 \ell^{2} - n^2 \sum_{j \in \{2,3,4\}}  1_{\{R_{j} \le \ell, R_{2} \le \ell\}} + n \ell \sum_{j \in \{2,3,4\}}  1_{\{R_{j} \le \ell\}}
\Big\} \\
&=
\frac1{n^2(n-3)} \Big\{(4 n  \ell -n^2) 1_{\{R_{2} \le \ell \}}- 3 \ell^{2} - n^2 \sum_{j \in \{3,4\}}  1_{\{R_{j} \le \ell, R_{2} \le \ell\}} + n \ell \sum_{j \in \{3,4\}}  1_{\{R_{j} \le \ell\}}
\Big\} .
\end{align*}
By \eqref{eq:condcdf}, we obtain, for $j\in\{3,4\}$,
\[
\p(R_2\le \ell, R_{j}\le \ell \mid \mathcal{G}_{\{3,4\}}) 
= 
1_{\{R_{j} \le \ell\}} \p(R_2\le \ell  \mid \mathcal{G}_{\{3,4\}})
=
1_{\{R_{j} \le \ell\}}  \dfrac{\ell-Z_{\{3,4\}}(\ell)}{n-2}.
\]
As a consequence, the conditional expectation in \eqref{eq:cui} can be written as
\begin{align*}
&\phantom{{}={}}
\frac1{n^2(n-3)}  \Big\{(4 n \ell-n^2) \dfrac{\ell-Z_{\{3,4\}}(\ell)}{n-2} - 3\ell^2  - n^2  \dfrac{\ell-Z_{\{3,4\}}(\ell)}{n-2}  Z_{\{3,4\}}(\ell)  + n\ell Z_{\{3,4\}}(\ell) \Big\} \\
&=
\frac1{n^2(n-2)(n-3)} \Big\{ F_{n1} + F_{n2} \Big\}, 
\end{align*}
where
\[
F_{n1} = n\ell^2-n^2\ell  , \qquad F_{n2} = 6 \ell^2 + Z_{\{3,4\}}(n^2-6n\ell+n^2Z_{\{3,4\}}).
\]
Hence, the assertion in \eqref{eq:goal8} with $|A|=2$ is shown once we prove that
\[
 \E\Big[F_{n\kappa} \prod_{j\in\{3,4\}} \Xi_{j, \ell_{\psi(j)}} \Big] =O(n^2), \quad \kappa \in \{1,2\}.
\]
For $\kappa=2$, the assertion is obvious in view of the fact that $|\Xi_{j, \ell_{\psi(j)}}| \le 1$ and that $|\ell| \le n$ and $Z_{\{3,4\}} \le 2$, which implies $|F_{n2}| \le 24n^2$. For $\kappa=1$, we need to condition once again: writing $\ell'=\ell_{\psi(3)}=\ell_{\psi(4)}$ and using that $|F_{n1}|\le n^3$, we have, by the arguments that lead to \eqref{eq:innercond},
\begin{align*}
 \E\Big[F_{n1} \prod_{j\in\{3,4\}} \Xi_{j, \ell_{\psi(j)}} \Big] 
&=
F_{n1}\E\Big[ \E\Big[ \Xi_{3, \ell'} \mid \mathcal G_{\{4\}} \Big]  \Xi_{4, \ell'}  \Big]   
 =
 F_{n1}  \E\Big[ \Big( \frac{\ell'  - n 1_{\{R_4 \le \ell'\}}}{n(n-1)}  \Big)
 \Xi_{4, \ell_{\psi(4)}}  \Big].
\end{align*}
The assertion then follows from $|F_{n1}| \le n^3, \ell' \le n$ and  $|\Xi_{4, \ell_{\psi(4)}}|\le 1$. \qed

\begin{proof}[Proof of Lemma \ref{lem:3m}]
In view of \eqref{itilde}, we may write
\[
\varphi_3\left( \mathbf{i}_0 \right) 
= 
\E \Big[\left(\tilde  I^{(1)}_{1,2}\right)^2 \tilde I^{(1)}_{3,4} \Big]
=
\E\Big[ \Big(\frac1{6n} + \frac1{n+1} \sum_{\ell=1}^n \Xi_{1,\ell} \Xi_{2,\ell}\Big)^2 \Big(\frac1{6n} + \frac1{n+1} \sum_{\ell=1}^n \Xi_{3,\ell} \Xi_{4,\ell}\Big) \Big],
\]
and we need to show that this expression is $O(n^{-1})$.
Expanding the product and using the fact that $\Xi_{i,\ell}$ is bounded by one, the only non-trivial summand is
\[
\frac{1}{(n+1)^3} \sum_{\ell_1, \ell_2, \ell_3=1}^n \E[ \Xi_{1,\ell_1} \Xi_{2,\ell_1} \Xi_{1,\ell_2} \Xi_{2,\ell_2} \Xi_{3,\ell_3} \Xi_{4,\ell_3}]
\]
It is sufficient to show that each summand with $\ell_3 \ge 7$ is $O(n^{-1})$, uniformly in $\bm \ell=(\ell_1, \ell_2, \ell_3)$. Conditioning on $\mathcal G_{\{1,2,3\}}$, and applying the same argument that lead to \eqref{eq:eit}, we have
\[
\E[ \Xi_{1,\ell_1} \Xi_{2,\ell_1} \Xi_{1,\ell_2} \Xi_{2,\ell_2} \Xi_{3,\ell_3} \Xi_{4,\ell_3}]
=
\Exp\Big[ \frac{3\ell_3-nZ_{\{1,2,3\}}(\ell_3)}{n(n-3) }  \Xi_{1,\ell_1} \Xi_{2,\ell_1} \Xi_{1,\ell_2} \Xi_{2,\ell_2} \Xi_{3,\ell_3} \Big]
\]
The integrand is bounded by $\frac{6}{n-3}$, which implies the assertion.
\end{proof}

\subsection{Results and proofs for Step 1} \label{sec:ps1}

\begin{proof}[Proof of Proposition~\ref{prop:tt}]
The first statement  regarding $T_n(2)$ readily follows from the fact that, for any arbitrary function $f$,
$
\sum_{i \ne j} f(R_{ip}, R_{jp}) = \sum_{i \ne j} f(i,j).
$
Concerning the case $k\ge 3$, we may use \eqref{eq:munk} to write 
	 \begin{align*}
	 T_n(k) - \nu_n(k)
	 &=	
	 \sum_{\mathbf{p}_k\in \mathcal{P}(d,k)} \left\{	S^{\textrm{M}}_{n,\mathbf{p}_k}-\E\left[S^{\textrm M}_{n,\mathbf{p}_k}\right]\right\}
	 = M_n(k) +  N_n(k),
	 \end{align*}
	 where 
	 \begin{align*}
	 M_n(k) &= \sum_{\mathbf{p}_k\in \mathcal{P}(d,k)}\left\{M_{n,\mathbf{p}_k}-\left(n-1 \right)\left( \frac{-1}{6n}\right)^{k} \right\}, \quad
		N_n(k) =
		\sum_{\mathbf{p}_k\in \mathcal{P}(d,k)}\left\{	N_{n,\mathbf{p}_k}-\left(\ff{6}-\ff{6n} \right)^{k}\right\},  
	\end{align*}
	and where
	\begin{align*} 
	 M_{n,\mathbf p_k} = \frac1n \sum_{i\ne j}^n \prod_{\ell=1}^k I_{i,j}^{(p_\ell)}, \qquad
 N_{n,\mathbf p_k} = \frac1n \sum_{i=1}^n \prod_{\ell=1}^k I_{i,i}^{(p_\ell)}.
\end{align*}

Since $\tilde M_n(k) = \tilde M_n(k) +  \tilde N_n(k)$, it is hence sufficient to show that
\begin{align} \label{eq:suff1}
M_n(k)  =  \tilde M_n(k) + o_\Prob\left( \delta_n(k)\right) , \qquad 
N_n(k)  =  \tilde N_n(k) + o_\Prob\left( \delta_n(k)\right) .
\end{align}

Regarding the first assertion in \eqref{eq:suff1}, note that we may rewrite
\begin{align} \label{eq:mnk1}
M_n(k)
&= 
\frac{1}{n}\sum_{\mathbf{p}_k\in \mathcal{P}(d,k)}\sum_{i\neq j}^n\left\{ \prod_{\ell=1}^k I^{(p_\ell)}_{i,j} -\left( \frac{-1}{6n}\right)^{k} \right\}.
\end{align}

 As a consequence of the multinomial identity (\ref{multinomial}) and the definition of $\ds{\tilde I_{i,j}^{(p)}}$ in \eqref{eq:bii},
	\begin{align*}
	\prod_{\ell=1}^k   I^{(p_\ell)}_{i,j}=\prod_{\ell=1}^k \left(  \tilde I^{(p_\ell)}_{i,j} -\ff{6n}\right) &= \prod_{\ell=1}^k  \tilde I^{(p_\ell)}_{i,j} + \sum\limits_{\substack{A\subset S_k \\ \zp[b]{A}<k}} 
\left(\frac{-1}{6n}\right)^{\zp[b]{A^c}}
	\prod_{\ell\in A} \tilde I^{(p_\ell)}_{i,j} ,	
	\end{align*}
	which allows to rewrite \eqref{eq:mnk1} as
	\begin{align}
	\nonumber
	M_n(k)&= \frac{1}{n}\sum_{\mathbf{p}_k\in \mathcal{P}(d,k)}\sum_{i\neq j}^n\left\{  \prod_{\ell=1}^k  \tilde I^{(p_\ell)}_{i,j} + \sum\limits_{\substack{A\subset S_k \\ 0<\zp[b]{A}<k}}  \left(\frac{-1}{6n}\right)^{\zp[b]{A^c}} \prod_{\ell\in A} \tilde I^{(p_\ell)}_{i,j} \right\}
		\\ 
		&= \tilde M_n(k) +  \frac{1}{n}\sum\limits_{\substack{A\subset S_k \\ 0<\zp[b]{A}<k}}  \left(\frac{-1}{6n}\right)^{\zp[b]{A^c}}  \sum_{\mathbf{p}_k\in \mathcal{P}(d,k)} \sum_{i\neq j}^n \prod_{\ell\in A} \tilde I^{(p_\ell)}_{i,j}
		\equiv\tilde M_n(k) + \tilde U_n(k),
		\label{eq:mtmu}
	\end{align}
	where
	\[
	\tilde U_n(k) = \sum\limits_{\substack{A\subset S_k \\ 0<\zp[b]{A}<k}} 
	\left(\frac{-1}{6n}\right)^{\zp[b]{A^c}} 
	\sum_{\mathbf{p}_k\in \mathcal{P}(d,k)} \tilde M_{n,\mathbf{p}_A}
	\]
	and
where $\tilde M_{n,\mathbf{p}_k}$ and $\tilde M_n(k)$ are defined in \eqref{eq:mnnn} and \eqref{eq:mnnn2}, respectively. Thus, the first assertion in \eqref{eq:suff1} is shown once we prove that $\tilde U_n(k) = o_\Prob\left( \delta_n(k)\right) $. 
For that purpose, consider, for fixed $A\subset \{1, \dots, k\}$ with $0<|A|<k$, the partition mapping
\[
\pi_A:\mathcal{P}(d,k)\to \mathcal{P}(d,\zp[b]{A})\times\mathcal{P}(d,k-\zp[b]{A}) ,
\qquad \mathbf p_k \mapsto \pi_A(\mathbf{p}_{k}) = \left( \mathbf{p}_A,\mathbf{p}_{A^c}\right), 
\]
which clearly is a bijection. 
We may then write
\begin{align}  \label{eq:un2}
\tilde U_n (k)
&= \nonumber
\sum\limits_{\substack{A\subset S_k \\ 0<\zp[b]{A}<k}} \left(\frac{-1}{6n}\right)^{\zp[b]{A^c}} 
 \sum_{ \left( \mathbf{p}_A,\mathbf{p}_{A^c}\right) \in \mathcal{P}(d,\zp[b]{A})\times\mathcal{P}(d,k-\zp[b]{A}) } 
 \tilde M_{n,\mathbf{p}_A} \\
&=
\sum\limits_{\substack{A\subset S_k \\ 0<\zp[b]{A}<k}} \left(\frac{-1}{6n}\right)^{\zp[b]{A^c}}  \binom{d}{k-|A|}
 \sum_{ \mathbf{p}_A  \in \mathcal{P}(d,\zp[b]{A}) } \tilde M_{n,\mathbf{p}_A}.
\end{align}
By Proposition~\ref{prop_s1} applied with $k= |A|$ in the proposition's notation and $(4k-7)$-wise independence, we have
\[
{\binom{d}{\zp[b]{A}}^{-\frac{1}{2}}}\sum_{\mathbf{p}_A\in \mathcal{P}(d,\zp[b]{A})} \tilde M_{n,\mathbf{p}_A} =O_\Prob(1).
\]
As a consequence, each summand in the outer sum in \eqref{eq:un2} is of the order $O_\Prob( n^{|A|-k} d^{k-\frac{|A|}2})) $,
which, together with $\delta_n(k)^{-1}=O( d^{-\frac{k}{2}})$, implies that
\[
\frac{\tilde U_n (k)}{\delta_n(k)} 
= 
\sum\limits_{\substack{A\subset S_k \\ 0<\zp[b]{A}<k}} 
O_\Prob\Big(\Big( \frac{d}{n^2}\Big)^{\frac{k-|A|}{2}}\Big) = O_\Prob\Big( \Big( \frac{d}{n^2}\Big)^{\frac{1}{2}} \Big)  =o_\Prob(1)
\]
by the assumption on $d$.

It remains to prove the second assertion in \eqref{eq:suff1}. In view of $I^{(p)}_{i,i}=\tilde I^{(p)}_{i,i}+\left(\frac{1}{6}-\frac{1}{6n}\right)$ by the definition in \eqref{eq:bii}, exactly the same arguments that lead to $M_n(k) = \tilde M_n(k) + \tilde U_n(k)$ in \eqref{eq:mtmu} allow to write $N_n(k) = \tilde N_n(k) + \tilde V_n(k)$, where
\begin{align*}
\tilde V_n(k)
&=
\sum\limits_{\substack{A\subset S_k \\ 0<\zp[b]{A}<k}} 
{\left( \ff{6}-\ff{6n}\right)^{\zp[b]{A^c}}}
\sum_{\mathbf{p}_k\in \mathcal{P}(d,k)}\tilde N_{n,\mathbf{p}_A} \\
&=
\sum\limits_{\substack{A\subset S_k \\ 0<\zp[b]{A}<k}} {\left( \ff{6}-\ff{6n}\right)^{\zp[b]{A^c}}}  \binom{d}{k-\zp[b]{A}}
\sum_{\mathbf{p}_A\in \mathcal{P}(d,\zp[b]{A})} \tilde N_{n,\mathbf{p}_A}.
\end{align*}
From Proposition~\ref{prop:nn}, 
we have 
\[
{\binom{d}{\zp[b]{A}}^{-\frac{1}{2}}}\sum_{\mathbf{p}_A\in \mathcal{P}(d,\zp[b]{A})} \tilde N_{n,\mathbf{p}_A} =O_\Prob\left(n^{-\ff{2}}\right).
\]
As consequence, since $\binom{d}{\ell}=O(d^{\ell})$ and $\delta_n^{-1}(k)=O(d^{-\frac{k}{2}})$,
\[
\frac{\tilde V_n (k)}{\delta_n(k)} = \sum\limits_{\substack{A\subset S_k \\ 0<\zp[b]{A}<k}} 
O_\Prob\left( n^{-\ff{2}}d^{\frac{k-|A|}{2}}\right)  = O_\Prob\left( n^{-\ff{2}} d^{\frac{k-1}{2}}\right) ,
\]
which converges to zero by the assumption on $d$.
\end{proof}

\subsection{Results and proofs for Step 2}\label{sec:ps2}

\begin{proof}[Proof of Proposition~\ref{prop:nn}] 
Recall $\tilde N_{n,A}$ from \eqref{eq:mnnn} and $\delta_n(k)$ from \eqref{eq:deltan}. It is sufficient to show that
\begin{align} \label{eq:l2conv}
\delta_n^{-2}(k) \E\Big[ \tilde N_{n}^2(k) \Big] =O(n^{-1}).
\end{align}
For $\mu\in\{0,\dots, k\}$, let
\[
\mathcal{Z}_{k}(\mu) = \left\{ (\mathbf{p}_{k,1},\mathbf{p}_{k,2})\in \mathcal{P}(d,k) \times \mathcal{P}(d,k) : \zp[b]{\mathbf{p}_{k,1}\cap \mathbf{p}_{k,2}}=\mu   \right\}.
\]
We may then decompose
\begin{align*}
\tilde N_{n}^2(k)
=
\sum\limits_{\substack{\mathbf{p}_{k,1},\mathbf{p}_{k,2} \in \mathcal{P}(d,k) }} 
  \tilde N_{n, \mathbf{p}_{k,1} }(k) \tilde N_{n,\mathbf{p}_{k,2}  }(k)
= 
\sum_{\mu=0}^k \tilde N_{n,\mu}(k),
\end{align*}
where
\[
\tilde N_{n,\mu}(k) 
=
\sum\limits_{(\mathbf{p}_{k,1},\mathbf{p}_{k,2} )\in \mathcal{Z}_k(\mu) }   
\tilde N_{n, \mathbf{p}_{k,1} }(k) \tilde N_{n,\mathbf{p}_{k,2}  }(k).
\]
As a consequence, the assertion in \eqref{eq:l2conv} is shown once we prove that, for each $\mu\in\{0,\dots, k\}$,
\begin{align} \label{eq:l2conv2}
\delta_n^{-2}(k)\E[\tilde N_{n,\mu}(k)] = O(n{-1}).
\end{align}
We split the proof into three cases: $\mu=0, \mu\in\{1, \dots, k-1\}$ and $\mu=k$.

Consider $\mu=0$. Writing  $\mathbf p_{k,j}=(p_{1,j}, \dots, p_{\ell, j})$ for $j\in\{1,2\}$, we have
\[
 \tilde N_{n, \mathbf{p}_{k,1} }(k) \tilde N_{n,\mathbf{p}_{k,2}  }(k)
=
\frac{1}{n^2} \sum_{i,j=1}^n \prod_{\ell=1}^{k} \tilde I^{(p_{\ell,1})}_{i,i}\tilde I^{(p_{\ell,2})}_{j,j}.
\]
As a consequence, by $2k$-wise independence and since $\tilde I_{i,i}^{(p)}$ is centred, we have 
$
\E[\tilde N_{n,0}(k)] = 0,
$
which readily implies \eqref{eq:l2conv2} for $\mu=0$.

Consider $\mu\in\{1, \dots, k-1\}$. Then,  for each $(\mathbf{p}_{k,1},\mathbf{p}_{k,2} )\in \mathcal{Z}_k(\mu)$, we may decompose
\[
\prod_{\ell=1}^{k}\tilde I^{(p_{\ell,1})}_{i,i}\tilde I^{(p_{\ell,2})}_{j,j} 
=
\Big( \prod_{ p \in \mathbf{p}_{k,1} \cap \mathbf{p}_{k,2} }\tilde I^{(p)}_{i,i}\tilde I^{(p)}_{j,j}  \Big) \cdot
\Big( \prod_{ p \in \mathbf{p}_{k,1} \setminus \mathbf{p}_{k,2} }\tilde I^{(p)}_{i,i}  \Big) \cdot
\Big( \prod_{ p \in \mathbf{p}_{k,2} \setminus \mathbf{p}_{k,1} }\tilde I^{(p)}_{j,j}  \Big),
\]
with each product being non-empty. Since $\tilde I_{i,i}^{(p)}$ is centred, $2k$-wise independence implies
that the expression in the previous display has expectation zero. This implies \eqref{eq:l2conv2} for $\mu\in\{1, \dots, k-1\}$.

Finally, consider $\mu=k$, such that $p_{\ell,1}=p_{\ell,2}$ for all $\ell\in\{1, \dots, k\}$.  Hence, since $|\mathcal Z_k(k)|=\binom{d}{k}$ and by equidistribution,
\[
\E[\tilde N_{n,k}(k)]  
= 
\binom{d}{k} n^{-2} \Big( n \Exp\Big[ (\tilde I^{(p)}_{1,1})^2\Big]^k + n(n-1) \Exp\Big[\tilde I^{(p)}_{1,1}\tilde I^{(p)}_{2,2}\Big]^k \Big)
\]
Since $\delta_n^2(k)=\frac{2}{90} \binom{d}{k}$ and
$
\Exp\big[ (\tilde I^{(p)}_{1,1})^2\big] = \frac1{180}$ and $\Exp\big[\tilde I^{(p)}_{1,1}\tilde I^{(p)}_{2,2}\big] = O(n^{-1})
$
by  Lemma \ref{lem:2m}, we obtain that $\delta_n^{-2}(k)\E[\tilde N_{n,\mu}(k)]=O(n^{-1})$, which is \eqref{eq:l2conv2} for $\mu=k$.
\end{proof}

\subsection{Results and proofs for Step 3: marginal weak convergence} \label{sec:ps3a}
 
It is instructive to start proving marginal weak convergence in Proposition~\ref{prop:mm}, which we summarize in the following proposition. Joint convergence will be discussed in Section~\ref{sec:ps3b} in the supplementary material.

 \begin{proposition}\label{prop_s1}
 Let $k \in \N_{\ge 2}$ and assume $(4k-3)$-wise independence. Then,
\begin{align*}
 \frac{\tilde M_{n}(k) }{\delta_n(k)  }= \delta_n^{-1}(k)  \cdot 
 \sum_{ \mathbf{p}_k \in \mathcal{P}(d,k)  }\tilde{M}_{n,\mathbf{p}_{k}}&\weak\mathcal{N}\left( 0,1\right) .
 \end{align*}
 \end{proposition}
 
The key tool is the following martingale array central limit theorem:
\begin{theorem}[Corollary 3.1 in \citealp{hall1980martingale}]\label{TCL}	
Let $d=d_n \to \infty$. Let $\ds{\eta^2}$ be an a.s.\ finite r.v.\ and let $\ds{\{ (S_{n,r}, \mathcal{F}_{n,r}):1\leq r\leq d,n\geq 1\}}$ be a zero-mean, square integrable martingale array with differences $X_{n,r}=S_{n,r}-S_{n,r-1}$. If
$\mathcal{F}_{n,r}\subset \mathcal{F}_{n+1,r}$  for all $1\leq r\leq d$ and $n\geq 1$ and if
\begin{align}\label{lind1}
 		\forall \eps >0:\quad &\sum_{r=1}^{d}\E\left[ X_{n,r} 1_{ \{ \zp[b]{X_{n,r}}>\eps  \} }| \mathcal{F}_{n,r-1} \right]  \xrightarrow[n\to +\infty]{\p} 0,\\
		\label{lind2} 
 		&\sum_{r=1}^{d}\E \left[X^2_{n,r} | \mathcal{F}_{n,r-1}\right]  \xrightarrow[n\to +\infty]{\p} \eta^2,
 	\end{align}
then 
$S_{n,d}=\sum_{r=1}^{d} X_{n,r} \weak Z,$
where $Z$ is a random variable distributed as $\eta N$ with $N\sim \mathcal{N}(0,1)$ independent of $\eta$.
Moreover, the Lindeberg condition in \eqref{lind1} is a consequence of the Lyapunov condition: 
\begin{align}\label{lyapunov}
 		\sum_{r=1}^{d}\E \left[X^4_{n,r} | \mathcal{F}_{n,r-1}\right]  &\xrightarrow[n\to +\infty]{\p} 0.
 \end{align}. 
 \end{theorem}

\begin{proof}[Proof of  Proposition~\ref{prop_s1}]
We start by identifying $\ds{\tilde M_n(k) =S_{n, d}}$ as part of a martingale array. For $\ds{k \le r \le d}$, let
 \[
 S_{n,r} :=
 \frac1{\delta_n(k)} \sum\limits_{\substack{\mathbf p_k \in \mathcal P(d,k) \\ p_k \le r }} \tilde M_{n, \mathbf{p}_k}
 =
 \ff{\delta_n(k)}\sum_{p_{k}=k}^{r}\sum_{p_{k-1}=k-1}^{p_k-1}\cdots \sum_{p_1=1}^{p_2-1} \tilde{M}_{n,\mathbf{p}_{k} }
 \]
 with $\ds{\mathbf p_k = (p_1, \dots, p_k)}$ and $\ds{\tilde M_{n,\mathbf{p}_k}}$ from \eqref{eq:mnnn2}. We are going to apply Theorem~\ref{TCL} with $\ds{\eta=1}$,
 \begin{align} \label{eq:fnr}
 \mathcal{F}_{n,r} := \sigma \left\{\bm{U}^{(p)} : 1\leq p \leq r   \right\} ,\quad \bm{U}^{(p)} = \left(U_{ip} \right)_{1\le i \le n},
 \end{align}
 where $U_{ip}=F_p(X_{ip}) \sim \mathrm{Unif}([0,1])$
and
 \begin{align}  \label{eq:xnr}
 	X_{n,r} := X_{n,r}(k)
 	:=
 	\begin{cases}
 		\frac1{\delta_n(k)} \sum\limits_{\substack{\mathbf p_k \in \mathcal P(d,k) \\ p_k = r }} \tilde M_{n, \mathbf{p}_k}&, r \ge k, \\
 		0  &, r < k.
 	\end{cases}
 \end{align}
 Note that $\ds{S_{n,k}=X_{n,k}}$ and $\ds{S_{n,r}=0}$ for $\ds{r <k}$. Now, $(S_{n,r}, \mathcal F_{n,r})$ is a martingale array:
 First, note that $\ds{S_{n,r}}$ is centered by $k$-wise independence and the fact that $\ds{\tilde{I}^{(p)}_{i,j}}$ is centered. 
 Next, it is sufficient to show the martingale property for $r \ge k$. For that purpose, write:
 \begin{align*} 
 	S_{n,r} &=\ff{\delta_n(k) } 
 \sum\limits_{\substack{\mathbf p_k \in \mathcal P(d,k) \\ p_k \le r-1 }} \tilde{M}_{n,\mathbf{p}_{k} }+\ff{\delta_n(k)}
 	\sum\limits_{\substack{\mathbf p_k \in \mathcal P(d,k) \\ p_k = r }} \tilde{M}_{n,\mathbf{p}_{k}},
 \end{align*}
 where the first sum is equal to zero if $r=k$.
 Conditioning with respect to $\ds{\mathcal{F}_{n,r-1}}$, we have
 \begin{align*} 
 	{\delta_n(k) }\cdot \E[{S}_{n,r}|\mathcal{F}_{n,r-1}] 
	= 
	&\sum\limits_{\substack{\mathbf p_k \in \mathcal P(d,k) \\ p_k \le r -1}} \E\left[\tilde{M}_{n,\mathbf{p}_{k} }\big|\mathcal{F}_{n,r-1}\right] 
 	+\sum\limits_{\substack{\mathbf p_k \in \mathcal P(d,k) \\ p_k = r }} \E\left[\tilde{M}_{n,\mathbf{p}_{k} }\big|\mathcal{F}_{n,r-1}\right].
 \end{align*}
 The first sum is equal to $\ds{{\delta_n(k) }\cdot S_{n,r-1}}$, while the second sum has summands equal to
 \[
 \frac{2}{{n}} \sum_{i< j}^n \E\Big[\tilde I^{(r)}_{i,j} \cdot \prod_{\ell=1}^{k-1}\tilde I^{(p_\ell )}_{i,j}  ~\Big|~\bm {U}^{(p)},~1\le p\le r-1\Big],\quad 1 \le p_1<\cdots < p_{k-1}<r. 
 \]
 The latter quantity vanishes by $k$-wise independence and the fact that $\ds{\tilde{I}^{(p)}_{i,j}}$ is centered. Thus, the martingale property follows.
 
 We will next show \eqref{lind2} with $\eta=1$, for which it is sufficient to show that 
\[
\lim_{n\to\infty}  \E \Big[  \Big( \sum_{p_k=1}^{d}\E\left( X^2_{n,p_k} | \mathcal{F}_{n,p_k-1}  \right) - \eta^2\Big)^2  \Big] =0.
\] 
Consequently, since $\ds{\eta =1}$ is deterministic, it is enough to show that:
$$ 
\lim_{n\to\infty}  \E \Big[  \Big( \sum_{p_k=1}^{d}\E\left( X^2_{n,p_k} | \mathcal{F}_{n,p_k-1}  \right) \Big)^2  \Big]=1,
\quad  
\lim_{n\to\infty} \E \Big[ \sum_{p_k=1}^{d}\E\Big( X^2_{n,p_k} | \mathcal{F}_{n,p_k-1}  \Big)  \Big]=1.
$$
These convergences are a consequence of Lemma~\ref{lem_exp} and \ref{lem_cond_exp}.

Finally, since $L^1$-convergence implies convergence in probability and since $X_{nr}^4$ is non-negative, the Lyapunov condition (\ref{lyapunov}) is a consequence of Lemma~\ref{lyapunovlem}. Proposition~\ref{prop_s1} is now a consequence of Theorem~\ref{TCL}.
 \end{proof}

 \begin{lemma}\label{lem_exp}
  Let $k \in \N_{\ge 2}$ and assume $(2k-1)$-wise independence. Then,  with $X_{n,r}=X_{n,r}(k)$ from \eqref{eq:xnr},
 \begin{align*}
 \zeta_n 
 :=
 \E \Big[ \sum_{p_k=1}^{d}\E\left( X^2_{n,p_k} | \mathcal{F}_{n,p_k-1}  \right)  \Big] 
 =
 \sum_{p_k=1}^{d} \E \left[ X_{n,p_k}^2\right]\xrightarrow[n\to +\infty]{}1. 
 \end{align*}
 \end{lemma}
 
\begin{proof}
For any $\ds{ \mathbf{p}_{k,1}=(p_{1,1}, \dots, p_{k,1}),\mathbf{p}_{k,2}=(p_{1,2}, \dots, p_{k,2}) \in \mathcal{P}\left( d,k\right)}$ and $\ds{\mathbf{i}=(i_1,\ldots,i_4)\in \mathcal{J}}$ with $\mathcal J$ from \eqref{eq:jn}, recalling $\ds{\tilde M_{n,\mathbf{p}_k}}$ from \eqref{eq:mnnn2}, we have
\begin{align*}
\tilde M_{n,\mathbf{p}_{k,1} }\cdot {\tilde M}_{n,\mathbf{p}_{k,2}  }
&=
\frac{4}{n^2} \sum_{  \mathbf{i}\in \mathcal{J}} \prod_{\ell=1}^{k}  \tilde I^{(p_{\ell,1})}_{i_1,i_2}\tilde I^{(p_{\ell,2})}_{i_3,i_4},
\end{align*}
whence, by the definition of $X_{n,p_k}$ in \eqref{eq:xnr}, for $p_k\ge k$,
\begin{align}\label{eq:product2}
	X_{n,p_k}^2 =  
	\frac4{n^2\delta_n^2(k)} 
	\sum_{\substack{\mathbf{p}_{k,1} \in \mathcal P(d,k): p_{k,1}=p_k  \\ \mathbf{p}_{k,2} \in \mathcal P(d,k): p_{k,2}=p_k}}  
	\sum_{  \mathbf{i}\in \mathcal{J}} \Big(\prod_{\ell=1}^{k-1}  \tilde I^{(p_{\ell,1})}_{i_1,i_2}\tilde I^{(p_{\ell,2})}_{i_3,i_4} \Big) \cdot \tilde I^{(p_{k})}_{i_1,i_2}\tilde I^{(p_{k})}_{i_3,i_4} .
	\end{align}
Using $(2k-1)$-wise independence, we obtain	
\begin{align*}
	\zeta_n
		&=
		\frac{4}{n^2\delta^2_n(k) }
			\sum_{p_k=k}^{d} 
			\sum_{\substack{\mathbf{p}_{k,1} \in \mathcal P(d,k):p_{k,1}=p_k \\ \mathbf{p}_{k,2} \in \mathcal P(d,k): p_{k,2}=p_{k}}} 
			\sum_{  \mathbf{i}\in \mathcal{J}} \E \left[\prod_{\ell=1}^{k-1} \tilde  I^{(p_{\ell,1})}_{i_1,i_2}\tilde I^{(p_{\ell,2})}_{i_3,i_4} \right]\cdot \E\left[\tilde I^{(p_{k})}_{i_1,i_2}\tilde I^{(p_{k})}_{i_3,i_4}\right]  \\
			&=
		\frac{4}{n^2\delta^2_n(k) }
			\sum_{\substack{\mathbf{p}_{k,1}, \mathbf{p}_{k,2} \in \mathcal P(d,k)\\ p_{k,1}=p_{k,2} }} 
			\sum_{  \mathbf{i}\in \mathcal{J}} \E \left[\prod_{\ell=1}^{k-1} \tilde  I^{(p_{\ell,1})}_{i_1,i_2}\tilde I^{(p_{\ell,2})}_{i_3,i_4} \right]\cdot \E\left[\tilde I^{(p_{k,1})}_{i_1,i_2}\tilde I^{(p_{k,2})}_{i_3,i_4}\right].
	\end{align*}
	Decomposing the sum over $\mathcal J$ into $\mathcal{J}=\sqcup_{h=2}^4 \mathcal{I}_h$ with $\mathcal I_h$ from \eqref{eq:iell}, we have $\zeta_n=\zeta_n^{(2)}+\zeta_{n}^{(3)}+\zeta_{n}^{(4)}$, where
\begin{align*}
	\zeta_n^{(h)}
			&=
		\frac{4}{n^2\delta^2_n(k) }
			\sum_{\substack{\mathbf{p}_{k,1}, \mathbf{p}_{k,2} \in \mathcal P(d,k)\\ p_{k,1}=p_{k,2} }} 
			\sum_{  \mathbf{i}\in \mathcal{I}_h} \E \left[\prod_{\ell=1}^{k-1} \tilde  I^{(p_{\ell,1})}_{i_1,i_2}\tilde I^{(p_{\ell,2})}_{i_3,i_4} \right]\cdot \E\left[\tilde I^{(p_{k,1})}_{i_1,i_2}\tilde I^{(p_{k,2})}_{i_3,i_4}\right].
	\end{align*}
	Consider the expectation of the product on the right-hand side, i.e., 
	\begin{align} \label{eq:eprod}
	 \E \Big[\prod_{\ell=1}^{k-1} \tilde  I^{(p_{\ell,1})}_{i_1,i_2}\tilde I^{(p_{\ell,2})}_{i_3,i_4} \Big].
	 \end{align}
	   By $k$-wise independence and since $\tilde I_{i,j}^{(p)}$ is centered, the expectation is zero as soon as 
	\[
	\exists \ell \in \{1,\ldots,k-1\}~\forall \ell'\neq \ell:~ p_{\ell,1} \neq p_{\ell',2}.
	\]
Thus, the expectation is non-zero when 
\begin{align}\label{matching}
	\forall \ell \in \{1,\ldots,k-1\}~\exists \ell'\neq \ell: p_{\ell,1} = p_{\ell',2},
 \end{align}
which we coin \textit{$2$-matching condition}. In the latter case, there exists an integer $\ds{ \ell \leq k-1}$ such that $\ds{p_{k-1,1}=p_{\ell,2}}$. 
If $\ds{\ell < k -1}$, then, by $(2k-1)$-wise independence, we may split off `isolated' factors 
	\[
	\ds{\E\left[\tilde I^{(p_{\ell+1,2})}_{i_3,i_4}\right] \times \dots \times \E\left[\tilde I^{(p_{k-1,2})}_{i_3,i_4}\right]},
	\]
each of which is zero by centredness of $\tilde I^{(p)}_{i,j}$. Hence, the expression in \eqref{eq:eprod} can only be non-zero if $\ds{ p_{k-1,1}=p_{k-1,2} }$. Thus,
\begin{align*}
	\zeta^{(h)}_n 
	&= 
	\frac{4}{n^2\delta^2_n(k) }\cdot
	\sum_{\substack{\mathbf{p}_{k,1}, \mathbf{p}_{k,2} \in \mathcal P(d,k):\\ p_{k,1}=p_{k,2}, p_{k-1,1}=p_{k-1,2} }} 
	\sum_{  \mathbf{i}\in \mathcal{I}_h}  \E \left[\prod_{\ell=1}^{k-2} \tilde  I^{(p_{\ell,1})}_{i_1,i_2}\tilde I^{(p_{\ell,2})}_{i_3,i_4} \right] \cdot  \E\left[\tilde I^{(p_{k,1})}_{i_1,i_2}\tilde I^{(p_{k,2})}_{i_3,i_4}\right]^2.
	\end{align*}
By the same reasoning as before, the expectation of the product can only be non-zero when $p_{k-2,1}=p_{k-2,2}$. Hence, 
iterating the $2$-matching procedure, we obtain that
\begin{align*}
	\zeta^{(h)}_n 
	&= 
	\frac{4}{n^2\delta^2_n(k) }\cdot
	\sum_{\substack{\mathbf{p}_{k,1} \in \mathcal P(d,k)}} 
	\sum_{  \mathbf{i}\in \mathcal{I}_h} \E\left[\tilde I^{(p_{k,1})}_{i_1,i_2}\tilde I^{(p_{k,2})}_{i_3,i_4}\right]^k.
\end{align*}
As a consequence,  by equidistribution, \eqref{eq:deltan}, \eqref{cardinal} and Lemma~\ref{lem:2m}, we have
\[
\zeta^{(2)}_n  
=  
\frac{4 }{n^2\delta^2_n(k) }\cdot |\mathcal P(d,k)|\cdot |\mathcal I_2| \cdot \E\left[(\tilde I^{(p)}_{1,2})^2\right]^{k}   = \frac{n-1}{n} =1+o(1).
\]
Likewise, for $h\in\{3,4\}$, using \eqref{eq:vfi2} from Lemma~\ref{lem:2m},
\[
\zeta^{(h)}_n  
=  
\frac{4 }{n^2\delta^2_n(k) }\cdot |\mathcal P(d,k)|\cdot |\mathcal I_h| \cdot O(n^{k(2-h)})  = O(n^{(k-1)(2-h)}) = o(1),
\]
which implies the assertion.
\end{proof}

\begin{lemma}\label{lem_cond_exp}
 Let $k \in \N_{\ge 2}$ and assume $(4k-4)$-wise independence. Then, recalling $X_{n,r}=X_{n,r}(k)$ from \eqref{eq:xnr},
\begin{align*}
\Lambda_n(k)
:=
\E \Big[ \Big(\sum_{p_k=k}^{d_n}\E\left( X^2_{n,p_k} | \mathcal{F}_{n,p_k-1}  \right) \Big)^2  \Big] \xrightarrow[n\to +\infty]{} 1.
\end{align*}
\end{lemma}

 \begin{proof}
	Decompose 
	$\Lambda_n(k) =\Lambda_{n,1}(k) +2\cdot\Lambda_{n,2}(k)  ,$
	where
\begin{align}
\Lambda_{n,1} 
&:= \nonumber
\Lambda_{n,1}(k)  =\sum_{p_k=k}^{d}\E \left[ \left( \E\left( X^2_{n,p_k} | \mathcal{F}_{n,p_k-1}  \right) \right)^2  \right],  \\
\Lambda_{n,2} 
&:=  \label{eq:ln2k}
\Lambda_{n,2}(k) =\sum_{k \le p_k< p'_k \le d}  \E \left[\E\left( X^2_{n,p_k} | \mathcal{F}_{n,p_k-1}  \right)  \E\left( X^2_{n,p'_k} | \mathcal{F}_{n,p'_k-1}  \right)  \right]. 
\end{align}
It is sufficient to show that
\begin{align} \label{eq:lconv}
		\lim_{n\to\infty}\Lambda_{n,1}=0, 
		\qquad 
		\lim_{n\to\infty}\Lambda_{n,2}=\frac12.
\end{align}	
For that purpose, let
\begin{align*}
	\chi_{n,p_k} := \E\left( X^2_{n,p_k} | \mathcal{F}_{n,p_k-1}  \right) .
\end{align*}	
By \eqref{eq:product2} and $(2k-1)$-wise independence, we can write	
\begin{align*}
\chi_{n,p_k} 
&= 
\frac{4}{n^2\delta^2_n(k) }\cdot
\sum_{\substack{\mathbf{p}_{k,1}, \mathbf{p}_{k,2} \in \mathcal P(d,k)\\ p_{k,1}=p_{k,2}=p_k }} 
\sum_{  \mathbf{i}\in \mathcal{J}} 
\E\left[\tilde I^{(1)}_{i_1,i_2}\tilde I^{(1)}_{i_3,i_4}\right]\cdot \prod_{\ell=1}^{k-1} \tilde  I^{(p_{\ell,1})}_{i_1,i_2}\tilde I^{(p_{\ell,2})}_{i_3,i_4}.
\end{align*}
For $m \in\{1, \dots,  k\}$, 
	$\ds{\mathbf{i}=\left(i_1,\ldots,i_8 \right)\in \mathcal{J}^2 }$ and 
	$\mathbf{P}_k=(\mathbf{p}_{k,1}, \dots, \mathbf{p}_{k,4}) \in \mathcal P(d,k)^4$ with
	$\mathbf{p}_{k,i}=(p_{1,i}, \dots, p_{k,i}) \in \mathcal P(d,k)$ for $i\in\{1,\dots, 4\}$,
let
	\begin{align}\label{varphi_def}
\varphi_m\left( \mathbf{P}_{k}, \mathbf{i} \right) 
	&= 
	\E\left[\prod_{\ell=1}^{m} \tilde  I^{(p_{\ell,1})}_{i_1,i_2}\tilde  I^{(p_{\ell,2})}_{i_3,i_4}\tilde I^{(p_{\ell,3})}_{i_5,i_6}\tilde I^{(p_{\ell,4})}_{i_7,i_8}\right].
	\end{align}
The previous notation allows to write
\begin{align*}
\Lambda_{n,1} 
&= 
\frac{16}{n^4\delta^4_n(k) }\cdot
	\sum\limits_{\substack{\mathbf{P}_{k}  \in \mathcal P(d,k)^4  \\ p_{k,1}=p_{k,2}=p_{k,3}=p_{k,4}}} 
	\sum_{  \mathbf{i}\in \mathcal{J}^2}  
	\varphi_{k-1}\left( \mathbf{P}_{k},\mathbf i \right)
	\varphi_2\left(\mathbf i_{1:4} \right)
	\varphi_2\left(\mathbf i_{5:8} \right),
	\\
\Lambda_{n,2} 
&= 
\frac{16}{n^4\delta^4_n(k) }\cdot
	\sum\limits_{\substack{\mathbf{P}_{k}  \in \mathcal P(d,k)^4  \\ p_{k,1}=p_{k,2}<p_{k,3}=p_{k,4}}} 
	\sum_{  \mathbf{i}\in \mathcal{J}^2}  
	\varphi_{k-1}\left( \mathbf{P}_{k},\mathbf i \right)
	\varphi_2\left(\mathbf i_{1:4} \right)
	\varphi_2\left(\mathbf i_{5:8} \right).
\end{align*}
Now, the same arguments that lead to \eqref{matching} imply that $\varphi_{k-1}\left( \mathbf{P}_{k},\mathbf i \right)$ can only be non-zero when $\mathbf{P}_k = \left(\mathbf{p}_{k,1},\mathbf{p}_{k,2},\mathbf{p}_{k,3},\mathbf{p}_{k,4}  \right) \in \mathcal{P}\left(d,k \right)^4$ satisfies the 2-matching condition
\begin{align}\label{pairing}
	\forall~ \ell \in\{1, \dots, k-1\}, i\in \{1,\ldots,4\} ~ \exists ~(\ell',j) \neq (\ell, i)~ \text{such that } p_{\ell,i}=p_{\ell',j}.
\end{align}
Hence, we may rewrite 
\begin{align*}
\Lambda_{n,1} 
&= 
\frac{16}{n^4\delta^4_n(k) }\cdot
	\sum\limits_{\substack{\mathbf{P}_{k} \in \mathcal{U}_{k} }} 
	\sum_{  \mathbf{i}\in \mathcal{J}^2}  
	\varphi_{k-1}\left( \mathbf{P}_{k},\mathbf i \right)
	\varphi_2\left(\mathbf i_{1:4} \right)
	\varphi_2\left(\mathbf i_{5:8} \right),
	\\
\Lambda_{n,2} 
&= 
\frac{16}{n^4\delta^4_n(k) }\cdot
	\sum\limits_{\substack{\mathbf{P}_{k} \in \mathcal{W}_{k} }} 
	\sum_{  \mathbf{i}\in \mathcal{J}^2}  
	\varphi_{k-1}\left( \mathbf{P}_{k},\mathbf i \right)
	\varphi_2\left(\mathbf i_{1:4} \right)
	\varphi_2\left(\mathbf i_{5:8} \right),
\end{align*}
where 
\begin{align}
\mathcal{U}_{k} 
&= \label{eq:Uk}
\big\{ \mathbf{P}_k  \in \mathcal{P}\left(d,k \right)^4 ~\text{satisfying (\ref{pairing}) and}~p_{k,1}=p_{k,2}=p_{k,3}=p_{k,4}  \big\}, \\
\mathcal{W}_{k} 
&= \nonumber 
\big\{ \mathbf{P}_k  \in \mathcal{P}\left(d,k \right)^4 ~\text{satisfying (\ref{pairing}) and}~p_{k,1}=p_{k,2}<p_{k,3}=p_{k,4}  \big\}.
\end{align}

Next, 
decomposing $\mathcal J^2 = \sqcup_{h,h'=2}^4 \mathcal{I}_h \times \mathcal I_{h'}$, we obtain the decomposition
\[
\Lambda_{n,1}  = \sum_{h,h' \in \{2,3,4\}} \Lambda^{(h:h')}_{n,1} , \qquad 
\Lambda_{n,2}  = \sum_{h,h' \in \{2,3,4\}} \Lambda^{(h:h')}_{n,2},
\]
where
\begin{align*}
\Lambda^{(h:h')}_{n,1}  
&= 
\frac{16}{n^4\delta^4_n(k) }\cdot
	\sum\limits_{\substack{\mathbf{P}_{k} \in \mathcal{U}_{k} }} 
	\sum_{  \mathbf{i}\in \mathcal I_h \times \mathcal I_{h'}}  
	\varphi_{k-1}\left( \mathbf{P}_{k},\mathbf i \right)
	\varphi_2\left(\mathbf i_{1:4} \right)
	\varphi_2\left(\mathbf i_{5:8} \right), \\
\Lambda^{(h:h')}_{n,2}  
&= 
\frac{16}{n^4\delta^4_n(k) }\cdot
	\sum\limits_{\substack{\mathbf{P}_{k} \in \mathcal{W}_{k} }} 
	\sum_{  \mathbf{i}\in \mathcal I_h \times \mathcal I_{h'}}  
	\varphi_{k-1}\left( \mathbf{P}_{k},\mathbf i \right)
	\varphi_2\left(\mathbf i_{1:4} \right)
	\varphi_2\left(\mathbf i_{5:8} \right),
\end{align*}
The assertion in \eqref{eq:lconv} is shown once we prove that
\begin{align} \label{eq:lconv2}
\forall \, (h,h',i)\ne(2,2,2): \lim_{n\to\infty} \Lambda^{(h:h')}_{n,i} = 0, \qquad
\lim_{n\to\infty} \Lambda^{(2:2)}_{n,2} = \frac12.
\end{align}

Write $c_{k-1}(\mathbf P_k)$ for the cardinality of $\{p_{\ell,i}:\ell=1, \dots, k-1; i=1, \dots, 4\}$ and note that $c_{k-1}(\mathbf P_k)\le 2k$ for all $\mathbf P_k \in \mathcal U_k, \mathcal W_k$. For $\lambda \in \{0, \dots, k-1\}$, let
\begin{align}
\mathcal{U}_{k}(\lambda)
&=  \label{eq:Ukl}
\big\{  \mathbf{P}_k \in  \mathcal{U}_k : c_{k-1}(\mathbf P_k) = 2k-2-\lambda \big\}, \\
\mathcal{W}_{k}(\lambda)
&=  \label{eq:Wkl}
\big\{  \mathbf{P}_k \in  \mathcal{W}_k : c_{k-1}(\mathbf P_k) = 2k-2-\lambda \big\} ,
\end{align}
and note that $\mathcal{W}_k=\sqcup_{\lambda = 0}^{k-1}\, \mathcal{W}_{k}(\lambda)$ and $\mathcal{U}_k=\sqcup_{\lambda = 0}^{k-1}\, \mathcal{U}_{k}(\lambda)$, with $\mathcal{W}_{k}(0)$ and $\mathcal{U}_{k}(0)$ corresponding to the case of perfect 2-matching where each $p_{\ell,i}$ is matched with a unique $p_{\ell',j}$, for $\ell, \ell' \le k-1$ and $i,j \in \{1, \dots, 4\}$. 

We start by proving the second assertion in \eqref{eq:lconv2}. Observing that $\mathbf i \in \mathcal{I}_2 \times \mathcal I_2$ implies that
$
\varphi\left(\mathbf i_{1:4} \right) = \varphi\left(\mathbf i_{5:8} \right)  = \frac{1}{90}(1+o(1))
$
by Lemma~\ref{lem:2m},
we may write
 \begin{align} \label{eq:lam22}
 \Lambda^{(2:2)}_{n,2} &=\Lambda^{(2:2)}_{n,2,0}+\sum_{\lambda=1}^{k-1}\Lambda^{(2:2)}_{n,2,\lambda},
\end{align}
where
\begin{align*}
\Lambda^{(2:2)}_{n,2,\lambda} 
&=
\frac{16\cdot \left(1+o(1) \right) }{90^2n^4\delta^4_n(k) }\cdot
\sum\limits_{\mathbf{P}_{k} \in \mathcal{W}_{k}(\lambda)}
\sum\limits_{  \mathbf{i} \in \mathcal{I}_2 \times \mathcal I_2} \varphi_{k-1}\left( \mathbf{P}_{k},\mathbf{i}\right) .
\end{align*}
Now, by Lemma \ref{inj2},
$
| \mathcal{W}_{k}(\lambda)|  
\le 
256^k \cdot \binom{d}{2k-\lambda},
$
and the upper bound may further be bounded by
$
256^k \cdot \binom{d}{2k-1}
$
for sufficiently large $n$.  Hence, by (\ref{cardinal}) and since $|\varphi_{k-1}(\cdot, \cdot)| \le 1$, there exists some numerical constant $c>0$ independent of $n$ and $k$ such that: 
\begin{align}
\label{eq:lam22l}
\zp[b]{~\sum_{\lambda=1}^{k-1}\Lambda^{(2:2)}_{n,2,\lambda}~}
&\leq 
c\cdot 256^k\cdot (k-1)\cdot \binom{d}{2k-1} \cdot \binom{d}{k}^{-2}\xrightarrow[n\to +\infty]{}0.
\end{align}

Next, consider $\ds{\Lambda^{(2:2)}_{n,2,0}}$, which is based on $\mathbf {P}_k \in \mathcal{W}_{k}(0)$ with perfect 2-matching. 
Write $\mathbf p_{k-1,i}=(p_{1,i}, \dots, p_{k-1,i})$ and let
\begin{align} \label{eq:msk}
\mathcal{S}_k
&\coloneqq 
\left\{ \mathbf{P}_k   \in \mathcal{W}_{k}(0): \mathbf{p}_{k-1,1}=\mathbf{p}_{k-1,2},~\mathbf{p}_{k-1,3}=\mathbf{p}_{k-1,4}
\right\},
\\
&= \nonumber
\big\{ \mathbf{P}_k   \in \mathcal{W}_{k}(0): (\mathbf p_{k-1,1} \cup \mathbf p_{k-1,2}) \cap (\mathbf p_{k-1,3} \cup \mathbf p_{k-1,4})=\varnothing \big\}.
\end{align}
We may then decompose
\begin{align} \label{eq:l22n2}
\Lambda^{(2:2)}_{n,2,0} 
&=
\Lambda^{(2:2)}_{n,2,0}(\mathcal{S}_{k}) 
+ 
\Lambda^{(2:2)}_{n,2,0}(\mathcal{W}_{k}(0)\setminus \mathcal{S}_k),
\end{align}
where, for any set $\ds{\mathcal{A}} \subset \mathcal{W}_{k}(0)$,
\begin{align*}
\Lambda^{(2:2)}_{n,2,0}(\mathcal{A})
&=
\frac{16\cdot \left(1+o(1) \right) }{90^2n^4\delta^4_n(k) }\cdot
\sum\limits_{\mathbf{P}_{k} \in \mathcal{A}}
\sum\limits_{  \mathbf{i}\in \mathcal I_2 \times \mathcal I_2} \varphi_{k-1}\left( \mathbf{P}_{k},\mathbf{i}\right) .
\end{align*}
Now,  $\mathbf{i} \in \mathcal I_2 \times \mathcal I_2$ allows to write $\mathbf{i}=(i_1, i_2, i_1, i_2, i_5, i_6, i_5, i_6)$, while $\mathbf{P}_{k} \in \mathcal S_k$ implies $p_{\ell,1}=p_{\ell,2} \ne p_{\ell,3}=p_{\ell,4}$. Hence, by the definition of $\varphi_{k-1}$ in \eqref{varphi_def}, using $(2k-2)$-wise independence, equidistribution and Lemma \ref{lem:2m}, we obtain  
\begin{align*}
		\Lambda^{(2:2)}_{n,2,0}(\mathcal{S}_{k})
		&=
		\frac{16\cdot \left(1+o(1) \right) }{90^2n^4\delta^4_n(k) }\cdot
		\sum\limits_{\mathbf{P}_{k} \in  \mathcal{S}_{k}} 
		\sum\limits_{  \mathbf{i}_1, \mathbf i_2 \in \mathcal{I}_2}
		\E\left[\prod_{\ell=1}^{k-1} \left( \tilde  I^{(p_{\ell,1})}_{i_1,i_2}\right)^2\left( \tilde I^{(p_{\ell,3})}_{i_5,i_6}\right)^2 \right]\\
		&=
		\frac{16\cdot \left(1+o(1) \right) }{90^2n^4\delta^4_n(k) }\cdot
		\sum\limits_{\mathbf{P}_{k} \in  \mathcal{S}_{k}}
		\sum\limits_{  \mathbf{i}_1, \mathbf i_2 \in \mathcal{I}_2}
		\E\left[ \left( \tilde  I^{(1)}_{1,2}\right)^2\right]^{2(k-1)}
		\\&=\frac{16\cdot \left(1+o(1) \right) }{90^{2k}n^4\delta^4_n(k) }\cdot \zp[b]{\mathcal{S}_{k}} \cdot \zp[b]{\mathcal{I}_2}^2.
\end{align*}
Now $|\mathcal S_k| = \binom{d}{2k} \cdot\binom{2k-1}{k-1}$ by Corollary~\ref{cor:cards} and $|\mathcal I_2|=n(n-1)/2$ by \eqref{cardinal}.  As a consequence, by the definition of $\delta_n(k)$ in \eqref{eq:deltan}, a small computation yields
\begin{align} \label{eq:lam2201}
	\Lambda^{(2:2)}_{n,2,0}(\mathcal{S}_{k})
	&=
	\left(1+o(1) \right) \cdot \binom{d}{k}^{-2}\cdot \binom{d}{2k} \cdot \binom{2k-1}{k-1} \xrightarrow[n\to +\infty]{}\ff{2}.
\end{align}

Next, consider the second summand on the right-hand side of \eqref{eq:l22n2}. For $\mathbf{P}_k\in\mathcal{W}_{k}(0)$, 
recall the notation $\mathbf p_{k-1,i}=(p_{1,i}, \dots, p_{k-1,i})$ and let
\begin{align*} 
\Delta\left( \mathbf{P}_k\right)  = 
(\mathbf p_{k-1,1} \cup \mathbf p_{k-1,2}) \cap (\mathbf p_{k-1,3} \cup \mathbf p_{k-1,4})
\end{align*}
which is non-empty for $\ds{\mathbf{P}_k\in\mathcal{W}_{k}(0) \setminus \mathcal{S}_k}$ with cardinality less or equal to $2k-2$. For such $\mathbf{P}_k$ and $\mathbf{i}=(i_1, i_2, i_1, i_2,i_5, i_6, i_5, i_6) \in \mathcal I_2 \times \mathcal I_2$ we obtain that, using $(4k-4)$-wise independence and Lemma \ref{lem:2m},  
\begin{align*}
	\varphi_{k-1}\left( \mathbf{P}_{k},\mathbf{i}\right) 
	&=
		\E\left[\prod_{\ell=1}^{k-1} \tilde  I^{(p_{\ell,1})}_{i_1,i_2}\tilde  I^{(p_{\ell,2})}_{i_1,i_2}\tilde I^{(p_{\ell,3})}_{i_5,i_6}\tilde I^{(p_{\ell,4})}_{i_5,i_6}\right]  
		=
		\begin{cases}
		90^{-2\left( k-1\right) } + O(n^{-1}) &,  |\mathbf{i}|=2, \\
		O(n^{-|\Delta(\mathbf{P}_k)|}) = O(n^{-1}) &, |\mathbf{i} |=3, \\
		O(n^{-2|\Delta(\mathbf P_k)|}) = O(n^{-2}) &, |\mathbf{i} | = 4.
		\end{cases}
\end{align*}

Hence, since the number of $\mathbf i \in \mathcal I_2 \times \mathcal I_2$ with $|\mathbf{i}|=h$ is equal to $|\mathcal I_h|$ from \eqref{eq:iell} for $h\in\{2,3,4\}$, we obtain that for some constant $c=c_k<\infty$,
\begin{align*}
\zp[b]{\Lambda^{(2:2)}_{n,2,0}(\mathcal{W}_{k}(0)\setminus \mathcal{S}_k)}
&\leq 
\frac{c_k\left(1+o(1) \right)}{n^4\delta^4_n(k) } \cdot 
|\mathcal{W}_{k}(0)\setminus \mathcal{S}_k| \cdot
\left(\zp[b]{\mathcal{I}_2} +\frac{\zp[b]{\mathcal{I}_3}}{n} + \frac{|\mathcal I_4|}{n^2}\right) .
\end{align*}
Next, 
$
|\mathcal{W}_{k}(0)\setminus \mathcal{S}_k|
\le 
|\mathcal{W}_{k}(0)| 
\le
C_k  \binom{d}{2k}
$
by Lemma~\ref{bij2}, whence we obtain, by (\ref{cardinal}),
\begin{align} \label{eq:lam2202}
	\zp[b]{\Lambda^{(2:2)}_{n,2,0}(\mathcal{W}_{k}(0)\setminus \mathcal{S}_k)}
=O(n^{-2})=o(1).
	\end{align}
As a summary, by \eqref{eq:lam22}, \eqref{eq:lam22l}, \eqref{eq:l22n2}, \eqref{eq:lam2201} and \eqref{eq:lam2202}, we obtain that
\[
\lim_{n\to\infty} \Lambda^{(2:2)}_{n,2} = \frac1{2},
\]	
and it remains to show the convergences in \eqref{eq:lconv2} for $(h,h',i)\ne(2,2,2)$.

We start by treating $i=1$.
Let $2\le h\le h'\le 4$. After using Lemma \ref{lem:2m} and recalling that $\mathcal{U}_k=\sqcup_{\lambda = 0}^{k-1}\, \mathcal{U}_{k}(\lambda)$, we may decompose 
\[ 
\Lambda^{(h:h')}_{n,1}= \frac{c\left(1+o(1) \right) }{n^{h+h'}\delta^4_n(k) }\cdot
\sum_{\lambda=0}^{k-1}
\sum\limits_{\substack{\mathbf{P}_{k} \in \mathcal{U}_k(\lambda)}}
\sum_{\mathbf i\in \mathcal{I}_{h} \times \mathcal{I}_{h'}}
 \varphi_{k-1}\left( \mathbf{P}_{k},\mathbf{i} \right),
\]
where $c:=c(h,h')>0$ is a numerical constant independent of $n$ and $k$. In view of the trivial bound $|\varphi_{k-1}(\cdot, \cdot)| \le 1$, the previous expression is bounded by
\[
\frac{c\left(1+o(1) \right) }{n^{h+h'}\delta^4_n(k) }\cdot\sum_{\lambda=0}^{k-1}
| \mathcal{U}_k(\lambda)| \cdot 
| \mathcal{I}_{h} \times \mathcal{I}_{h'}|.
\]
Since $| \mathcal{I}_{h} \times \mathcal{I}_{h'}| = O(n^{h+h'})$ by \eqref{cardinal}, we may invoke Lemma \ref{bij2} and Lemma \ref{inj2} to control the leading order in $d$ (which is obtained for $\lambda=0$) to obtain that,  
$$ 
\Lambda^{(h:h')}_{n,1}
= O(d^{-1})=o(1)
$$ 
for all $2\le h\le h'\le 4$.
This proves \eqref{eq:lconv2} for $i=1$, and it remains to consider $i=2$ and $(h,h')\ne (2,2)$.

For that purpose, similar as in \eqref{eq:lam22}, we may use $\mathcal{W}_k=\sqcup_{\lambda = 0}^{k-1}\, \mathcal{W}_{k}(\lambda)$ and Lemma~\ref{lem:2m} to write
$$ 
\Lambda^{(h:h')}_{n,2}=\Lambda^{(h:h')}_{n,2,0}+ \sum_{\lambda=1}^{k-1}\Lambda^{(h:h')}_{n,2,\lambda},
$$ 
where, for $0\le \lambda\le k-1$, 
$$
\Lambda^{(h:h')}_{n,2,\lambda}
=
\frac{c\left(1+o(1) \right) }{n^{h+h'}\delta^4_n(k) }\cdot\sum\limits_{\substack{\mathbf{P}_{k} \in \mathcal{W}_k(\lambda)}}
\sum\limits_{\substack{\mathbf i \in \mathcal{I}_{h} \times \mathcal{I}_{h'}}}
\varphi_{k-1}\left( \mathbf{P}_{k},\mathbf{i} \right) 
$$
and where $c:=c(h,h')>0$ is independent of $n$ and $k$. For $\lambda>0$, we may use the trivial bound $|\varphi_{k-1}(\cdot, \cdot)| \le 1$ and Lemma \ref{inj2} to obtain that $\Lambda^{(h:h')}_{n,2,\lambda} = O(d^{-\lambda})=o(1)$, whence
 \begin{align} \label{eq:lhhn2}
 \Lambda^{(h:h')}_{n,2}=\Lambda^{(h:h')}_{n,2,0} + o(1) = \Lambda^{(h:h')}_{n,2,0}(\mathcal S_k) +\Lambda^{(h:h')}_{n,2,0}(\mathcal{W}_k(0)\setminus \mathcal{S}_k)  + o(1)
 \end{align}
with $\mathcal S_k$ from \eqref{eq:msk}, where, for $\mathcal A \subset \mathcal W_k(0)$,
$$
\Lambda^{(h:h')}_{n,2,0}(\mathcal A)=
\frac{c\left(1+o(1) \right) }{n^{h+h'}\delta^4_n(k) }\cdot\sum\limits_{\substack{\mathbf{P}_{k} \in\mathcal A}}
\sum\limits_{\substack{\mathbf i \in \mathcal{I}_{h} \times \mathcal{I}_{h'}}}
\varphi_{k-1}\left( \mathbf{P}_{k},\mathbf{i} \right).
$$

For each $\mathbf{P}_k\in \mathcal{S}_k$, by $(4k-4)$-wise independence, equidistribution and Lemma \ref{lem:2m}, we have uniformly in $\mathbf{P}_{k} \in  \mathcal{S}_k$ and $\mathbf i \in \mathcal{I}_{h} \times \mathcal{I}_{h'} $,
$$
\varphi_{k-1}\left( \mathbf{P}_{k},\mathbf{i} \right)  = \E\left[ \tilde  I^{(1)}_{i_1,i_2}\tilde  I^{(1)}_{i_3,i_4}\right]^{k-1}\cdot  \E\left[\tilde I^{(1)}_{i_5,i_6}\tilde I^{(1)}_{i_7,i_8} \right]^{k-1} = O\left( n^{2\left( 4-h-h'\right) \left( k-1\right) } \right) =o(1).
$$
Hence, using $ |\mathcal{I}_{h}\times \mathcal{I}_{h'}|=O(n^{h+h'})$ from (\ref{cardinal}) and $|\mathcal S_k| =\binom{d}{2k} \cdot\binom{2k-1}{k-1}=O(d^{2k})$ from Lemma~\ref{bij}, it follows that $\Lambda^{(h:h')}_{n,2,0}(\mathcal S_k)=o(1)$. 

On the other hand, for $\mathbf P_k \notin \mathcal \mathcal W_k(0)\setminus S_k$, we have $|\Delta\left( \mathbf{P}_k\right)|\ge 1 $. In fact, by the pigeonhole principle, $\zp[b]{\Delta\left( \mathbf{P_k}\right)} \ge 2$. Then, there exists $1\le \ell_1,\ell_2,\ell_3,\ell_4\le k-1$ such that $p_{\ell_1,1} =p_{\ell_3,3} \neq p_{\ell_2,2}=p_{\ell_4,4}$ and therefore, by $(4k-4)$-wise independence,
 \begin{align*}
 \varphi_{k-1}\left( \mathbf{P}_{k},\mathbf{i} \right)
 & =   
 \E\Bigg[ \frac{\prod\limits_{\substack{\ell=1 \\ }}^{k-1}\tilde  I^{(p_{\ell,1})}_{i_1,i_2}\tilde  I^{(p_{\ell,2})}_{i_3,i_4}\tilde I^{(p_{\ell,3})}_{i_5,i_6}\tilde I^{(p_{\ell,4})}_{i_7,i_8}}{\tilde  I^{(p_{\ell_1,1})}_{i_1,i_2}\tilde  I^{(p_{\ell_2,2})}_{i_3,i_4}\tilde I^{(p_{\ell_3,3})}_{i_5,i_6}\tilde I^{(p_{\ell_4,4})}_{i_7,i_8}}\Bigg]
 \cdot \E\left[ \tilde  I^{(1)}_{i_1,i_2}\tilde  I^{(1)}_{i_5,i_6}\right]\cdot \E\left[\tilde I^{(1)}_{i_3,i_4}\tilde I^{(1)}_{i_7,i_8}\right],
 \end{align*}
 where we use the convention that $\frac00=1$. Note that the first expectation on the right-hand side is bounded by 1 and that at least one of the other two expectations is $O(n^{-1})$ (uniformly in $\mathbf i $ and $\mathbf P_k$) by Lemma~\ref{lem:2m}, while the other is bounded by 1. Hence,  using 
 $
|\mathcal{W}_{k}(0)\setminus \mathcal{S}_k|
\le 
|\mathcal{W}_{k}(0)| 
\le
C_k  \binom{d}{2k}
$
by Lemma~\ref{bij2} and $|\mathcal I_h \times \mathcal I_{h'}|=O(n^{h+h'})$ by \eqref{cardinal},
\[
\Lambda^{(h:h')}_{n,2,0}(\mathcal{W}_{k}(0) \setminus \mathcal S_k) 
= 
\frac{c\left(1+o(1) \right) }{ n^{h+h'}\delta^4_n(k) }\cdot |\mathcal{W}_{k}(0) \setminus \mathcal S_k| \cdot 
|\mathcal{I}_h \times \mathcal{I}_{h'}|
\cdot O(n^{-1}) 
=O(n^{-1})
\]	
As a summary, $\Lambda_{n,2,0}^{(h:h')}=o(1)$. Hence, by \eqref{eq:lhhn2}, we obtain the convergences on the left-hand side of \eqref{eq:lconv2}  for $i=2$ and the proof is finished.
\end{proof}

 \begin{lemma} 
 \label{lyapunovlem}
 Let $k \in \N_{\ge 2}$ and assume $(4k-3)$-wise independence. Then,
 recalling $X_{n,r}=X_{n,r}(k)$ from \eqref{eq:xnr},
$$ \Theta_{n}:=
\E \Big| \sum_{p_k=1}^{d}\E\left( X^4_{n,p_k} | \mathcal{F}_{n,p_k-1}  \right)  \Big| 
=
\sum_{p_k=1}^{d}\E\left( X^4_{n,p_k} \right)  \xrightarrow[n\to +\infty]{} 0 .$$
 \end{lemma}

 \begin{proof}
Recall the notation $\mathbf P_k=(\mathbf p_{k,1}, \dots, p_{k,4}) \in \mathcal P(d,k)^4$ with
	$\mathbf{p}_{k,j}=(p_{1,j}, \dots, p_{k,j}) \in \mathcal P(d,k)$ for $j\in\{1,\dots, 4\}$.
By \eqref{eq:product2}, we have, for $p_k\ge k$ 
\begin{align*} 
	X_{n,p_k}^4 =  
	\frac{16}{n^4\delta_n^4(k)} 
	\sum_{\substack{\mathbf{P}_{k} \in \mathcal P(d,k)^4: \\  p_{k,j}=p_k \forall j=1, \dots, 4}}  
	\sum_{  \mathbf{i}\in \mathcal{J}^2} 
	\Big(\prod_{\ell=1}^{k-1}  
	\tilde I^{(p_{\ell,1})}_{i_1,i_2}\tilde I^{(p_{\ell,2})}_{i_3,i_4}  \tilde I^{(p_{\ell,3})}_{i_5,i_6}\tilde I^{(p_{\ell,4})}_{i_7,i_8} \Big) \cdot \tilde I^{(p_{k})}_{i_1,i_2}\tilde I^{(p_{k})}_{i_3,i_4} \tilde I^{(p_{k})}_{i_5,i_6}\tilde I^{(p_{k})}_{i_7,i_8} .
\end{align*}
As a consequence, recalling $\varphi_m(\mathbf P_k, \mathbf i)$ from \eqref{varphi_def} and $\varphi_m(\mathbf i)$ from \eqref{eq:plainphi}, we may write,  by $(4k-3)$-wise independence,
\begin{align*} 
\Theta_n
&=
\frac{16}{n^4\delta^4_n(k) }
\sum_{\substack{\mathbf{P}_{k} \in \mathcal P(d,k)^4: \\  p_{k,1}=p_{k,2}=p_{k,3}=p_{k,4}}}  
\sum_{  \mathbf{i}\in \mathcal{J}^2}   
\varphi_{k-1}\left( \mathbf{P}_{k},\mathbf{i} \right) \cdot  \varphi_{4}\left( \mathbf{i} \right).
\end{align*}
Similar as in the proof of Lemma~\ref{lem_cond_exp} (treatment of $\Lambda_{n,1}$) we have, recalling $\mathcal U_k$ from \eqref{eq:Uk},
\begin{align*} 
\Theta_n
&=
\frac{16}{n^4\delta^4_n(k) }
\sum_{\mathbf{P}_{k} \in \mathcal{U}_k}
\sum_{  \mathbf{i}\in \mathcal{J}^2}   
\varphi_{k-1}\left( \mathbf{P}_{k},\mathbf{i} \right) \cdot  \varphi_{4}\left( \mathbf{i} \right) \\
&=
\frac{16}{n^4\delta^4_n(k) } 
\sum_{\mu=2}^8
\sum_{\mathbf{P}_{k} \in \mathcal{U}_k}
\sum_{\substack{\mathbf{i}\in \mathcal{J}^2\\ |\mathbf i| = \mu}  } 
\varphi_{k-1}\left( \mathbf{P}_{k},\mathbf{i} \right) \cdot  \varphi_{4}\left( \mathbf{i} \right):=\sum_{\mu=2}^8\Theta_{n}(\mu).
\end{align*}
For any $\ds{\mathbf{i}=\left( i_1,\ldots,i_8\right)  \in \mathcal{J}^2}$ with $|\mathbf i|=\mu\in \{5,\ldots,8\}$, one has $\varphi_4(\mathbf i)=O(n^{4-\mu})$ by Lemma~\ref{moment_goal}, uniformly in $\mathbf i \in \mathcal J^2$ such that $5\le |\mathbf i|\le 8$. 
Moreover, since $\mathcal{U}_k=\sqcup_{\lambda=0}^{k-1}\mathcal{U}_k(\lambda)$, we have $|\mathcal U_k| = O(\binom{d}{2k-1})=O(d^{2k-1})$ by Lemma~\ref{inj2}.
In view of the fact that $|\{ \mathbf{i}\in \mathcal{J}^2: |\mathbf i| = \mu\}| = O(n^{\mu})$, we obtain $\Theta_n(\mu) = O(d^{-1}) $ for $\mu\in\{5, \dots, 8\}$. The remaining cases $\mu \in\{2,\ldots,4\}$ are treated in a simpler way, exploiting that $|\varphi_{k-1}(\cdot, \cdot)|, |\varphi_4(\cdot)| \le 1.$ 
 \end{proof}

\subsection{Auxiliary Combinatorial Results} \label{sec:aux}

Recall $\mathcal P(d,k)$ from \eqref{eq:pdk}, with $2 \le k \le d$. For $\ell \in \{0, \dots, k-1\}$, we define 
\[
\mathcal{V}_{k}(\ell)
:=
\mathcal{V}_{k,d}(\ell)
:= 
\big\{ \left(\mathbf{p}_{k,1},\mathbf{p}_{k,2} \right) \in \mathcal{P}\left(d,k \right)^2,  p_{k,1}<p_{k,2}  ,  \zp[b]{\mathbf{p}_{k-1,1}\cap \mathbf{p}_{k-1,2}}=\ell  \big\},
\] 
where $\mathbf p_{k,j}=(p_{1,j}, \dots, p_{k,j})$ and $\mathbf p_{k-1,j} = (p_{1,j}, \dots, p_{k-1,j})$.

\begin{lemma}\label{bij} 
Suppose $k\le \frac{d}{2}$. For any $\ell \in \{0, \dots, k-1\}$, there is one-to-one mapping between
$\mathcal{V}_{k}(\ell)$ and 
$\mathcal{P}\left(d,2k -\ell \right) \times 
\mathcal{P}(2k-\ell-1,k-1) \times
\mathcal{P}\left( k-1,\ell \right)$. As a consequence,
\begin{align*} 
|\mathcal{V}_{k}(\ell) | = \binom{d}{2k-\ell}\cdot \binom{2k-\ell-1}{k-1} \cdot \binom{k-1}{\ell}.
\end{align*}
\end{lemma}

\begin{proof}
Let $\ds{ \mathbf{P}=\left(\mathbf{p}_{k,1},\mathbf{p}_{k,2} \right) \in\mathcal{V}_{k}(\ell)}$. Define
\begin{compactitem} 
\item $\ds{\mathbf{q}=\mathbf{p}_{k,1}\dot\cup \mathbf{p}_{k,2} \in \mathcal{P}\left( d,2k-\ell\right) }$ as the unique ordered vector obtained after ordering the set-union $\ds{ \mathbf{p}_{k,1}\cup \mathbf{p}_{k,2}  }$;
\item 		
$\mathbf{r}=\left( r_1,\ldots,r_{k-1}\right) \in \mathcal{P}(2k-\ell-1,k-1)$ by 
$
r_j = \sum_{s=1}^{2k-\ell-1}1_{\{ q_{s} \leq p_{j,2}   \}}$, $j \in \{1, \dots, k-1\},
$
\item $\tilde{\mathbf{q}} = \mathbf{p}_{k,1}\cap \mathbf{p}_{k,2}=  \mathbf{p}_{k-1,1}\cap \mathbf{p}_{k-1,2} \in \mathcal P(d,\ell)$ the unique vector obtained after ordering the set-intersection $\ds{ \mathbf{p}_{k,1}\cap \mathbf{p}_{k,2}  }$;
\item $\ds{\mathbf{i} = \left(i_1,\ldots,i_{\ell} \right) \in \mathcal{P}(k-1,\ell) }$ by
$
i_j = \sum_{s=1}^{k-1}1_{\{ q_{r_s} \le \tilde{q}_{j}   \}}$, $j \in \{1, \dots, \ell\}.
$
\end{compactitem}
	It is instructive to illustrate with a particular example  when $d \ge 27$, $\ds{k=5}$ and $\ds{\ell=3}$:
	{\small $$\mathbf p_{5,1} = \begin{pmatrix}
		5\\ 7\\ 12 \\ 17 \\ 22
	\end{pmatrix} , 
	\mathbf p_{5,2} =\begin{pmatrix}
		5\\ 9 \\ 12\\ 22 \\ 27
	\end{pmatrix} 
	\quad \longrightarrow  \quad
	\mathbf q =\begin{pmatrix}
		5\\ 7\\ 9\\ 12 \\ 17 \\ 22 \\ 27
	\end{pmatrix}  ,
	 \mathbf{r}= \begin{pmatrix}
		1\\ 3\\4 \\6
	\end{pmatrix}, 
	\tilde{\mathbf{q}}= \begin{pmatrix}
		5\\ 12\\22
	\end{pmatrix}, 
	\mathbf{i}=\begin{pmatrix}
		1\\ 3\\4
	\end{pmatrix}.  $$}
The example indicates that $\mathbf p_{k,1}$ and $\mathbf p_{k,2}$ can be uniquely recovered from $\mathbf q, \mathbf r, \mathbf i$.  Indeed, $\mathbf r$ encodes the coordinates of $\bm q$ that belong to $\mathbf p_{k,2}$ (except for the last one, which is part of $\mathbf p_{k,2}$ anyways by the requirement $p_{k,1}<p_{k,2}$), while $(r_{i_1}, r_{i_2}, r_{i,3})=(1,4,6)$ encodes the coordinates of $\bm q$ that belong to both $\mathbf p_{k,1}$ and $\mathbf p_{k,2}$. The vector $\mathbf p_{k,1}$ can then be obtained by filling up with the remaining arguments from $\mathbf q$ that do not belong to $\mathbf p_{k,2}$.

More formally, the map 
\begin{align*}
\Phi \colon \mathcal{V}_{k}(\ell) 
& \longrightarrow 
\mathcal{P}\left(d,2k -\ell \right) \times \mathcal{P}\left( 2k-\ell-1,k-1 \right)\times\mathcal{P}\left( k-1,\ell \right) , \quad \mathbf{P} \longmapsto \left(   \mathbf q, \mathbf r, \mathbf i\right) 
\end{align*}
is bijective with inverse as follows:
for given $(\mathbf q, \mathbf r, \mathbf i)$ in the image of $\Phi$, define $\mathbf p_{k,2}$ by
\[
p_{k,2} = q_{2k-\ell}, 
\quad
p_{j,2} = q_{r_j} \quad (j=1, \dots, k-1)
\]
and let $\mathbf p_{k,1}$ be the vector obtained by ordering the set
\[
\{ q_{r_{i_j}}: j \in \{1, \dots, \ell\} \} 
\cup
\{ q_j : j \in \{1, \dots,  2k-\ell-1\} \setminus \mathbf r\}. \qedhere
\]
\end{proof}

Recall the set $\mathcal S_k$ from \eqref{eq:msk}, which may be written as
\begin{multline*}
	\mathcal S_k 
	=
	\{(\mathbf p_{k,1}, \dots, \mathbf p_{k,4}) \in \mathcal P(d,k)^4: 
	\mathbf p_{k,1} = \mathbf p_{k,2}, \mathbf p_{k,3}=\mathbf p_{k,4}, p_{k,1} < p_{k,3}, \\
	| \mathbf p_{k-1,1} \cup  \dots \cup \mathbf p_{k-1,4}| = 2k-2\}. 
\end{multline*}
Clearly, $\mathcal S_k $ is in bijection with $\mathcal{V}_{k}(0)$, which immediately yields the following corollary.

\begin{corollary} \label{cor:cards}
	For any integer $k \in \{2, \dots, \lfloor d/2 \rfloor\}$, we have
	\begin{align}\label{eq:mskc}
		|\mathcal{S}_k| &= \binom{d}{2k}\cdot \binom{2k-1}{k-1}.
	\end{align}
\end{corollary}

Throughout the proofs in Section~\ref{sec:ps3a} and \ref{sec:ps3b}, we need (bounds on) the cardinality of $\mathcal U_k(\ell)$ and $\mathcal W_k(\ell)$ from \eqref{eq:Ukl} and \eqref{eq:Wkl}, respectively.  For their derivation, some additional notation is needed.
For a matrix $M=\left( M_{i,j}\right)_{1\le i \le p,1\le j\le q}$, we write $M_{x:y,z:t} = \left( M_{i,j} \right)_{x\le i \le y,z\le j \le t}$. Let $\ds{\mathcal{N}_k}$ denote the set of matrices with entries in $\{0,1\}$ (subsequently called $\{0,1\}$-matrices) with size $2k\times 4$ such that each row sums to $\ds{2}$ and each column sums to $\ds{k}$. Let $\mathcal{ M}_k\subset \mathcal{N}_k $ the set of $\{0,1\}$-matrices with size $2k \times 4$ such that each row sums to $2$, each column sums to $\ds{k}$ and the last row is of the form $(0,0,1,1)$, that is, $ \mathcal{ M}_k= \left\{ M\in \mathcal{N}_k:~ M_{2k,1:4} = (0,0,1,1)   \right\}$.	Let also $\ds{\mathcal{\tilde M}_k }$ the set of $\{0,1\}$-matrices with size $\ds{(2k-1) \times 4}$ such that each row except the last one sums to $2$, each column sums to $\ds{k}$ and the last row has only $1$ components, that is 
\[
 \mathcal{\tilde M}_k= \big\{ \{0,1\}^{(2k-1) \times 4}:~M_{1:2k-2,1:4} \in \mathcal{N}_{k-1},~ M_{2k-1,1:4} = (1,1,1,1)   \big\}.
 \]

Clearly, $|\mathcal M_k|, | \tilde {\mathcal	M_k}| \le |\mathcal N_k| =: C_k<\infty$. Note that the order of $|\mathcal N_k|$ for $k\to\infty$ has been derived in \cite[Theorem~1.c]{CanfieldMckay}. More precisely, for some universal constant 
$\ds{c\in (0,+\infty)}$, 
\[
\zp[b]{\mathcal{N}_{k}} =c\left(1+o(1) \right) \cdot 6^{2k}\cdot \binom{2k}{k}^4 \cdot \binom{8k}{4k}^{-1}, \quad k \to\infty.
\]
This expansion may be used to extend the `fixed $k$' results from Proposition~\ref{prop_s1} to the case `$k=k_n\to\infty$'. However, in view of the fact that the respective test statistics would be computationally infeasible, we dispense with any details.

\begin{lemma}\label{bij2} 
Suppose $k\le \frac{d+1}{2}$. 
There exists a one-to-one mapping between $\mathcal{W}_k(0)$ from \eqref{eq:Wkl} and $ \mathcal{P}\left(d,2k \right)\times \mathcal{M}_k$ and between $\mathcal{U}_k(0)$ from \eqref{eq:Ukl} and $\ds{ \mathcal{P}\left(d,2k-1 \right)\times \mathcal{\tilde M}_k  }$. As a consequence, with $C_k=|\mathcal N_k|$ depending only on $k$,
\begin{align*} 
\zp[b]{\mathcal{W}_k(0)} &\le C_k \binom{d}{2k}  ,
\qquad 	
\zp[b]{\mathcal{U}_k(0)} \le C_k \binom{d}{2k-1},		
\end{align*}
\end{lemma}

\begin{proof}
We only prove the  assertion regarding $\mathcal W_k(0)$, as the one regarding $\mathcal U_k(0)$ can be treated similarly.  Note that $\mathcal W_{k}(0)$ may be written as
\begin{multline*}
\mathcal W_k(0) 
= 
\big\{
	\mathbf P_k=(\mathbf{p}_{k,1},\mathbf{p}_{k,2},\mathbf{p}_{k,3},\mathbf{p}_{k,4}) \in \mathcal P(d,k)^4: 
	p_{k,1}=p_{k,2}<p_{k,3}=p_{k,4}, \\
	\forall~ (\ell,i) \in\{1, \dots, k-1\} \times \{1,\ldots,4\} ~ \exists! ~(\ell',j)\neq (\ell,i)~ \text{such that } p_{\ell,i}=p_{\ell',j}.
\big\},
\end{multline*}
where `$\exists!$' means `there exists a unique'.
Note that $|\mathbf p_{k,1} \cup \dots \cup \mathbf p_{k,4}| = 2k$ for $\mathbf{P}_k\in \mathcal W_k(0)$.

For $\mathbf{P}_k \in \mathcal{W}_k(0)$ let $\mathbf{q}$ denote the unique vector in $\mathcal P(d,2k)$ obtained by ordering the set union $\mathbf{p}_{k,1} \cup \dots \cup \mathbf p_{k,4}$. Next, define a $\{0,1\}$-matrix  of dimension $(2k \times 4)$ by
\[
\mathbf{R}
=
\big(1_{\{ q_\ell \in \mathbf{p}_{k,j}  \}} \big)_{\ell=1, \dots, 2k, j=1, \dots, 4}.
\]
Observe that each row of $\ds{\mathbf{R}}$ sums to $\ds{2}$ and that each column sums to $\ds{k}$ and that the last row of $\mathbf R$ is equal to $(0,0,1,1)$. In other words,  $\mathbf R \in \mathcal M_k$. It is instructive to provide an example, with $d\ge 37$ and $k=3$:
{\small \[
\mathbf{P} =\left[ \begin{pmatrix}
			10\\ 15\\ 33
		\end{pmatrix}  ,\begin{pmatrix}
			5\\ 10\\ 33
		\end{pmatrix}  ,\begin{pmatrix}
			15\\ 20\\ 37
		\end{pmatrix} ,\begin{pmatrix}
			5\\ 20\\ 37
		\end{pmatrix} \right] \quad \longrightarrow  \quad
		\mathbf{q}=\begin{pmatrix}
			5\\ 10\\ 15\\ 20\\ 33 \\37
		\end{pmatrix}, \quad \mathbf{R} = \begin{pmatrix}
			0&1&0&1 \\ 1&1&0&0 \\ 1&0&1&0 \\ 0&0&1&1 \\1&1&0&0\\ 0&0&1&1
		\end{pmatrix}.  
\]}
Clearly, $\mathbf p_{k,j}$ can be reconstructed from $\mathbf q$ and the $j$th column of $\mathbf R$.
		
Formally, the following mapping is a bijection
\begin{align}
\label{eq:phibij}
\Phi \colon \mathcal{W}_k(0)   & \longrightarrow  \mathcal{P}\left(d,2k \right)\times \mathcal{M}_k, \quad \mathbf{P} \longmapsto \left(\textbf{q},\textbf{R}\right). 
\end{align}
We explicitly construct the inverse: let $\ds{\textbf{q}\in \mathcal{P}(d,2k)}$ and $\ds{\textbf{R}=\left(r_{\ell,j} \right) \in \mathcal{M}_k}$. Then, for any $\ell \in \{1, \dots, k\}$ and $j \in \{1, \dots, 4\}$, let
\[
p_{\ell,j} 
= 
q_{\psi_j(\ell) }, \quad 
\psi_j(\ell)
= 
\min \Big\{ m \in \{ 1,\dots, 2k\}: \sum_{\ell'=1}^{m} r_{\ell',j} = \ell  \Big\}. \qedhere
\]
\end{proof}

\begin{lemma}\label{inj2}
For $\ell \in \{0, \dots, k-1\}$, there is an injection from $\mathcal{W}_k(\ell)$ into $\mathcal{P}(d,2k-\ell) \times \{ 0,1 \}^{8k} $ and from $\mathcal{U}_k(\ell)$ into $\mathcal{P}(d,2k-1-\ell) \times \{ 0,1 \}^{8k} $. 
As a consequence,				
\begin{align*} 
\zp[b]{\mathcal{W}_k(\ell)} &\le 256^k\cdot \binom{d}{2k-\ell},
\qquad 	
\zp[b]{\mathcal{U}_k(\ell)} \le 256^k\cdot \binom{d}{2k-1-\ell}.
\end{align*}
\end{lemma}

\begin{proof} In view of the definition of $\mathcal W_k(\ell)$ in \eqref{eq:Wkl}, we may write
\begin{multline*}
\mathcal W_k(\ell) 
= 
\big\{
	\mathbf P_k=(\mathbf{p}_{k,1},\mathbf{p}_{k,2},\mathbf{p}_{k,3},\mathbf{p}_{k,4}) \in \mathcal P(d,k)^4: 
	p_{k,1}=p_{k,2}<p_{k,3}=p_{k,4}, \\
	\forall~ (\ell,i) \in\{1, \dots, k-1\} \times \{1,\ldots,4\} ~ \exists ~(\ell',j)\neq (\ell,i)~ \text{such that } p_{\ell,i}=p_{\ell',j}, \\
	| \mathbf p_{k,1} \cup \dots \cup \mathbf p_{k,4}| = 2k-\ell
\big\},
\end{multline*}
The map $\Phi$ from \eqref{eq:phibij}, viewed as a map from $\ds{\mathcal{W}_k(\ell)}$ to $\ds{\mathcal{P}(d,2k-\ell) \times \{ 0,1 \}^{8k} }$ is clearly an injection. A similar argument can be used for $\mathcal U_k(\ell)$.
\end{proof}

\section*{Acknowledgements.} This work has been supported by the Collaborative Research Center ``Statistical modeling of nonlinear dynamic processes'' (SFB 823) of the German Research Foundation, which is gratefully acknowledged.  Most parts of this paper were written when C. Pakzad was a postdoctoral researcher at Heinrich Heine University  Düsseldorf. Computational infrastructure and support were provided by the Centre for Information and Media Technology at Heinrich Heine University Düsseldorf, which is gratefully acknowledged.

\bibliographystyle{chicago}
\bibliography{biblio}

\newpage

\thispagestyle{empty}

\begin{center}
{\bfseries SUPPLEMENTARY MATERIAL ON  \\ [0mm] ``TESTING FOR INDEPENDENCE IN HIGH DIMENSIONS BASED ON EMPIRICAL COPULAS''}
\vspace{.5cm}

{\small AXEL BÜCHER AND CAMBYSE PAKZAD}
\blfootnote{\today}

\end{center}

\begin{abstract}
This supplementary material contains missing details for the proofs in the main paper. In Section~\ref{subsec:ausmom}, we provide some basic summation formulas that are then used to derive some moments involving ranks. These formulas immediately imply Lemma~\ref{lem:2m}, as illustrated in Section~\ref{sec:2m}.
Missing steps for the proof of Lemma~\ref{lem:4m} are provided in Section~\ref{sec:miss4m}. Finally, results and proofs for Step 3 that involve joint weak convergence are presented in Section~\ref{sec:ps3b}.
\end{abstract}

\appendix

\section{Auxiliary Summation and Moment Formulas} \label{subsec:ausmom}

The following summation formulas are needed:
{\small
\begin{align*}
&\sum_{\ell=1}^n \ell = \frac12 n(n+1), \\
&\sum_{\ell=1}^n \ell^2 = \frac16 n(n+1)(2n+1), \\
&\sum_{\ell=1}^n \ell (\ell-1) = \frac13 (n-1)n(n+1), \\
&\sum_{\ell=1}^n \ell^2(\ell - 1) = \frac1{12} (n-1)n(n+1)(3n+2), \\
&\sum_{\ell=1}^n \ell^2(\ell-1)^2 = \frac{1}{15} (n-1)n(n+1)(3n^2-2) \\
&\sum_{k,\ell=1~(p.d.)}^n  k \ell =  \frac{1}{12} (n-1)n(n+1)(3n+2), \\
&\sum_{k,\ell=1~(p.d.)}^n \ell (\ell-1)k(k-1) = \frac1{45} (n-2)(n-1)n(n+1)(5n^2+n-3),  \\
&\sum_{k,\ell=1~(p.d.)}^n \ell k(k-1) = \frac1{12}(n-1)n(n+1)(2n^2-n-2), \\
&\sum_{k,\ell=1~(p.d.)}^n \max(k, \ell) = \frac23 (n-1)n(n+1), \\
&\sum_{k,\ell=1~(p.d.)}^n \max(k, \ell)^2 = \frac16 (n-1)n(n+1)(3n+2), \\
&\sum_{k,\ell=1~(p.d.)}^n k \max(k, \ell) = \frac{1}{8} (n-1)n(n+1)(3n+2), \\
&\sum_{k,\ell=1~(p.d.)}^n k(k-1) \max(k,\ell) = \frac{1}{60} (n-1)n(n+1)(16n^2-5n-14),\\
&\sum_{k,\ell, m=1~(p.d.)}^n m(m-1)\max(k,\ell)= \frac1{90} (n-2)(n-1)n(n+1)(20n^2-8n-21), \\
&\sum_{k,\ell,m=1~(p.d.)}^n m \max(k,\ell)  = \frac1{12} (n-2)(n-1)n(n+1)(4n+3), \\
&\sum_{k,\ell,m=1~(p.d.)}^n \max(k,\ell) \max(\ell,m) = \frac7{60} (n-2)(n-1)n(n+1)(4n+3) \\
&\sum_{k,\ell, m,p=1~(p.d.)}^n \max(k,\ell)\max(m,p) = \frac4{45} (n-3)(n-2)(n-1)n(n+1)(5n+4),
\end{align*}
}where (p.d.) means pairwise different. The summation formulas readily imply:
{\small
\begin{align}
B^1 &=  \label{eq:bb1}
\E\Big[ \frac{R_{1p}(R_{1p}-1)}{n(n+1)} \Big]
= \frac{n-1}{3n} ,\\
B_{1:2} &=  \label{eq:bb2}
\E\Big[\frac{R_{1p} \vee R_{2p} }{n+1}\Big] 
= \frac{2}{3},\\
B_{1:1} &=  \nonumber
 \E\Big[\frac{R_{1p}}{n+1}\Big] = \E\Big[\frac{R_{1p} \vee R_{1p} }{n+1}\Big]
= \frac{1}{2},\\
 B^{1,1} &=  \nonumber
 \E\Big[ \frac{R_{1p}^2(R_{1p}-1)^2}{n^2(n+1)^2} \Big] 
= \frac{(n-1)(3n^2-2)}{15n^2(n+1) },\\
B^{1,2} &=  \label{eq:bb5}
\E\Big[ \frac{R_{1p}(R_{1p}-1)}{n(n+1)} \frac{R_{2p}(R_{2p}-1)}{n(n+1)} \Big]
= \frac{(n-2)(5n^2+n-3)}{45n^2(n+1)}, \\
B^1_{2:3} &=  \label{eq:bb6}
\E\Big[ \frac{R_{1p}(R_{1p}-1)}{n(n+1)} \frac{R_{2p} \vee R_{3p}}{n+1} \Big]
= \frac{20n^2-8n-21}{90n(n+1)}, \\
B^1_{1:1} &= \nonumber
\E\Big[ \frac{R_{1p}(R_{1p}-1)}{n(n+1)} \frac{R_{1p} }{n+1} \Big] 
= \frac{(n-1)(3n+2)}{12n(n+1)}, \\
B^1_{1:2} &=  \nonumber
\E\Big[ \frac{R_{1p}(R_{1p}-1)}{n(n+1)} \frac{R_{1p} \vee R_{2p}}{n+1} \Big]
= \frac{16n^2-5n-14}{60n(n+1)}, \\
B^1_{2:2} &=  \nonumber
\E\Big[ \frac{R_{1p}(R_{1p}-1)}{n(n+1)} \frac{R_{2p} }{n+1} \Big]
=\frac{2n^2-n-2}{12n(n+1)},\\
B_{1:1,1:1} &= \nonumber
\E\Big[\Big(\frac{R_{1p}}{n+1}\Big)^2\Big]
= \frac{2n+1}{6(n+1)},\\
B_{1:1,1:2} &= \nonumber
\E\Big[\frac{R_{1p}}{n+1}\frac{R_{1p}\vee R_{2p}}{n+1}\Big] 
=  \frac{3n+2}{8(n+1)},\\
B_{1:1,2:2} &= \nonumber
\E\Big[ \frac{R_{1p}}{n+1} \frac{R_{2p} }{n+1} \Big]
= \frac{3n+2}{12(n+1)}, \\
B_{1:1,2:3} &= \nonumber
\E\Big[ \frac{R_{1p} }{n+1} \frac{R_{2p} \vee R_{3p}}{n+1} \Big]
 =\frac{4n+3}{12(n+1)}, \\
B_{1:2,1:2} &= \nonumber
\E\Big[ \Big(\frac{R_{1p} \vee R_{2p}}{n+1}\Big)^2 \Big]
= \frac{3n+2}{6(n+1)}, \\
B_{1:2,3:4} &= \label{eq:bb15}
\E\Big[ \frac{R_{1p} \vee R_{2p}}{n+1} \frac{R_{3p} \vee R_{4p}}{n+1} \Big]
= \frac{4(5n+4)}{45 (n+1)}, \\
B_{1:2,2:3} &= \label{eq:bb16}
\E\Big[ \frac{R_{1p} \vee R_{2p}}{n+1} \frac{R_{2p} \vee R_{3p}}{n+1} \Big]
 = \frac{7(4n+3)}{60(n+1)}.  
\end{align}
}

As a consequence, recalling the notation $\tilde c_{\mathbf{i}}=\E[ I_{i_1,i_2}^{(p)}  I_{i_3,i_4}^{(p)} ]$ and writing $c=\frac{2n+1}{6n}$, we obtain that
{\small \begin{align} 
\tilde c_{1,2,3,4} 
&= 
\E\Big[ \Big\{c+ \frac{R_{1p}(R_{1p}-1)}{n(n+1)}  - \frac{R_{1p}\vee R_{2p} }{n+1} \Big\} \Big\{c+ \frac{R_{3p}(R_{3p}-1)}{n(n+1)}  - \frac{R_{3p}\vee R_{4p} }{n+1} \Big\}  \Big] \nonumber\\
&=
c^2 + (B^{1,2}) + B_{1:2,3:4} + 2c (B^1) - 2  c B_{1:2} - 2 (B^1_{2:3})  \nonumber\\
&= 
\frac1{20n^2}, \label{eq:ccc1} \\
\tilde c_{1,1,2,3} 
&= 
\E\Big[ \Big\{c+ \frac{R_{1p}(R_{1p}-1)}{n(n+1)}  - \frac{R_{1p}}{n+1} \Big\} \Big\{c+ \frac{R_{2p}^2}{(n+1)^2}  - \frac{R_{2p}\vee R_{3p} }{n+1} \Big\}  \Big] \nonumber\\
&=
c^2 + c B^{1} - c B_{1:2} + c B^1 + B^{1,2} - B^1_{2:3} - cB_{1:1} - B^1_{2:2} + B_{1:1,2:3}\nonumber\\
&=
-\frac{5n^2-2n-9}{180n^2(n+1)}, \nonumber\\
\tilde c_{1,2,2,3}  
&= 
\E\Big[ \Big\{c+ \frac{R_{1p}(R_{1p}-1)}{2n(n+1)} + \frac{R_{2p}(R_{2p}-1)}{2n(n+1)}  - \frac{R_{1p} \vee R_{2p}} {n+1} \Big\} \nonumber\\
& \hspace{4cm} \times \Big\{c+ \frac{R_{2p}(R_{2p}-1)}{2n(n+1)} + \frac{R_{3p}(R_{3p}-1)}{2n(n+1)}  - \frac{R_{2p} \vee R_{3p}} {n+1} \Big\}\Big] \nonumber\\
&=
(c^2 + c B^{1} - c B_{1:2})
+
(\tfrac{c}2 B^1 +  B^{1,2}  - \tfrac12 B^{1}_{2:3})
+
(\tfrac{c}2 B^1 + \tfrac14 B^{1,1} + \tfrac14 B^{1,2} - \tfrac12 B^{1}_{1:2}) \nonumber\\
& \hspace{4cm} - ( cB_{1:2} + \tfrac12 B^1_{1:2} + \tfrac12 B^{1}_{2:3} - B_{1:2,2:3}) \nonumber\\
&=
-\frac{2n^2-8n+9}{180n^2(n+1)}, \nonumber\\
\tilde c_{1,1,2,2} 
&= 
\E\Big[ \Big\{c+ \frac{R_{1p}(R_{1p}-1)}{n(n+1)}  - \frac{R_{1p}}{n+1} \Big\} \Big\{c+ \frac{R_{2p}(R_{2p}-1)}{n(n+1)}  - \frac{R_{2p}}{n+1} \Big\}  \Big]  \nonumber \\
&= 
(c^2 + cB^{1}- cB_{1:1}) + (cB^1+B^{1,2} - B^1_{2:2}) - (cB_{1:1} + B^{1}_{2:2} - B_{1:1,2:2})  \nonumber \\
&= \frac{5n^3-6n^2-5n+9}{180n^2(n+1)},  \nonumber\\
\tilde c_{1,2,1,2} 
&=  
\E\Big[ \Big\{c+ \frac{R_{1p}(R_{1p}-1)}{2n(n+1)} + \frac{R_{2p}(R_{2p}-1)}{2n(n+1)}  - \frac{R_{1p} \vee R_{2p}} {n+1} \Big\}^2 \Big]  \nonumber \\
&= 
c^2 + \tfrac12 B^{1,1} + B_{1:2, 1:2} + 2cB^1-2cB_{1:2}+\tfrac12B^{1,2}-2B^1_{1:2} \nonumber\\
&= \frac{2n^3-6n^2-7n+9}{180n^2(n+1)}, \nonumber\\
\tilde c_{1,1,1,2} &=
\E\Big[ \Big\{c+ \frac{R_{1p}(R_{1p}-1)}{n(n+1)}  - \frac{R_{1p}}{n+1} \Big\} \Big\{c+ \frac{R_{1p}(R_{1p}-1)}{2n(n+1)}  + \frac{R_{2p}(R_{2p}-1)}{2n(n+1)}  - \frac{R_{1p} \vee R_{2p}}{n+1} \Big\}  \Big] \nonumber\\
&= (c^2 + cB^{1}-cB_{1,2})  + (cB^1 + \tfrac12 B^{1,1} + \tfrac12 B^{1,2} - B^1_{1:2}) - ( c B_{1:1} + \tfrac12 B^1_{1:1} + \tfrac12 B^1_{2:2} - B_{1:1, 1:2}) \nonumber\\
&= - \frac{2n^2-3}{60n^2(n+1)}, \nonumber\\
\nonumber
\tilde c_{1,1,1,1} &= 
\E\Big[ \Big\{c+ \frac{R_{1p}^2}{(n+1)^2}  - \frac{R_{1p}}{n+1} \Big\}^2\Big] \\
&= c^2+B^{1,1} + B_{1:1,1:1} + 2cB^1- 2 c B_{1:1} - 2 B^1_{1:1} \nonumber\\
&= \frac{(n-1)(2n^2-3)}{60n^2(n+1)}.\label{eq:ccc2}
\end{align}
}

\section{Proof of Lemma~\ref{lem:2m}} \label{sec:2m}

\begin{proof}[Proof of Lemma~\ref{lem:2m}]
All formulas can be deduced from Equations~\eqref{eq:bb1}--\eqref{eq:bb16} in Section~\ref{subsec:ausmom} after some tedious calculations that have been checked with computer algebra systems. We exemplarily work out the case $|\mathbf i|=4$, and write $\mathbf i=(1,2,3,4)$. 
Since $\tilde I_{3,4}^{(p)}$ is centered, we have
\begin{align} \label{eq:p2m}
\E\left[  \tilde I_{1,2}^{(p)} \tilde I^{(p)}_{3,4} \right]
&= 
\E\left[ \frac{R_{1p}\left(R_{1p}-1 \right) }{n(n+1)} \tilde I^{(p)}_{3,4} \right] -\E\left[ \frac{R_{1p}\vee R_{2p}}{n+1}\tilde I^{(p)}_{3,4} \right].
\end{align}
By the definition of $\tilde I_{i,j}^{(p)}$ in \eqref{eq:bii}, the first summand on the right-hand side may be written as
	\begin{align} 
	\E\left[ \frac{R_{1p}\left(R_{1p}-1 \right) }{n(n+1)} \tilde I^{(p)}_{3,4} \right]
	&= \nonumber
	\E\left[ \frac{R_{1p}\left(R_{1p}-1 \right) }{3n^2} \right]+ \E\left[\frac{R_{1p}(R_{1p}-1)R_{3p}(R_{3p}-1)}{n^2(n+1)^2} \right]\\
	&\hspace{4cm} \nonumber
	- \E\left[\frac{R_{1p}(R_{1p}-1)\left( R_{3p}\vee R_{4p}\right) }{n(n+1)^2} \right] \\
	&= \nonumber
	\frac{n+1}{3n} B^1  + B^{1,2} -  B^{1}_{2:3}  \\
	&=\nonumber
	 \frac{(n-1)(n+1)}{9n^2} + \frac{(n-2)(5n^2+n-3)}{45n^2(n+1)} - \frac{20n^2-8n-21}{90n(n+1)} \\
	 &=\frac{n+2}{90n^2(n+1)},
	 \label{eq:p3m}
	\end{align}
	where we used \eqref{eq:bb1}, \eqref{eq:bb5} and \eqref{eq:bb6}.
	
	Next, the second summand on the right-hand side of \eqref{eq:p2m} may be written as
	\begin{align*}
	\E\left[ \frac{R_{1p}\vee R_{2p}}{n+1}\tilde I^{(p)}_{3,4} \right]
	&= 
	\E\left[ \frac{R_{1p}\vee R_{2p}}{3n}\right] +\E\left[\frac{\left( R_{1p}\vee R_{2p}\right) R_{3p}\left(R_{3p}-1 \right) }{n(n+1)^2}\right] \\
		&\hspace{4cm}
	-\E\left[\frac{\left(R_{1p}\vee R_{2p} \right)\left( R_{3p}\vee R_{4p}\right)  }{(n+1)^2}\right] \\
	&=
		\frac{n+1}{3n} B_{1:2}  + B^{1}_{2:3} - B_{1:2, 3:4} \\
		&=  \frac{2(n+1)}{9n}+\frac{20n^2-8n-21}{90n(n+1)}-\frac{4(5n+4)}{45(n+1)} \\
		&=-\frac{1}{90n(n+1)}.
	\end{align*}
		where we used \eqref{eq:bb2}, \eqref{eq:bb6} and \eqref{eq:bb15}. As a consequence, by \eqref{eq:p2m} and \eqref{eq:p3m},
	\[
	\E\left[  \tilde I_{1,2}^{(p)} \tilde I^{(p)}_{3,4} \right]
	=\frac{2}{90n^2}. \qedhere
	\]
\end{proof}

\section{Missing steps for the proof of Lemma~\ref{lem:4m}} \label{sec:miss4m}

\subsection{Proof of \eqref{eq:summand_rate} for $\kappa_{A,i}= -3$.}
\label{sect:ka-3}
 
We need to consider the cases $(|\mathbf i|,|A|)=(8,3)$ and  $(|\mathbf i|,|A|)=(7,4)$. We start by rewriting $E_{n,\bm \ell} (A)$ defined in \eqref{eq:enlab}: using identities (\ref{multinomial}) and (\ref{itilde}), we may write
\begin{align*} 
E_{n,\bm \ell} (A)&=\sum_{B\subset\psi^{-1}(A) }  E^{(A,B)}_{n, \bm \ell},
\end{align*}
where, for $B \subset \psi^{-1}(A) \subset S_8$ and $\ell =(\ell_s)_{s \in A} \in \{1, \dots, n\}^{|A|}$,
\begin{align} \label{eq:eabnl}
	E^{(A,B)}_{n,\bm \ell} =	\begin{dcases}
		\frac{{(-1)^{\zp[b]{B^c}}}}{n^{\zp[b]{B^c}}} \left( \prod_{j \in B^c} \ell_{\psi(j)}\right)  \p\left[R_{i_j}\leq \ell_{\psi(j)},j\in B\right] &, B\neq \emptyset \\
		\frac{{(-1)^{\zp[b]{B^c}}}}{n^{\zp[b]{B^c}}}  \cdot \prod_{j \in B^c} \ell_{\psi(j)} &, B=\emptyset. 
	\end{dcases}
\end{align}
Next, note that, for $\emptyset \ne D \subset \{1, \dots, n\}$ and $\bm k=(k_1, \dots, k_{|D|}) \in \{1, \dots, n\}^{|D|}$,
\begin{align}\label{cdf}
\p\left[R_{i}\leq  k_i, i \in D \right]
&=
\frac{k_{(1)} (k_{(2)}-1)\cdots (k_{(|D|)}-|D|+1)}{n(n-1)\cdots (n-|D|+1)},
\end{align}
which follows from a straightforward calculation. Here, $k_{(1)} \le \dots k_{(|D|)}$ denotes the order statistic of $k_1, \dots, k_{|D|}$.
For the case of $|\mathbf i|=8$, \eqref{cdf} allows to rewrite \eqref{eq:eabnl} as
\begin{align}   \label{eq:meabnl}
E^{(A,B)}_{n,\bm \ell}&=\left\{
	\frac{{(-1)^{\zp[b]{B^c}}}\prod_{j \in B^c} \ell_{\psi(j)}}{n^{\zp[b]{B^c}}} \right\}  \cdot \left\{	 \frac{\prod_{j\in B}\left(\ell_{ \psi(j) } -\text{rank}_B(j)+1 \right)   }{n(n-1)\cdots (n-\zp[b]{B}+1)}\right\}
\end{align}
for any nonempty set $B \subset \psi^{-1}(A)$, where $\text{rank}_{B}(x)=\sum_{\ell \in B} 1_{\{ x\leq \ell \}}$ is the rank of $x$ in $B$, e.g., $\ds{ \text{rank}_{\{ 1,3,4,7 \}}(4) = 3  }$.

Now, consider the case $(|\mathbf i|,|A|)=(8,3)$, for which we may assume that $\mathbf i=(1, \dots, 8)$ and $A=S_3=\{ 1,2,3 \}$, such that $\psi^{-1}(A)=S_6=\{1, \dots, 6\}$.
As before, we further assume that $\ell_1 \le \ell_2 \le \ell_3$. Recall that we must prove that
\begin{align} \label{eq:goal83}
\sum_{B\subset\psi^{-1}(A) }  E^{(A,B)}_{n, \bm \ell} 
= 
\sum_{\sigma=0}^6 \sum_{B\subset S_6: |B|=\sigma} E^{(S_3,B)}_{n, \bm \ell}  = O(n^{-3}).
\end{align}
For each $\sigma\in\{0, \dots, 6\}$, the sums in the previous decomposition may be calculated explicitly (all formulas have been checked with computer algebra systems):
{\small
\begin{asparaenum}
\item 
For $|B|=0$, we get
\[
\sum\limits_{\substack{B\subset S_6 \\ \zp[b]{B}=0 }} E^{(S_3,B)}_{n,\bm \ell} =E^{(S_3,\emptyset)}_{n,\bm \ell} = \frac{\ell^2_1\ell^2_2\ell^2_3}{n^6}.
\]
\item For $|B|=1$ (i.e., $B= \{j\}$ with $1\leq j \leq 6$), there are $\binom{6}{1}=6$ terms arising 
\begin{align*}
\sum\limits_{\substack{B\subset S_6 \\ \zp[b]{B}=1 }} E^{(S_3,B)}_{n,\ell} 
&= 
-6\cdot\frac{\ell^2_1\ell^2_2\ell^2_3}{n^6}.
\end{align*}
\item When $\ds{\zp[b]{B}=2}$,  there are $\binom{6}{2}=15$ terms arising
\begin{align*}
\sum\limits_{\substack{B\subset S_6 \\ \zp[b]{B}=2 }} E^{(S_3,B)}_{n,\ell} 
&= 
\frac{\ell^2_2\ell^2_3}{n^4} \cdot \frac{\ell_1(\ell_1-1)}{n(n-1)}+\frac{\ell^2_1\ell^2_3}{n^4} \cdot \frac{\ell_2(\ell_2-1)}{n(n-1)}+\frac{\ell^2_1\ell^2_2}{n^4} \cdot \frac{\ell_3(\ell_3-1)}{n(n-1)}
	\\& \quad +4\frac{\ell_1\ell_2\ell^2_3}{n^4} \cdot \frac{\ell_1(\ell_2-1)}{n(n-1)}+4\frac{\ell_1\ell^2_2\ell_3}{n^4} \cdot \frac{\ell_1(\ell_3-1)}{n(n-1)} + 4\frac{\ell^2_1\ell_2\ell_3}{n^4} \cdot \frac{\ell_2(\ell_3-1)}{n(n-1)} .
\end{align*}
\item 
When $\ds{\zp[b]{B}=3}$, there are $\binom{6}{3}=20$ terms arising
\begin{align*}-\sum\limits_{\substack{B\subset S_6 \\ \zp[b]{B}=3 }} E^{(S_3,B)}_{n,\ell} 
&=2\frac{\ell_2 \ell^2_3 }{n^3}\cdot\frac{\ell_1\left(\ell_1-1\right)\left(\ell_2-2\right) }{n(n-1)(n-2)}
	 +2\frac{\ell^2_2 \ell_3 }{n^3}\cdot\frac{\ell_1\left(\ell_1-1\right)\left(\ell_3-2\right) }{n(n-1)(n-2)}
	\\&\quad +2\frac{\ell_1 \ell^2_3 }{n^3}\cdot\frac{\ell_1\left(\ell_2-1\right)\left(\ell_2-2\right) }{n(n-1)(n-2)}
	+2\frac{\ell^2_1 \ell_3 }{n^3}\cdot\frac{\ell_2\left(\ell_2-1\right)\left(\ell_3-2\right) }{n(n-1)(n-2)}
	\\&\quad +2\frac{\ell_1 \ell^2_2 }{n^3}\cdot\frac{\ell_1\left(\ell_3-1\right)\left(\ell_3-2\right) }{n(n-1)(n-2)}
	+2\frac{\ell^2_1 \ell_2 }{n^3}\cdot\frac{\ell_2\left(\ell_3-1\right)\left(\ell_3-2\right) }{n(n-1)(n-2)}
	\\&\quad +8\frac{\ell_1 \ell_2 \ell_3}{n^3}\cdot\frac{\ell_1\left(\ell_2-1\right)\left(\ell_3-2\right) }{n(n-1)(n-2)}.
\end{align*}
\item
When $\ds{\zp[b]{B}=4}$, there are $\binom{6}{4}=15$ terms arising
\begin{align*}
\sum\limits_{\substack{B\subset S_6 \\ \zp[b]{B}=4 }} E^{(S_3,B)}_{n,\ell} 
&= \frac{\ell^2_1 }{n^2}\cdot\frac{\ell_2\left(\ell_2-1\right)\left(\ell_3-2\right)\left(\ell_3-3\right)}{n(n-1)(n-2)(n-3)}
	 \frac{\ell^2_2 }{n^2}\cdot\frac{\ell_1\left(\ell_1-1\right)\left(\ell_3-2\right)\left(\ell_3-3\right)}{n(n-1)(n-2)(n-3)}
	\\&\quad+ \frac{\ell^2_3 }{n^2}\cdot\frac{\ell_1\left(\ell_1-1\right)\left(\ell_2-2\right)\left(\ell_2-3\right)}{n(n-1)(n-2)(n-3)}
	 + 4 \frac{\ell_1\ell_2 }{n^2}\cdot\frac{\ell_1\left(\ell_2-1\right)\left(\ell_3-2\right)\left(\ell_3-3\right)}{n(n-1)(n-2)(n-3)}
	\\&\quad+ 4 \frac{\ell_1\ell_3 }{n^2}\cdot\frac{\ell_1\left(\ell_2-1\right)\left(\ell_2-2\right)\left(\ell_3-3\right)}{n(n-1)(n-2)(n-3)}
	+ 4 \frac{\ell_2\ell_3 }{n^2}\cdot\frac{\ell_1\left(\ell_1-1\right)\left(\ell_2-2\right)\left(\ell_3-3\right)}{n(n-1)(n-2)(n-3)}
\end{align*}
\item
When $\ds{\zp[b]{B}=5}$, there are $\binom{6}{5}=6$ terms arising
\begin{align*}-\sum\limits_{\substack{B\subset S_6 \\ \zp[b]{B}=5 }} E^{(S_3,B)}_{n,\ell} 
&= 
	2\frac{\ell_1}{n}\cdot\frac{\ell_1\left( \ell_2-1\right)\left( \ell_2-2\right)  \left( \ell_3-3\right) \left( \ell_3-4\right)}{n(n-1)(n-2)(n-3)(n-4)}
	\\&\quad + 2\frac{\ell_2}{n}\cdot\frac{\ell_1\left( \ell_1-1\right)\left( \ell_2-2\right)  \left( \ell_3-3\right) \left( \ell_3-4\right) }{n(n-1)(n-2)(n-3)(n-4)}
	\\&\quad +2\frac{\ell_3}{n}\cdot\frac{\ell_1\left( \ell_1-1\right)\left( \ell_2-2\right)  \left( \ell_2-3\right) \left( \ell_3-4\right) }{n(n-1)(n-2)(n-3)(n-4)}
\end{align*}
\item
In the case $|B|=6$, i.e., $\ds{B=S_6}$, one has $\ds{\text{rank}_B(j)=j}$. This immediately implies
\begin{align*}\sum\limits_{\substack{B\subset S_6 \\ \zp[b]{B}=6 }} E^{(S_3,B)}_{n,\ell}
	=\frac{ \ell_{1}\left( \ell_{1}-1\right) \left(\ell_{2}-2 \right)\left(\ell_{2}-3 \right) \left(\ell_{3}-4\right)\left(\ell_{3}-5\right) }{n(n-1)\cdots (n-5)}.
\end{align*}
\end{asparaenum}}
\noindent 
Assembling terms, one obtains
\begin{align*} 
\sum_{B\subset S_6}  E^{(S_3,B)}_{n,\ell}
	&=  \frac{-\ell_1 \left( n-\ell_3 \right) \left( \sum_{j=0}^4 n^j p_j\left( \ell_1,\ell_2,\ell_3 \right)  \right) }{n^6 (n-1)(n-2)(n-3)(n-4)(n-5)},
\end{align*}
	where $\ds{p_j \in\R_{(6-j)\wedge 4  }\left[X_1,X_2,X_3\right]}$ for $\ds{0\leq j \leq 4}$ with $\R_k[X_j,j=1,\ldots,d]$ denoting the set of $d$-variate polynomials with total degree at most $k$.
As a consequence, since $\ell_j \le n$, the numerator is of order $\ds{O(n^8)}$, which implies \eqref{eq:goal83}.

Now, consider the case $(|\mathbf i|,|A|)=(7,4)$. Without loss of generality, we may consider $\mathbf i=(1,2,3,4,5,6,7,1)$. As before, we assume that $\ell_1 \le \ell_2 \le \ell_3 \le \ell_4$ and $A=S_4=\{1, 2, 3,4\}$, which implies $\psi^{-1}(A)=S_8$. Recall that we must prove that
\begin{align} \label{eq:goal74}
\sum_{B\subset\psi^{-1}(A) }  E^{(A,B)}_{n, \bm \ell} 
= 
\sum_{\sigma=0}^8 \sum_{B\subset S_8: |B|=\sigma} E^{(A,B)}_{n, \bm \ell}  = O(n^{-3}).
\end{align}

For each $\sigma\in\{0, \dots, 8\}$, since $i_1=1=i_8$, we may decompose
\begin{align} \label{eq:esi}
\sum_{B\subset S_8: |B|=\sigma} E^{(A,B)}_{n, \bm \ell} 
= 
\sum\limits_{\substack{B\subset S_{8} : |B|=\sigma\\ \{1,8\} \cup B \neq B  }}  E^{(S_4,B)}_{n,\ell}
+
\sum\limits_{\substack{B\subset S_{8}: |B|=\sigma \\ \{1,8\} \subset B }}  E^{(S_4,B)}_{n,\ell}.
\end{align}
We are going to invoke \eqref{eq:eabnl}. For that purpose note that, by \eqref{cdf}, 
\begin{align*}
	\p\left[R_{i_j}\leq \ell_{\psi(j)},j\in B\right] =	
	\begin{dcases} 
	\frac{\prod_{j\in B\setminus \{8\}  }\left(\ell_{ \psi(j)} -\text{rank}_{B\setminus \{8\} }(j)+1 \right)   }{n(n-1)\cdots (n-\zp[b]{B}+2)} &, \{1,8\} \subset B \\
		\frac{\prod_{j\in B}\left(\ell_{ \psi(j) } -\text{rank}_B(j)+1 \right)   }{n(n-1)\cdots (n-\zp[b]{B}+1)} & ,  \{1,8\} \cup B \neq B ,
	\end{dcases}
\end{align*}
see also  \eqref{eq:meabnl}. The previous formula allows to calculate the sums in \eqref{eq:esi} explicitly (all formulas have been checked with computer algebra systems):

{\small
\begin{asparaenum}
\item When $|B=0$, i.e., $\ds{B=\emptyset}$, we get: 
\[
\sum\limits_{\substack{B\subset S_8 \\ \zp[b]{B}=0 }}  E^{(S_4,B)}_{n,\ell} = E^{(S_4,\emptyset)}_{n,\ell} = \frac{\ell^2_1\ell^2_2\ell^2_3\ell^2_4}{n^8}.
\]
\item
When $\ds{\zp[b]{B}=1}$, i.e. $\ds{B= \{j\}}$ with $\ds{1\leq j \leq 8}$, there are $\binom{8}{1}=8$ terms arising
\begin{align*}
-\sum\limits_{\substack{B\subset S_8 \\ \zp[b]{B}=1 }}  E^{(S_4,B)}_{n,\ell} &= 8\cdot\frac{\ell^2_1\ell^2_2\ell^2_3\ell^2_4}{n^8}.
\end{align*}
\item
When $\ds{\zp[b]{B}=2}$,  there are $\binom{8}{2}=28$ terms arising \begin{align*}\sum\limits_{\substack{B\subset S_8 \\ \zp[b]{B}=2 }}  E^{(S_4,B)}_{n,\ell} &= \frac{\ell^2_2\ell^2_3\ell^2_4}{n^6} \cdot \frac{\ell_1(\ell_1-1)}{n(n-1)}+\frac{\ell^2_1\ell^2_3\ell^2_4}{n^6} \cdot \frac{\ell_2(\ell_2-1)}{n(n-1)}+\frac{\ell^2_1\ell^2_2\ell^2_4}{n^6} \cdot \frac{\ell_3(\ell_3-1)}{n(n-1)}+\frac{\ell^2_1\ell^2_2\ell^2_3}{n^6} \cdot \frac{\ell_4(\ell_4-1)}{n(n-1)}
	\\&\qquad +4\frac{\ell_1\ell_2\ell^2_3\ell^2_4}{n^6} \cdot \frac{\ell_1(\ell_2-1)}{n(n-1)}+4\frac{\ell_1\ell^2_2\ell_3\ell^2_4}{n^6} \cdot \frac{\ell_1(\ell_3-1)}{n(n-1)}+3\frac{\ell_1\ell^2_2\ell^2_3\ell_4}{n^6} \cdot \frac{\ell_1(\ell_4-1)}{n(n-1)}
	\\&\qquad + \frac{\ell_1\ell^2_2\ell^2_3\ell_4}{n^6} \cdot \frac{\ell_1}{n}
	\\&\qquad + 4\frac{\ell^2_1\ell_2\ell_3\ell^2_4}{n^6} \cdot \frac{\ell_2(\ell_3-1)}{n(n-1)} + 4\frac{\ell^2_1\ell_2\ell^2_3\ell_4}{n^6} \cdot \frac{\ell_2(\ell_4-1)}{n(n-1)} +  4\frac{\ell^2_1\ell^2_2\ell_3\ell_4}{n^6} \cdot \frac{\ell_3(\ell_4-1)}{n(n-1)}.
\end{align*}
\item
When $\ds{\zp[b]{B}=3}$, there are $\binom{8}{3}=56$ terms arising:
\begin{align*}-\sum\limits_{\substack{B\subset S_8 \\ \zp[b]{B}=3 }}  E^{(S_4,B)}_{n,\ell} &=2\frac{\ell_2 \ell^2_3 \ell^2_4}{n^5}\cdot\frac{\ell_1\left(\ell_1-1\right)\left(\ell_2-2\right) }{n(n-1)(n-2)}
	\\&\qquad +2\frac{\ell^2_2 \ell_3 \ell^2_4}{n^5}\cdot\frac{\ell_1\left(\ell_1-1\right)\left(\ell_3-2\right) }{n(n-1)(n-2)}
	\\&\qquad +\frac{\ell^2_2 \ell^2_3 \ell_4}{n^5}\cdot\frac{\ell_1\left(\ell_1-1\right)\left(\ell_4-2\right) }{n(n-1)(n-2)} + \frac{\ell^2_2 \ell^2_3 \ell_4}{n^5}\cdot\frac{\ell_1 \left(\ell_1-1\right) }{n(n-1)}
	\\&\qquad +2\frac{\ell_1 \ell^2_3 \ell^2_4}{n^5}\cdot\frac{\ell_1\left(\ell_2-1\right)\left(\ell_2-2\right) }{n(n-1)(n-2)}
	\\&\qquad +2\frac{\ell^2_1 \ell_3 \ell^2_4}{n^5}\cdot\frac{\ell_2\left(\ell_2-1\right)\left(\ell_3-2\right) }{n(n-1)(n-2)}
	\\&\qquad +2\frac{\ell^2_1 \ell^2_3 \ell_4}{n^5}\cdot\frac{\ell_2\left(\ell_2-1\right)\left(\ell_4-2\right) }{n(n-1)(n-2)}
	\\&\qquad +2\frac{\ell_1 \ell^2_2 \ell^2_4}{n^5}\cdot\frac{\ell_1\left(\ell_3-1\right)\left(\ell_3-2\right) }{n(n-1)(n-2)}
	\\&\qquad +2\frac{\ell^2_1 \ell_2 \ell^2_4}{n^5}\cdot\frac{\ell_2\left(\ell_3-1\right)\left(\ell_3-2\right) }{n(n-1)(n-2)}
	\\&\qquad +2\frac{\ell^2_1 \ell^2_2 \ell_4}{n^5}\cdot\frac{\ell_3\left(\ell_3-1\right)\left(\ell_4-2\right) }{n(n-1)(n-2)}
	\\&\qquad +\frac{\ell_1 \ell^2_2 \ell^2_3}{n^5}\cdot\frac{\ell_1\left(\ell_4-1\right)\left(\ell_4-2\right) }{n(n-1)(n-2)}+\frac{\ell_1 \ell^2_2 \ell^2_3}{n^5}\cdot\frac{\ell_1 \left(\ell_4-1\right) }{n(n-1)}
	\\&\qquad +2\frac{\ell^2_1 \ell_2 \ell^2_3}{n^5}\cdot\frac{\ell_2\left(\ell_4-1\right)\left(\ell_4-2\right) }{n(n-1)(n-2)}
	\\&\qquad +2\frac{\ell^2_1 \ell^2_2 \ell_3}{n^5}\cdot\frac{\ell_3\left(\ell_4-1\right)\left(\ell_4-2\right) }{n(n-1)(n-2)}
	\\&\qquad +8\frac{\ell^2_1 \ell_2 \ell_3 \ell_4}{n^5}\cdot\frac{\ell_2\left(\ell_3-1\right)\left(\ell_4-2\right) }{n(n-1)(n-2)}
	\\&\qquad +6\frac{\ell_1 \ell^2_2 \ell_3 \ell_4}{n^5}\cdot\frac{\ell_1\left(\ell_3-1\right)\left(\ell_4-2\right) }{n(n-1)(n-2)} +2\frac{\ell_1 \ell^2_2 \ell_3 \ell_4}{n^5}\cdot\frac{\ell_1 \left(\ell_3-1\right) }{n(n-1)}
	\\&\qquad +6\frac{\ell_1 \ell_2 \ell^2_3 \ell_4}{n^5}\cdot\frac{\ell_1\left(\ell_2-1\right)\left(\ell_4-2\right) }{n(n-1)(n-2)}+2\frac{\ell_1 \ell_2 \ell^2_3 \ell_4}{n^5}\cdot\frac{\ell_1 \left(\ell_2-1\right) }{n(n-1)}
	\\&\qquad +8\frac{\ell_1 \ell_2 \ell_3 \ell^2_4}{n^5}\cdot\frac{\ell_1\left(\ell_2-1\right)\left(\ell_3-2\right) }{n(n-1)(n-2)}.
\end{align*}
\item
When $|B|=4$, there are $\binom{8}{4}=70$ terms arising
\begin{align*}\sum\limits_{\substack{B\subset S_8 \\ \zp[b]{B}=4 }} E^{(S_4,B)}_{n,\ell} &= \frac{\ell^2_1 \ell^2_2}{n^4}\cdot\frac{\ell_3\left(\ell_3-1\right)\left(\ell_4-2\right)\left(\ell_4-3\right)}{n(n-1)(n-2)(n-3)}
	\\&\qquad +\frac{\ell^2_1 \ell^2_3}{n^4}\cdot\frac{\ell_2\left(\ell_2-1\right)\left(\ell_4-2\right)\left(\ell_4-3\right)}{n(n-1)(n-2)(n-3)}
	\\&\qquad +\frac{\ell^2_1 \ell^2_4}{n^4}\cdot\frac{\ell_2\left(\ell_2-1\right)\left(\ell_3-2\right)\left(\ell_3-3\right)}{n(n-1)(n-2)(n-3)}
	\\&\qquad + \frac{\ell^2_2 \ell^2_3 }{n^4}\cdot\frac{\ell_1 \left(\ell_1-1\right)\left(\ell_4-2\right)}{n(n-1)(n-2)}
	\\&\qquad +\frac{\ell^2_2 \ell^2_4}{n^4}\cdot\frac{\ell_1\left(\ell_1-1\right)\left(\ell_3-2\right)\left(\ell_3-3\right)}{n(n-1)(n-2)(n-3)}
	\\&\qquad +\frac{\ell^2_3\ell^2_4}{n^4}\cdot\frac{\ell_1\left(\ell_1-1\right)\left(\ell_2-2\right)\left(\ell_2-3\right)}{n(n-1)(n-2)(n-3)}
	\\&\qquad +12\frac{\ell_1 \ell_2 \ell_3 \ell_4}{n^4}\cdot\frac{\ell_1\left(\ell_2-1\right)\left(\ell_3-2\right)\left(\ell_4-3\right)}{n(n-1)(n-2)(n-3)}+4\frac{\ell_1 \ell_2 \ell_3 \ell_4}{n^4}\cdot\frac{\ell_1 \left(\ell_2-1\right)\left(\ell_3-2\right)}{n(n-1)(n-2)}
	\\&\qquad +4\frac{\ell^2_1 \ell_2 \ell_3 }{n^4}\cdot\frac{\ell_2\left(\ell_3-1\right)\left(\ell_4-2\right)\left(\ell_4-3\right)}{n(n-1)(n-2)(n-3)}
	\\&\qquad +4\frac{\ell^2_1 \ell_2 \ell_4 }{n^4}\cdot\frac{\ell_2\left(\ell_3-1\right)\left(\ell_3-2\right)\left(\ell_4-3\right)}{n(n-1)(n-2)(n-3)}
	\\&\qquad +4\frac{\ell^2_1 \ell_3 \ell_4}{n^4}\cdot\frac{\ell_2\left(\ell_2-1\right)\left(\ell_3-2\right)\left(\ell_4-3\right)}{n(n-1)(n-2)(n-3)}
	\\&\qquad +2\frac{\ell_1 \ell^2_2 \ell_3 }{n^4}\cdot\frac{\ell_1\left(\ell_3-1\right)\left(\ell_4-2\right)\left(\ell_4-3\right)}{n(n-1)(n-2)(n-3)}+2\frac{\ell_1 \ell^2_2 \ell_3 }{n^4}\cdot\frac{\ell_1 \left(\ell_3-1\right)\left(\ell_4-2\right)}{n(n-1)(n-2)}
	\\&\qquad +3\frac{\ell_1 \ell^2_2 \ell_4 }{n^4}\cdot\frac{\ell_1\left(\ell_3-1\right)\left(\ell_3-2\right)\left(\ell_4-3\right)}{n(n-1)(n-2)(n-3)} + \frac{\ell_1 \ell^2_2 \ell_4 }{n^4}\cdot\frac{\ell_1 \left(\ell_3-1\right)\left(\ell_3-2\right)}{n(n-1)(n-2)}
	\\&\qquad +2\frac{\ell^2_2 \ell_3 \ell_4}{n^4}\cdot\frac{\ell_1\left(\ell_1-1\right)\left(\ell_3-2\right)\left(\ell_4-3\right)}{n(n-1)(n-2)(n-3)}+2\frac{\ell^2_2 \ell_3 \ell_4}{n^4}\cdot\frac{\ell_1 \left(\ell_1-1\right)\left(\ell_3-2\right)}{n(n-1)(n-2)}
	\\&\qquad +2\frac{\ell_1 \ell_2 \ell^2_3}{n^4}\cdot\frac{\ell_1\left(\ell_2-1\right)\left(\ell_4-2\right)\left(\ell_4-3\right)}{n(n-1)(n-2)(n-3)}+2\frac{\ell_1 \ell_2 \ell^2_3}{n^4}\cdot\frac{\ell_1 \left(\ell_2-1\right)\left(\ell_4-2\right)}{n(n-1)(n-2)}
	\\&\qquad +3\frac{\ell_1 \ell^2_3 \ell_4 }{n^4}\cdot\frac{\ell_1\left(\ell_2-1\right)\left(\ell_2-2\right)\left(\ell_4-3\right)}{n(n-1)(n-2)(n-3)}+\frac{\ell_1 \ell^2_3 \ell_4 }{n^4}\cdot\frac{\ell_1 \left(\ell_2-1\right)\left(\ell_2-2\right)}{n(n-1)(n-2)}
	\\&\qquad +2\frac{\ell_2 \ell^2_3 \ell_4 }{n^4}\cdot\frac{\ell_1\left(\ell_1-1\right)\left(\ell_2-2\right)\left(\ell_4-3\right)}{n(n-1)(n-2)(n-3)}+2\frac{\ell_2 \ell^2_3 \ell_4 }{n^4}\cdot\frac{\ell_1 \left(\ell_1-1\right)\left(\ell_2-2\right)}{n(n-1)(n-2)}
	\\&\qquad +4\frac{\ell_1 \ell_2 \ell^2_4}{n^4}\cdot\frac{\ell_1\left(\ell_2-1\right)\left(\ell_3-2\right)\left(\ell_3-3\right)}{n(n-1)(n-2)(n-3)}
	\\&\qquad +4\frac{\ell_1 \ell_3 \ell^2_4}{n^4}\cdot\frac{\ell_1\left(\ell_2-1\right)\left(\ell_2-2\right)\left(\ell_3-3\right)}{n(n-1)(n-2)(n-3)}
	\\&\qquad +4\frac{\ell_2 \ell_3 \ell^2_4}{n^4}\cdot\frac{\ell_1\left(\ell_1-1\right)\left(\ell_2-2\right)\left(\ell_3-3\right)}{n(n-1)(n-2)(n-3)}.
\end{align*}
\item
When $\ds{\zp[b]{B}=5}$, there are $\binom{8}{5}=56$ terms arising
\begin{align*}-\sum\limits_{\substack{B\subset S_8 \\ \zp[b]{B}=5 }}  E^{(S_4,B)}_{n,\ell}&=2\frac{\ell^2_1\ell_2}{n^3}\cdot\frac{\ell_2\left(\ell_3-1\right)\left(\ell_3-2\right)\left(\ell_4-3\right)\left(\ell_4-4\right)}{n(n-1)(n-2)(n-3)(n-4)}
	\\&\qquad +2\frac{\ell^2_1\ell_3}{n^3}\cdot\frac{\ell_2\left(\ell_2-1\right)\left(\ell_3-2\right)\left(\ell_4-3\right)\left(\ell_4-4\right)}{n(n-1)(n-2)(n-3)(n-4)}
	\\&\qquad +2\frac{\ell^2_1\ell_4}{n^3}\cdot\frac{\ell_2\left(\ell_2-1\right)\left(\ell_3-2\right)\left(\ell_3-3\right)\left(\ell_4-4\right)}{n(n-1)(n-2)(n-3)(n-4)}
	\\&\qquad  +\frac{\ell_1\ell^2_2}{n^3}\cdot\frac{\ell_1\left(\ell_3-1\right)\left(\ell_3-2\right)\left(\ell_4-3\right)\left(\ell_4-4\right)}{n(n-1)(n-2)(n-3)(n-4)} + \frac{\ell_1\ell^2_2}{n^3}\cdot\frac{\ell_1\left(\ell_3-1\right)\left(\ell_3-2\right)\left(\ell_4-3\right)}{n(n-1)(n-2)(n-3)}
	\\&\qquad  +2\frac{\ell^2_2\ell_3}{n^3}\cdot\frac{\ell_1\left(\ell_1-1\right)\left(\ell_3-2\right)\left(\ell_4-3\right)}{n(n-1)(n-2)(n-3)}
	\\&\qquad  +\frac{\ell^2_2\ell_4}{n^3}\cdot\frac{\ell_1\left(\ell_1-1\right)\left(\ell_3-2\right)\left(\ell_3-3\right)\left(\ell_4-4\right)}{n(n-1)(n-2)(n-3)(n-4)} +\frac{\ell^2_2\ell_4}{n^3}\cdot\frac{\ell_1\left(\ell_1-1\right)\left(\ell_3-2\right)\left(\ell_3-3\right)}{n(n-1)(n-2)(n-3)} 
	\\&\qquad  +\frac{\ell_1\ell^2_3}{n^3}\cdot\frac{\ell_1\left(\ell_2-1\right)\left(\ell_2-2\right)\left(\ell_4-3\right)\left(\ell_4-4\right)}{n(n-1)(n-2)(n-3)(n-4)} +\frac{\ell_1\ell^2_3}{n^3}\cdot\frac{\ell_1\left(\ell_2-1\right)\left(\ell_2-2\right)\left(\ell_4-3\right)}{n(n-1)(n-2)(n-3)}
	\\&\qquad  +2\frac{\ell_2\ell^2_3}{n^3}\cdot\frac{\ell_1\left(\ell_1-1\right)\left(\ell_2-2\right)\left(\ell_4-3\right)}{n(n-1)(n-2)(n-3)} 
	\\&\qquad  +\frac{\ell^2_3\ell_4}{n^3}\cdot\frac{\ell_1\left(\ell_1-1\right)\left(\ell_2-2\right)\left(\ell_2-3\right)\left(\ell_4-4\right)}{n(n-1)(n-2)(n-3)(n-4)} + \frac{\ell^2_3\ell_4}{n^3}\cdot\frac{\ell_1\left(\ell_1-1\right)\left(\ell_2-2\right)\left(\ell_2-3\right)}{n(n-1)(n-2)(n-3)} 
	\\&\qquad  +2\frac{\ell_1\ell^2_4}{n^3}\cdot\frac{\ell_1\left(\ell_2-1\right)\left(\ell_2-2\right)\left(\ell_3-3\right)\left(\ell_3-4\right)}{n(n-1)(n-2)(n-3)(n-4)} 
	\\&\qquad +2\frac{\ell_2\ell^2_4}{n^3}\cdot\frac{\ell_1\left(\ell_1-1\right)\left(\ell_2-2\right)\left(\ell_3-3\right)\left(\ell_3-4\right)}{n(n-1)(n-2)(n-3)(n-4)}
	\\&\qquad +2\frac{\ell_3\ell^2_4}{n^3}\cdot\frac{\ell_1\left(\ell_1-1\right)\left(\ell_2-2\right)\left(\ell_2-3\right)\left(\ell_3-4\right)}{n(n-1)(n-2)(n-3)(n-4)}
	\\&\qquad  +4\frac{\ell_1\ell_2\ell_3}{n^3}\cdot\frac{\ell_1\left(\ell_2-1\right)\left(\ell_3-2\right)\left(\ell_4-3\right)\left(\ell_4-4\right)}{n(n-1)(n-2)(n-3)(n-4)} + 4\frac{\ell_1\ell_2\ell_3}{n^3}\cdot\frac{\ell_1\left(\ell_2-1\right)\left(\ell_3-2\right)\left(\ell_4-3\right)}{n(n-1)(n-2)(n-3)}
	\\&\qquad  +6\frac{\ell_1\ell_2\ell_4}{n^3}\cdot\frac{\ell_1\left(\ell_2-1\right)\left(\ell_3-2\right)\left(\ell_3-3\right)\left(\ell_4-4\right)}{n(n-1)(n-2)(n-3)(n-4)} + 2\frac{\ell_1\ell_2\ell_4}{n^3}\cdot\frac{\ell_1\left(\ell_2-1\right)\left(\ell_3-2\right)\left(\ell_3-3\right)}{n(n-1)(n-2)(n-3)}
	\\&\qquad  +4\frac{\ell_2\ell_3\ell_4}{n^3}\cdot\frac{\ell_1\left(\ell_1-1\right)\left(\ell_2-2\right)\left(\ell_3-3\right)\left(\ell_4-4\right)}{n(n-1)(n-2)(n-3)(n-4)} + 4\frac{\ell_2\ell_3\ell_4}{n^3}\cdot\frac{\ell_1\left(\ell_1-1\right)\left(\ell_2-2\right)\left(\ell_3-3\right)}{n(n-1)(n-2)(n-3)}
	\\&\qquad  +6\frac{\ell_1\ell_3\ell_4}{n^3}\cdot\frac{\ell_1\left(\ell_2-1\right)\left(\ell_2-2\right)\left(\ell_3-3\right)\left(\ell_4-4\right)}{n(n-1)(n-2)(n-3)(n-4)} +2\frac{\ell_1\ell_3\ell_4}{n^3}\cdot\frac{\ell_1\left(\ell_2-1\right)\left(\ell_2-2\right)\left(\ell_3-3\right)}{n(n-1)(n-2)(n-3)} .
\end{align*}
\item
When $\ds{\zp[b]{B}=6}$, there are $\binom{8}{2}=28$ terms arising
\begin{align*}\sum\limits_{\substack{B\subset S_8 \\ \zp[b]{B}=6 }} E^{(S_4,B)}_{n,\ell} &= \frac{\ell^2_1}{n^2}\cdot\frac{\ell_2\left(\ell_2-1\right)\left(\ell_3-2\right)\left(\ell_3-3\right)\left(\ell_4-4\right)\left(\ell_4-5\right) }{n(n-1)(n-2)(n-3)(n-4)(n-5)} 
	\\&\quad+{\frac{\ell^2_2}{n^2}\cdot\frac{\ell_1\left(\ell_1-1\right)\left(\ell_3-2\right)\left(\ell_3-3\right)\left(\ell_4-4\right)}{n(n-1)(n-2)(n-3)(n-4)} }
	\\&\quad+{\frac{\ell^2_3}{n^2}\cdot\frac{\ell_1\left(\ell_1-1\right)\left(\ell_2-2\right)\left(\ell_2-3\right)\left(\ell_4-4\right)}{n(n-1)(n-2)(n-3)(n-4)} }
	\\&\quad+\frac{\ell^2_4}{n^2}\cdot\frac{\ell_1\left(\ell_1-1\right)\left(\ell_2-2\right)\left(\ell_2-3\right)\left(\ell_3-4\right)\left(\ell_3-5\right) }{n(n-1)(n-2)(n-3)(n-4)(n-5)} 
	\\&\quad +2\frac{\ell_1\ell_2}{n^2}\cdot\frac{\ell_1\left(\ell_2-1\right)\left(\ell_3-2\right)\left(\ell_3-3\right)\left(\ell_4-4\right)\left(\ell_4-5\right) }{n(n-1)(n-2)(n-3)(n-4)(n-5)}
	\\&\quad
	+2\frac{\ell_1\ell_2}{n^2}\cdot\frac{\ell_1\left(\ell_2-1\right)\left(\ell_3-2\right)\left(\ell_3-3\right)\left(\ell_4-4\right) }{n(n-1)(n-2)(n-3)(n-4)} 
	\\&\quad +2\frac{\ell_1\ell_3}{n^2}\cdot\frac{\ell_1\left(\ell_2-1\right)\left(\ell_2-2\right)\left(\ell_3-3\right)\left(\ell_4-4\right)\left(\ell_4-5\right) }{n(n-1)(n-2)(n-3)(n-4)(n-5)} 
	\\&\quad
	+2\frac{\ell_1\ell_3}{n^2}\cdot\frac{\ell_1\left(\ell_2-1\right)\left(\ell_2-2\right)\left(\ell_3-3\right)\left(\ell_4-4\right) }{n(n-1)(n-2)(n-3)(n-4)} 
	\\&\quad +3\frac{\ell_1\ell_4}{n^2}\cdot\frac{\ell_1\left(\ell_2-1\right)\left(\ell_2-2\right)\left(\ell_3-3\right)\left(\ell_3-4\right)\left(\ell_4-5\right) }{n(n-1)(n-2)(n-3)(n-4)(n-5)} 
	\\&\quad
	+\frac{\ell_1\ell_4}{n^2}\cdot\frac{\ell_1\left(\ell_2-1\right)\left(\ell_2-2\right)\left(\ell_3-3\right)\left(\ell_3-4\right) }{n(n-1)(n-2)(n-3)(n-4)} 
	\\&\quad+4\frac{\ell_2\ell_3}{n^2}\cdot\frac{\ell_1\left(\ell_1-1\right)\left(\ell_2-2\right)\left(\ell_3-3\right)\left(\ell_4-4\right) }{n(n-1)(n-2)(n-3)(n-4)} 
	\\&\quad+2\frac{\ell_2\ell_4}{n^2}\cdot\frac{\ell_1\left(\ell_1-1\right)\left(\ell_2-2\right)\left(\ell_3-3\right)\left(\ell_3-4\right)\left(\ell_4-5\right) }{n(n-1)(n-2)(n-3)(n-4)(n-5)} 
	\\&\quad
	+2\frac{\ell_2\ell_4}{n^2}\cdot\frac{\ell_1\left(\ell_1-1\right)\left(\ell_2-2\right)\left(\ell_3-3\right)\left(\ell_3-4\right) }{n(n-1)(n-2)(n-3)(n-4)} 
	\\&\quad +2\frac{\ell_3\ell_4}{n^2}\cdot\frac{\ell_1\left(\ell_1-1\right)\left(\ell_2-2\right)\left(\ell_2-3\right)\left(\ell_3-4\right)\left(\ell_4-5\right) }{n(n-1)(n-2)(n-3)(n-4)(n-5)} 
	\\&\quad  
	+2\frac{\ell_3\ell_4}{n^2}\cdot\frac{\ell_1\left(\ell_1-1\right)\left(\ell_2-2\right)\left(\ell_2-3\right)\left(\ell_3-4\right)}{n(n-1)(n-2)(n-3)(n-4)} .
\end{align*}
\item
When $\ds{\zp[b]{B}=7}$, there are $\binom{8}{1}=8$ terms arising
\begin{align*}
-\sum\limits_{\substack{B\subset S_8 \\ \zp[b]{B}=7 }} E^{(S_4,B)}_{n,\ell}
&= {\frac{\ell_1}{n}\cdot\frac{\ell_1\left( \ell_2-1\right)\left( \ell_2-2\right)  \left( \ell_3-3\right) \left( \ell_3-4\right) \left( \ell_4-5\right) \left( \ell_4-6\right) }{n(n-1)(n-2)(n-3)(n-4)(n-5)(n-6)}}
	\\&\quad+{\frac{\ell_1}{n}\cdot\frac{\ell_1\left( \ell_2-1\right)\left( \ell_2-2\right)  \left( \ell_3-3\right) \left( \ell_3-4\right) \left( \ell_4-5\right) }{n(n-1)(n-2)(n-3)(n-4)(n-5)}}
	\\&\quad + {2\frac{\ell_2}{n}\cdot\frac{\ell_1\left( \ell_1-1\right)\left( \ell_2-2\right)  \left( \ell_3-3\right) \left( \ell_3-4\right) \left( \ell_4-5\right)  }{n(n-1)(n-2)(n-3)(n-4)(n-5)}}
	\\&\quad +{2\frac{\ell_3}{n}\cdot\frac{\ell_1\left( \ell_1-1\right)\left( \ell_2-2\right)  \left( \ell_2-3\right) \left( \ell_3-4\right) \left( \ell_4-5\right)  }{n(n-1)(n-2)(n-3)(n-4)(n-5)}}
	\\&\quad +{\frac{\ell_4}{n}\cdot\frac{\ell_1\left( \ell_1-1\right)\left( \ell_2-2\right)  \left( \ell_2-3\right) \left( \ell_3-4\right) \left( \ell_3-5\right) \left( \ell_4-6\right) }{n(n-1)(n-2)(n-3)(n-4)(n-5)(n-6)}}
	\\&\quad +{\frac{\ell_4}{n}\cdot\frac{\ell_1\left( \ell_1-1\right)\left( \ell_2-2\right)  \left( \ell_2-3\right) \left( \ell_3-4\right) \left( \ell_3-5\right) }{n(n-1)(n-2)(n-3)(n-4)(n-5)}}.
\end{align*}
\item
When $\ds{\zp[b]{B}=8}$, i.e., $\ds{B=S_8}$, we have
 \begin{align*}\sum\limits_{\substack{B\subset S_8 \\ \zp[b]{B}=8 }}  E^{(S_4,B)}_{n,\ell}
&=\frac{ \ell_{1}\left( \ell_{1}-1\right) \left(\ell_{2}-2 \right)\left(\ell_{2}-3 \right) \left(\ell_{3}-4\right)\left(\ell_{3}-5\right)\left(\ell_{4}-6 \right)}{n(n-1)\cdots (n-6)}.
\end{align*}
\end{asparaenum}}

Assembling terms, one obtains 
	\[
	 \sum_{B\subset S_8}  E^{(S_4,B)}_{n,\ell}  = \frac{\ell_1 \left( n-\ell_4 \right) \left( \sum_{j=0}^6 n^j p_j\left( \ell_1,\ell_2,\ell_3,\ell_4 \right)  \right) }{n^8 (n-1)(n-2)(n-3)(n-4)(n-5)(n-6)},
	\]	where, for $\ds{j\in \{0,\ldots,6\}\setminus\{3\}}$, $\ds{p_j \in\R_{\left( 9-j\right)\wedge 6 }\left[X_1,\ldots,X_4\right]}$, and $p_3\in\R_{5}\left[X_1,\ldots,X_4\right]$.
	As a consequence, the numerator is of the order $\ds{O(n^{11})}$, which implies \eqref{eq:goal74}. 
	\qed

\subsection{Proof of \eqref{eq:summand_rate} for $\kappa_{A,i}= -4$.}\label{sect:ka-4} 
We need to consider the cases $(|\mathbf i|,|A|)=(8,4)$  for which we may assume that $\mathbf i=(1, \dots, 8)$ and $A=S_4=\{ 1,2,3,4 \}$, such that $\psi^{-1}(A)=S_8=\{1, \dots, 8\}$.
As before, we further assume that $\ell_1 \le \ell_2 \le \ell_3 \le \ell_4$. We proceed similarly as for the case $(|\mathbf i|,|A|)=(8,3)$. Recall that we must prove that
\begin{align} \label{eq:goal84}
\sum_{B\subset\psi^{-1}(A) }  E^{(A,B)}_{n, \bm \ell} 
= 
\sum_{\sigma=0}^8 \sum_{B\subset S_8: |B|=\sigma} E^{(S_4,B)}_{n, \bm \ell}  = O(n^{-4}).
\end{align}
For each $\sigma\in\{0,\dots, 8\}$, the sums in the previous decomposition may be calculated explicitly (all formulas have been checked with computer algebra systems):
{\small
\begin{asparaenum}
\item 
When $\ds{B=\emptyset}$, we get: $$\sum\limits_{\substack{B\subset S_8 \\ \zp[b]{B}=0 }}  E^{(S_4,B)}_{n,\ell} = E^{(S_4,\emptyset)}_n = \frac{\ell^2_1\ell^2_2\ell^2_3\ell^2_4}{n^8}.$$
\item
When $\ds{\zp[b]{B}=1}$, i.e. $\ds{B= \{j\}}$ with $\ds{1\leq j \leq 8}$, there are $\binom{8}{1}=8$ terms arising
\begin{align*}
-\sum\limits_{\substack{B\subset S_8 \\ \zp[b]{B}=1 }}  E^{(S_4,B)}_{n,\ell} &= 8\cdot\frac{\ell^2_1\ell^2_2\ell^2_3\ell^2_4}{n^8}.
\end{align*}
\item
When $\ds{\zp[b]{B}=2}$,  there are $\binom{8}{2}=28$ terms arising
\begin{align*}\sum\limits_{\substack{B\subset S_8 \\ \zp[b]{B}=2 }}  E^{(S_4,B)}_{n,\ell} &= \frac{\ell^2_2\ell^2_3\ell^2_4}{n^6} \cdot \frac{\ell_1(\ell_1-1)}{n(n-1)}+\frac{\ell^2_1\ell^2_3\ell^2_4}{n^6} \cdot \frac{\ell_2(\ell_2-1)}{n(n-1)}+\frac{\ell^2_1\ell^2_2\ell^2_4}{n^6} \cdot \frac{\ell_3(\ell_3-1)}{n(n-1)}+\frac{\ell^2_1\ell^2_2\ell^2_3}{n^6} \cdot \frac{\ell_4(\ell_4-1)}{n(n-1)}
	\\&\quad +4\frac{\ell_1\ell_2\ell^2_3\ell^2_4}{n^6} \cdot \frac{\ell_1(\ell_2-1)}{n(n-1)}+4\frac{\ell_1\ell^2_2\ell_3\ell^2_4}{n^6} \cdot \frac{\ell_1(\ell_3-1)}{n(n-1)}+4\frac{\ell_1\ell^2_2\ell^2_3\ell_4}{n^6} \cdot \frac{\ell_1(\ell_4-1)}{n(n-1)}
	\\&\quad + 4\frac{\ell^2_1\ell_2\ell_3\ell^2_4}{n^6} \cdot \frac{\ell_2(\ell_3-1)}{n(n-1)} + 4\frac{\ell^2_1\ell_2\ell^2_3\ell_4}{n^6} \cdot \frac{\ell_2(\ell_4-1)}{n(n-1)} +  4\frac{\ell^2_1\ell^2_2\ell_3\ell_4}{n^6} \cdot \frac{\ell_3(\ell_4-1)}{n(n-1)}.
\end{align*}
\item
When $\ds{\zp[b]{B}=3}$, there are $\binom{8}{3}=56$ terms arising:
\begin{align*}-\sum\limits_{\substack{B\subset S_8 \\ \zp[b]{B}=3 }}  E^{(S_4,B)}_{n,\ell} 
&=2\frac{\ell_2 \ell^2_3 \ell^2_4}{n^5}\cdot\frac{\ell_1\left(\ell_1-1\right)\left(\ell_2-2\right) }{n(n-1)(n-2)}
	\\&\qquad +2\frac{\ell^2_2 \ell_3 \ell^2_4}{n^5}\cdot\frac{\ell_1\left(\ell_1-1\right)\left(\ell_3-2\right) }{n(n-1)(n-2)}
	\\&\qquad +2\frac{\ell^2_2 \ell^2_3 \ell_4}{n^5}\cdot\frac{\ell_1\left(\ell_1-1\right)\left(\ell_4-2\right) }{n(n-1)(n-2)}
	\\&\qquad +2\frac{\ell_1 \ell^2_3 \ell^2_4}{n^5}\cdot\frac{\ell_1\left(\ell_2-1\right)\left(\ell_2-2\right) }{n(n-1)(n-2)}
	\\&\qquad +2\frac{\ell^2_1 \ell_3 \ell^2_4}{n^5}\cdot\frac{\ell_2\left(\ell_2-1\right)\left(\ell_3-2\right) }{n(n-1)(n-2)}
	\\&\qquad +2\frac{\ell^2_1 \ell^2_3 \ell_4}{n^5}\cdot\frac{\ell_2\left(\ell_2-1\right)\left(\ell_4-2\right) }{n(n-1)(n-2)}
	\\&\qquad +2\frac{\ell_1 \ell^2_2 \ell^2_4}{n^5}\cdot\frac{\ell_1\left(\ell_3-1\right)\left(\ell_3-2\right) }{n(n-1)(n-2)}
	\\&\qquad +2\frac{\ell^2_1 \ell_2 \ell^2_4}{n^5}\cdot\frac{\ell_2\left(\ell_3-1\right)\left(\ell_3-2\right) }{n(n-1)(n-2)}
	\\&\qquad +2\frac{\ell^2_1 \ell^2_2 \ell_4}{n^5}\cdot\frac{\ell_3\left(\ell_3-1\right)\left(\ell_4-2\right) }{n(n-1)(n-2)}
	\\&\qquad +2\frac{\ell_1 \ell^2_2 \ell^2_3}{n^5}\cdot\frac{\ell_1\left(\ell_4-1\right)\left(\ell_4-2\right) }{n(n-1)(n-2)}
	\\&\qquad +2\frac{\ell^2_1 \ell_2 \ell^2_3}{n^5}\cdot\frac{\ell_2\left(\ell_4-1\right)\left(\ell_4-2\right) }{n(n-1)(n-2)}
	\\&\qquad +2\frac{\ell^2_1 \ell^2_2 \ell_3}{n^5}\cdot\frac{\ell_3\left(\ell_4-1\right)\left(\ell_4-2\right) }{n(n-1)(n-2)}
	\\&\qquad +8\frac{\ell^2_1 \ell_2 \ell_3 \ell_4}{n^5}\cdot\frac{\ell_2\left(\ell_3-1\right)\left(\ell_4-2\right) }{n(n-1)(n-2)}
	\\&\qquad +8\frac{\ell_1 \ell^2_2 \ell_3 \ell_4}{n^5}\cdot\frac{\ell_1\left(\ell_3-1\right)\left(\ell_4-2\right) }{n(n-1)(n-2)}
	\\&\qquad +8\frac{\ell_1 \ell_2 \ell^2_3 \ell_4}{n^5}\cdot\frac{\ell_1\left(\ell_2-1\right)\left(\ell_4-2\right) }{n(n-1)(n-2)}
	\\&\qquad +8\frac{\ell_1 \ell_2 \ell_3 \ell^2_4}{n^5}\cdot\frac{\ell_1\left(\ell_2-1\right)\left(\ell_3-2\right) }{n(n-1)(n-2)}.
\end{align*}
\item
When $\ds{\zp[b]{B}=4}$, there are $\binom{8}{4}=70$ terms arising
\begin{align*}\sum\limits_{\substack{B\subset S_8 \\ \zp[b]{B}=4 }} E^{(S_4,B)}_{n,\ell} 
&= \frac{\ell^2_1 \ell^2_2}{n^4}\cdot\frac{\ell_3\left(\ell_3-1\right)\left(\ell_4-2\right)\left(\ell_4-3\right)}{n(n-1)(n-2)(n-3)}
	\\&\qquad +\frac{\ell^2_1 \ell^2_3}{n^4}\cdot\frac{\ell_2\left(\ell_2-1\right)\left(\ell_4-2\right)\left(\ell_4-3\right)}{n(n-1)(n-2)(n-3)}
	\\&\qquad +\frac{\ell^2_1 \ell^2_4}{n^4}\cdot\frac{\ell_2\left(\ell_2-1\right)\left(\ell_3-2\right)\left(\ell_3-3\right)}{n(n-1)(n-2)(n-3)}
	\\&\qquad +\frac{\ell^2_2 \ell^2_3 }{n^4}\cdot\frac{\ell_1\left(\ell_1-1\right)\left(\ell_4-2\right)\left(\ell_4-3\right)}{n(n-1)(n-2)(n-3)}
	\\&\qquad +\frac{\ell^2_2 \ell^2_4}{n^4}\cdot\frac{\ell_1\left(\ell_1-1\right)\left(\ell_3-2\right)\left(\ell_3-3\right)}{n(n-1)(n-2)(n-3)}
	\\&\qquad +\frac{\ell^2_3\ell^2_4}{n^4}\cdot\frac{\ell_1\left(\ell_1-1\right)\left(\ell_2-2\right)\left(\ell_2-3\right)}{n(n-1)(n-2)(n-3)}
	\\&\qquad +16\frac{\ell_1 \ell_2 \ell_3 \ell_4}{n^4}\cdot\frac{\ell_1\left(\ell_2-1\right)\left(\ell_3-2\right)\left(\ell_4-3\right)}{n(n-1)(n-2)(n-3)}
	\\&\qquad +4\frac{\ell^2_1 \ell_2 \ell_3 }{n^4}\cdot\frac{\ell_2\left(\ell_3-1\right)\left(\ell_4-2\right)\left(\ell_4-3\right)}{n(n-1)(n-2)(n-3)}
	\\&\qquad +4\frac{\ell^2_1 \ell_2 \ell_4 }{n^4}\cdot\frac{\ell_2\left(\ell_3-1\right)\left(\ell_3-2\right)\left(\ell_4-3\right)}{n(n-1)(n-2)(n-3)}
	\\&\qquad +4\frac{\ell^2_1 \ell_3 \ell_4}{n^4}\cdot\frac{\ell_2\left(\ell_2-1\right)\left(\ell_3-2\right)\left(\ell_4-3\right)}{n(n-1)(n-2)(n-3)}
	\\&\qquad +4\frac{\ell_1 \ell^2_2 \ell_3 }{n^4}\cdot\frac{\ell_1\left(\ell_3-1\right)\left(\ell_4-2\right)\left(\ell_4-3\right)}{n(n-1)(n-2)(n-3)}
	\\&\qquad +4\frac{\ell_1 \ell^2_2 \ell_4 }{n^4}\cdot\frac{\ell_1\left(\ell_3-1\right)\left(\ell_3-2\right)\left(\ell_4-3\right)}{n(n-1)(n-2)(n-3)}
	\\&\qquad +4\frac{\ell^2_2 \ell_3 \ell_4}{n^4}\cdot\frac{\ell_1\left(\ell_1-1\right)\left(\ell_3-2\right)\left(\ell_4-3\right)}{n(n-1)(n-2)(n-3)}
	\\&\qquad +4\frac{\ell_1 \ell_2 \ell^2_3}{n^4}\cdot\frac{\ell_1\left(\ell_2-1\right)\left(\ell_4-2\right)\left(\ell_4-3\right)}{n(n-1)(n-2)(n-3)}
	\\&\qquad +4\frac{\ell_1 \ell^2_3 \ell_4 }{n^4}\cdot\frac{\ell_1\left(\ell_2-1\right)\left(\ell_2-2\right)\left(\ell_4-3\right)}{n(n-1)(n-2)(n-3)}
	\\&\qquad +4\frac{\ell_2 \ell^2_3 \ell_4 }{n^4}\cdot\frac{\ell_1\left(\ell_1-1\right)\left(\ell_2-2\right)\left(\ell_4-3\right)}{n(n-1)(n-2)(n-3)}
	\\&\qquad +4\frac{\ell_1 \ell_2 \ell^2_4}{n^4}\cdot\frac{\ell_1\left(\ell_2-1\right)\left(\ell_3-2\right)\left(\ell_3-3\right)}{n(n-1)(n-2)(n-3)}
	\\&\qquad +4\frac{\ell_1 \ell_3 \ell^2_4}{n^4}\cdot\frac{\ell_1\left(\ell_2-1\right)\left(\ell_2-2\right)\left(\ell_3-3\right)}{n(n-1)(n-2)(n-3)}
	\\&\qquad +4\frac{\ell_2 \ell_3 \ell^2_4}{n^4}\cdot\frac{\ell_1\left(\ell_1-1\right)\left(\ell_2-2\right)\left(\ell_3-3\right)}{n(n-1)(n-2)(n-3)}.
\end{align*}
\item
When $\ds{\zp[b]{B}=5}$, there are $\binom{8}{5}=56$ terms arising
\begin{align*}-\sum\limits_{\substack{B\subset S_8 \\ \zp[b]{B}=5 }}  E^{(S_4,B)}_{n,\ell}&=2\frac{\ell^2_1\ell_2}{n^3}\cdot\frac{\ell_2\left(\ell_3-1\right)\left(\ell_3-2\right)\left(\ell_4-3\right)\left(\ell_4-4\right)}{n(n-1)(n-2)(n-3)(n-4)}
	\\&\qquad +2\frac{\ell^2_1\ell_3}{n^3}\cdot\frac{\ell_2\left(\ell_2-1\right)\left(\ell_3-2\right)\left(\ell_4-3\right)\left(\ell_4-4\right)}{n(n-1)(n-2)(n-3)(n-4)}
	\\&\qquad +2\frac{\ell^2_1\ell_4}{n^3}\cdot\frac{\ell_2\left(\ell_2-1\right)\left(\ell_3-2\right)\left(\ell_3-3\right)\left(\ell_4-4\right)}{n(n-1)(n-2)(n-3)(n-4)}
	\\&\qquad +2\frac{\ell_1\ell^2_2}{n^3}\cdot\frac{\ell_1\left(\ell_3-1\right)\left(\ell_3-2\right)\left(\ell_4-3\right)\left(\ell_4-4\right)}{n(n-1)(n-2)(n-3)(n-4)}
	\\&\qquad +2\frac{\ell^2_2\ell_3}{n^3}\cdot\frac{\ell_1\left(\ell_1-1\right)\left(\ell_3-2\right)\left(\ell_4-3\right)\left(\ell_4-4\right)}{n(n-1)(n-2)(n-3)(n-4)}
	\\&\qquad +2\frac{\ell^2_2\ell_4}{n^3}\cdot\frac{\ell_1\left(\ell_1-1\right)\left(\ell_3-2\right)\left(\ell_3-3\right)\left(\ell_4-4\right)}{n(n-1)(n-2)(n-3)(n-4)} 
	\\&\qquad +2\frac{\ell_1\ell^2_3}{n^3}\cdot\frac{\ell_1\left(\ell_2-1\right)\left(\ell_2-2\right)\left(\ell_4-3\right)\left(\ell_4-4\right)}{n(n-1)(n-2)(n-3)(n-4)}
	\\&\qquad +2\frac{\ell_2\ell^2_3}{n^3}\cdot\frac{\ell_1\left(\ell_1-1\right)\left(\ell_2-2\right)\left(\ell_4-3\right)\left(\ell_4-4\right)}{n(n-1)(n-2)(n-3)(n-4)}
	\\&\qquad +2\frac{\ell^2_3\ell_4}{n^3}\cdot\frac{\ell_1\left(\ell_1-1\right)\left(\ell_2-2\right)\left(\ell_2-3\right)\left(\ell_4-4\right)}{n(n-1)(n-2)(n-3)(n-4)}
	\\&\qquad +2\frac{\ell_1\ell^2_4}{n^3}\cdot\frac{\ell_1\left(\ell_2-1\right)\left(\ell_2-2\right)\left(\ell_3-3\right)\left(\ell_3-4\right)}{n(n-1)(n-2)(n-3)(n-4)}
	\\&\qquad +2\frac{\ell_2\ell^2_4}{n^3}\cdot\frac{\ell_1\left(\ell_1-1\right)\left(\ell_2-2\right)\left(\ell_3-3\right)\left(\ell_3-4\right)}{n(n-1)(n-2)(n-3)(n-4)}
	\\&\qquad +2\frac{\ell_3\ell^2_4}{n^3}\cdot\frac{\ell_1\left(\ell_1-1\right)\left(\ell_2-2\right)\left(\ell_2-3\right)\left(\ell_3-4\right)}{n(n-1)(n-2)(n-3)(n-4)}
	\\&\qquad +8\frac{\ell_1\ell_2\ell_3}{n^3}\cdot\frac{\ell_1\left(\ell_2-1\right)\left(\ell_3-2\right)\left(\ell_4-3\right)\left(\ell_4-4\right)}{n(n-1)(n-2)(n-3)(n-4)}
	\\&\qquad +8\frac{\ell_1\ell_2\ell_4}{n^3}\cdot\frac{\ell_1\left(\ell_2-1\right)\left(\ell_3-2\right)\left(\ell_3-3\right)\left(\ell_4-4\right)}{n(n-1)(n-2)(n-3)(n-4)}
	\\&\qquad +8\frac{\ell_2\ell_3\ell_4}{n^3}\cdot\frac{\ell_1\left(\ell_1-1\right)\left(\ell_2-2\right)\left(\ell_3-3\right)\left(\ell_4-4\right)}{n(n-1)(n-2)(n-3)(n-4)}
	\\&\qquad +8\frac{\ell_1\ell_3\ell_4}{n^3}\cdot\frac{\ell_1\left(\ell_2-1\right)\left(\ell_2-2\right)\left(\ell_3-3\right)\left(\ell_4-4\right)}{n(n-1)(n-2)(n-3)(n-4)}.
\end{align*}
\item
When $\ds{\zp[b]{B}=6}$, there are $\binom{8}{2}=28$ terms arising
\begin{align*}\sum\limits_{\substack{B\subset S_8 \\ \zp[b]{B}=6 }} E^{(S_4,B)}_{n,\ell} &= \frac{\ell^2_1}{n^2}\cdot\frac{\ell_2\left(\ell_2-1\right)\left(\ell_3-2\right)\left(\ell_3-3\right)\left(\ell_4-4\right)\left(\ell_4-5\right) }{n(n-1)(n-2)(n-3)(n-4)(n-5)} 
	\\&\qquad+\frac{\ell^2_2}{n^2}\cdot\frac{\ell_1\left(\ell_1-1\right)\left(\ell_3-2\right)\left(\ell_3-3\right)\left(\ell_4-4\right)\left(\ell_4-5\right) }{n(n-1)(n-2)(n-3)(n-4)(n-5)} 
	\\&\qquad+\frac{\ell^2_3}{n^2}\cdot\frac{\ell_1\left(\ell_1-1\right)\left(\ell_2-2\right)\left(\ell_2-3\right)\left(\ell_4-4\right)\left(\ell_4-5\right) }{n(n-1)(n-2)(n-3)(n-4)(n-5)} 
	\\&\qquad+\frac{\ell^2_4}{n^2}\cdot\frac{\ell_1\left(\ell_1-1\right)\left(\ell_2-2\right)\left(\ell_2-3\right)\left(\ell_3-4\right)\left(\ell_3-5\right) }{n(n-1)(n-2)(n-3)(n-4)(n-5)} 
	\\&\qquad+4\frac{\ell_1\ell_2}{n^2}\cdot\frac{\ell_1\left(\ell_2-1\right)\left(\ell_3-2\right)\left(\ell_3-3\right)\left(\ell_4-4\right)\left(\ell_4-5\right) }{n(n-1)(n-2)(n-3)(n-4)(n-5)} 
	\\&\qquad+4\frac{\ell_1\ell_3}{n^2}\cdot\frac{\ell_1\left(\ell_2-1\right)\left(\ell_2-2\right)\left(\ell_3-3\right)\left(\ell_4-4\right)\left(\ell_4-5\right) }{n(n-1)(n-2)(n-3)(n-4)(n-5)} 
	\\&\qquad+4\frac{\ell_1\ell_4}{n^2}\cdot\frac{\ell_1\left(\ell_2-1\right)\left(\ell_2-2\right)\left(\ell_3-3\right)\left(\ell_3-4\right)\left(\ell_4-5\right) }{n(n-1)(n-2)(n-3)(n-4)(n-5)} 
	\\&\qquad+4\frac{\ell_2\ell_3}{n^2}\cdot\frac{\ell_1\left(\ell_1-1\right)\left(\ell_2-2\right)\left(\ell_3-3\right)\left(\ell_4-4\right)\left(\ell_4-5\right) }{n(n-1)(n-2)(n-3)(n-4)(n-5)} 
	\\&\qquad+4\frac{\ell_2\ell_4}{n^2}\cdot\frac{\ell_1\left(\ell_1-1\right)\left(\ell_2-2\right)\left(\ell_3-3\right)\left(\ell_3-4\right)\left(\ell_4-5\right) }{n(n-1)(n-2)(n-3)(n-4)(n-5)} 
	\\&\qquad+4\frac{\ell_3\ell_4}{n^2}\cdot\frac{\ell_1\left(\ell_1-1\right)\left(\ell_2-2\right)\left(\ell_2-3\right)\left(\ell_3-4\right)\left(\ell_4-5\right) }{n(n-1)(n-2)(n-3)(n-4)(n-5)} .
\end{align*}
\item
When $\ds{\zp[b]{B}=7}$, there are $\binom{8}{1}=8$ terms arising: 
\begin{align*}-\sum\limits_{\substack{B\subset S_8 \\ \zp[b]{B}=7 }} E^{(S_4,B)}_{n,\ell}&= 2\frac{\ell_1}{n}\cdot\frac{\ell_1\left( \ell_2-1\right)\left( \ell_2-2\right)  \left( \ell_3-3\right) \left( \ell_3-4\right) \left( \ell_4-5\right) \left( \ell_4-6\right) }{n(n-1)(n-2)(n-3)(n-4)(n-5)(n-6)}
	\\&\qquad + 2\frac{\ell_2}{n}\cdot\frac{\ell_1\left( \ell_1-1\right)\left( \ell_2-2\right)  \left( \ell_3-3\right) \left( \ell_3-4\right) \left( \ell_4-5\right) \left( \ell_4-6\right) }{n(n-1)(n-2)(n-3)(n-4)(n-5)(n-6)}
	\\&\qquad +2\frac{\ell_3}{n}\cdot\frac{\ell_1\left( \ell_1-1\right)\left( \ell_2-2\right)  \left( \ell_2-3\right) \left( \ell_3-4\right) \left( \ell_4-5\right) \left( \ell_4-6\right) }{n(n-1)(n-2)(n-3)(n-4)(n-5)(n-6)}
	\\&\qquad +2\frac{\ell_4}{n}\cdot\frac{\ell_1\left( \ell_1-1\right)\left( \ell_2-2\right)  \left( \ell_2-3\right) \left( \ell_3-4\right) \left( \ell_3-5\right) \left( \ell_4-6\right) }{n(n-1)(n-2)(n-3)(n-4)(n-5)(n-6)}.
\end{align*}
\item
In the case $\ds{\zp[b]{B}=8}$, i.e., $\ds{B=S_8}$, one has $\ds{\text{rank}_B(j)=j}$, which implies 
\begin{align*}\sum\limits_{\substack{B\subset S_8 \\ \zp[b]{B}=8 }}  E^{(S_4,B)}_{n,\ell}
=\frac{ \ell_{1}\left( \ell_{1}-1\right) \left(\ell_{2}-2 \right)\left(\ell_{2}-3 \right) \left(\ell_{3}-4\right)\left(\ell_{3}-5\right)\left(\ell_{4}-6 \right)\left(\ell_{4}-7 \right) }{n(n-1)\cdots (n-7)}.
\end{align*}
\end{asparaenum}}

Assembling terms, we obtain
\begin{align*} 
\sum_{B\subset S_8}   E^{(S_4,B)}_{n,\ell}
	&=  \frac{\ell_1 \left( n-\ell_4 \right) \left( \sum_{j=0}^6 n^j p_j\left( \ell_1,\ldots,\ell_4 \right)  \right) }{n^8 (n-1)(n-2)(n-3)(n-4)(n-5)(n-6)(n-7)},
\end{align*}
where, for $\ds{0\leq j \leq 6}$, $\ds{p_j \in\R_{\left( 9-j\right)\wedge 6 }\left[X_1,\ldots,X_4\right]}$. 
As a consequence, the numerator is of the order $\ds{O(n^{11})}$, which implies \eqref{eq:goal84}. \qed

 \section{Results and proofs for Step 3: joint convergence} \label{sec:ps3b}

\begin{proof}[Proof of Proposition~\ref{prop:mm}]
For notational convenience, we only cover the case $m=3$ (note that the case $m=2$ was proven in Section~\ref{sec:ps3a}). By the Cram\'er-Wold device, it is sufficient to show that 
\begin{align} \label{eq:tbs3}
\lambda_2 \frac{\tilde M_n(2)}{\delta_n(2)} + \lambda_3 \frac{\tilde M_n(3)}{\delta_n(3)}
&\weak
\mathcal{N}_1\left( 0,\lambda_2^2+ \lambda_3^2\right) 
\end{align}
for all $\lambda_2, \lambda_3 \in \R$. For notational convenience, we restrict attention to $\lambda_2 = \lambda_3=1$.

For that purpose, recalling $\tilde M_n(k)$ from \eqref{eq:mnnn2}, observe that we can write
\begin{align*}
\frac{\tilde M_n(2)}{\delta_n(2)}+\frac{\tilde M_n(3)}{\delta_n(3)}
= 
\delta^{-1}_n(3)\cdot  \sum_{\ell=1}^{d} \sum_{q=1}^{\ell-1} 
\left[\frac{\delta_n(3)}{\delta_n(2)}\cdot   \tilde M_{n,\{ q,\ell \}}+\sum_{p=1}^{q-1} \tilde M_{n,\{ p,q,\ell \}}\right].
\end{align*}
 We are going to apply Theorem \ref{TCL} with $\eta=2$, with $ \mathcal{F}_{n,r} := \sigma\{U_{i,p} : i\in\{1, \dots, n\},p\in\{1, \dots, r\} \}$ as in \eqref{eq:fnr} and with 
\begin{align}
X_{n,r} 
&:=  \label{eq:xnr-mult}
\begin{dcases}
\delta^{-1}_n(3)\cdot \sum_{q =1}^{r-1} 
\left[\frac{\delta_n(3)}{\delta_n(2)}\cdot   \tilde M_{n,\{ q,r \}}+\sum_{p=1}^{q-1} \tilde M_{n,\{ p,q,r \}}\right] &, r \ge 2 \\
0 & , r=1  ,
\end{dcases}
\end{align}
where the empty sum is defined as zero (which is obtained for $q=1$).
Throughout this section, we always refer to the definition of $X_{n,r}$ in \eqref{eq:xnr-mult} rather than in \eqref{eq:xnr}, unless mentioned otherwise.

Clearly, $S_{n,r}$ has zero mean by 3-variate independence and centredness of $I_{i,j}^{(p)}$. Next, the martingale property is a mere consequence of
\begin{align*}
\E[X_{n,r} \mid \mathcal F_{n,r-1}]
&=
\frac2n \sum_{i<j}^n
\frac1{\delta_n(3)} \sum_{q =1}^{r-1} 
\left[\frac{\delta_n(3)}{\delta_n(2)}\cdot   \E[ \tilde I_{i,j}^{(q)}   \tilde I_{i,j}^{(r)}   \mid \mathcal F_{n,r-1} ]+\sum_{p=1}^{q-1}  \E[ \tilde I_{i,j}^{(p)}  \tilde I_{i,j}^{(q)}   \tilde I_{i,j}^{(r)}   \mid \mathcal F_{n,r-1} ] \right] \\
&=
\frac2n \sum_{i<j}^n
\frac1{\delta_n(3)} \sum_{q =1}^{r-1} 
\left[\frac{\delta_n(3)}{\delta_n(2)}\cdot   \tilde I_{i,j}^{(q)}  \E[  \tilde I_{i,j}^{(r)}  ]+\sum_{p=1}^{q-1}   \tilde I_{i,j}^{(p)}  \tilde I_{i,j}^{(q)}   \E[ \tilde I_{i,j}^{(r)}  ] \right],
\end{align*}
which is zero by centredness of $\tilde I_{i,j}^{(p)}$.

As in the proof of Proposition~\ref{prop_s1}, the condition in \eqref{lind2} with $\eta^2=2$ is a mere consequence of 
\[ 
\lim_{n\to\infty} 
\E \left[  \left( \sum_{r=1}^{d}\E\left( X^2_{n,r} | \mathcal{F}_{n,r-1}  \right) \right)^2 \right] = 4,
\qquad  
\lim_{n\to\infty} 
\E \left[ \sum_{r=1}^{d}\E\left( X^2_{n,r} | \mathcal{F}_{n,r-1}  \right)  \right]= 2.
\]
which follows from Lemma~\ref{lem_exp_sum} and \ref{lem_cond_exp_sum}.

Finally, since $L^1$-convergence implies convergence in probability and since $X_{nr}^4$ is non-negative, the Lyapunov condition (\ref{lyapunov}) is a consequence of Lemma~\ref{lem:lya3}. The weak convergence in \eqref{eq:tbs3} with $\lambda_2=\lambda_3=1$ is now a consequence of Theorem~\ref{TCL}.
\end{proof}

\begin{lemma}\label{lem_exp_sum}
Assume $5$-wise independence. Then, with $X_{n,r}$ from \eqref{eq:xnr-mult},
\begin{align*}
\lim_{n\to\infty}
\E \left[ \sum_{r=1}^{d}\E\left( X^2_{n,r} | \mathcal{F}_{n,r-1}  \right)  \right] =2.
\end{align*}
\end{lemma}

\begin{proof}
By definition of $X_{n,r}$ in \eqref{eq:xnr-mult}, we have
\begin{align}
X^2_{n,r}
=
\frac1{\delta_n^2(2)} \sum_{q,q'=1}^{r-1}  \tilde M_{n,\{ q,r \}} \tilde M_{n,\{ q',r \}} 
& \nonumber
+ \frac2{\delta_n(2) \delta_n(3)} \sum_{q,q'=1}^{r-1} \sum_{p'=1}^{q'-1} \tilde M_{n,\{ q,r \}}\tilde M_{n,\{ p',q',r \}}  \\
& \label{eq:xnr2-mult}
+ \frac1{\delta_n^2(3)}\sum_{q,q'=1}^{r-1} \sum_{p=1}^{q-1}\sum_{p'=1}^{q'-1}  \tilde M_{n,\{ p,q,r \}} \tilde M_{n,\{ p',q',r \}} .
\end{align}
Taking the expectation, by  $5$-wise independence and centredness of $\tilde I_{i,j}^{(p)}$, 
\begin{align*}
\E \left[X^{2}_{n,r}\right] 
&= 
\frac1{\delta_n^2(2)} \sum_{q=1}^{r-1} \E \left[\left( \tilde M_{n,\{ q,r \}}\right)^2 \right]
+ 
\frac1{\delta_n^2(3)}\sum_{q=1}^{r-1} \sum_{p=1}^{q-1}\E \left[\left( \tilde M_{n,\{ p,q,r \}}\right)^2  \right].
\end{align*}
The assertion now follows from Lemma~\ref{lem_exp}, applied with $k=2$ and $k=3$.	
\end{proof}

\begin{lemma}\label{lem_cond_exp_sum}
Assume $8$-wise independence. Then, with $X_{n,r}$ from \eqref{eq:xnr-mult},
\begin{align*}
\Lambda_n
\coloneqq
\E \left[ \left(\sum_{r=1}^{d}\E\left( X^2_{n,r} | \mathcal{F}_{n,r-1}  \right) \right)^2  \right]\xrightarrow[n\to +\infty]{} 4.
\end{align*}
\end{lemma}

\begin{proof} 
As in the proof of Lemma~\ref{lem_cond_exp}, we decompose
\[
\Lambda_n=\Lambda_{n,1}+\Lambda_{n,2},
\]
where
\[
\Lambda_{n,1} 
:=  
\sum_{r=1}^{d}\E \left[ \left( \E\left( X^2_{n,r} | \mathcal{F}_{n,r-1}  \right) \right)^2  \right] ,\quad 
\Lambda_{n,2} 
:= 
\sum_{r \ne r'}^d  \E \left[\E\left( X^2_{n,r} | \mathcal{F}_{n,r-1}  \right) \E\left( X^2_{n,r'} | \mathcal{F}_{n,r'-1}  \right)  \right]. 
\]
It is sufficient to show that
\begin{align} \label{eq:lconv-mult}
\lim_{n\to\infty}\Lambda_{n,1}=0, 
\qquad 
\lim_{n\to\infty}\Lambda_{n,2}=4.
\end{align}	
	
For that purpose, we start by decomposing, for $r\ge 2$, the random variable 
\begin{align} \label{eq:xidecomp}
\chi_{n,r} := \E\left( X^2_{n,r} | \mathcal{F}_{n,r-1}  \right)  =\sum_{  \ell =1}^3 \chi^{(\ell )}_{n,r},
\end{align}
where, in view of  \eqref{eq:xnr2-mult},
\begin{align*}
\chi^{(1)}_{n,r}
&:=
\frac1{\delta_n^2(2)} \sum_{q,q'=1}^{r-1} \E\left( \tilde M_{n,\{ q,r \}} \tilde M_{n,\{ q',r \}}  | \mathcal{F}_{n,r-1}  \right),\\
\chi^{(2)}_{n,r}
&:= 
\frac2{\delta_n(2) \delta_n(3)}  \sum_{q,q'=1}^{r-1} \sum_{p'=1}^{q'-1} \E \left(\tilde M_{n,\{ q,r \}} \tilde M_{n,\{ p',q',r \}}  | \mathcal{F}_{n,r-1} \right), \\
\chi^{(3)}_{n,r}
&:=
\frac1{\delta_n^2(3)} \sum_{q,q'=1}^{r-1} \sum_{p=1}^{q-1}\sum_{p'=1}^{q'-1} \E \left( \tilde M_{n,\{ p,q,r \}}\tilde M_{n,\{ p',q',r \}}  | \mathcal{F}_{n,r-1} \right).
\end{align*}
Here and throughout, empty sums are again interpreted as zero.
By $5$-wise independence and Lemma~\ref{lem:2m},
\begin{align*}
\chi^{(1)}_{n,r}
&=
\frac{4}{n^2\delta^{2}_n(2)} \sum_{q_1,q_2=1}^{r-1} \sum_{\mathbf{i} \in \mathcal{J}} 
\varphi_2\left( \mathbf{i} \right)  \tilde  I^{(q_1)}_{i_1,i_2}\tilde  I^{(q_2)}_{i_3,i_4}\\
\chi^{(2)}_{n,r} 
&=
\frac{8}{n^2\delta_n(2) \delta_n(3)} \sum_{q_1,q_2=1}^{r-1} \sum_{p_1=1}^{q_1-1} \sum_{\mathbf{i}\in \mathcal{J}}
\varphi_2\left( \mathbf{i} \right) \tilde I^{(p_1)}_{i_1,i_2}\tilde  I^{(q_1)}_{i_1,i_2}\tilde  I^{(q_2)}_{i_3,i_4}  \\
\chi^{(3)}_{n,r}
&=
\frac{4}{n^2\delta^{2}_n(3) } \sum_{q_1,q_2=1}^{r-1} \sum_{p_1=1}^{q_1-1}\sum_{p_2=1}^{q_2-1} \sum_{\mathbf{i}\in \mathcal{J}}\varphi_2\left(\mathbf{i}\right)\tilde I^{(p_1)}_{i_1,i_2}\tilde  I^{(q_1)}_{i_1,i_2}\tilde  I^{(p_2)}_{i_3,i_4} \tilde I^{(q_2)}_{i_3,i_4}.
\end{align*}
The first assertion in \eqref{eq:lconv-mult} is shown once we prove that
\begin{align} \label{eq:l1ll}
\Lambda_{n,1}^{(\ell, \ell')} 
:=
\sum_{r=1}^{d}\E\left[\chi^{(\ell)}_{n,r} \chi^{(\ell')}_{n,r}\right]=o(1)  \qquad \forall\, 1 \le \ell \le \ell' \le 3.
\end{align}

For $\ell=1<2=\ell'$, we have	
\begin{align} \label{eq:ln112}
\Lambda_{n,1}^{(1,2)}
&=
\frac{32}{n^4  \delta^{3}_n(2) \delta_n(3) } 
\sum_{r=1}^{d}\sum_{q_1,\ldots,q_4=1}^{r-1}\sum_{p_1=1}^{q_1-1}\sum_{ \mathbf{i}\in \mathcal{J}^2} 
\Psi_{\mathbf{i}} 
\E\left[\tilde I^{(p_1)}_{i_1,i_2}\tilde  I^{(q_1)}_{i_3,i_4}\tilde  I^{(q_2)}_{i_3,i_4}\tilde  I^{(q_3)}_{i_5,i_6}\tilde  I^{(q_4)}_{i_7,i_8}\right],
\end{align}
where, for $\mathbf i \in \mathcal J^2$,
\[
\Psi_{\mathbf i} = \varphi_2\left(\mathbf{i}_{1:4} \right) \cdot\varphi_2\left(\mathbf{i}_{5:8} \right).
\]		
Note that, by Lemma~\ref{lem:2m}, 
\begin{align} \label{eq:psisum}
\sum_{ \mathbf{i}\in \mathcal{J}^2}  |\Psi_{\mathbf{i}}| = \sum_{h,h' =2}^4 \sum_{ \mathbf{i}\in \mathcal I_h \times \mathcal I_{h'}}  |\Psi_{\mathbf{i}}|  =O(n^4).
\end{align}
In view of the 2-matching condition, see \eqref{eq:eprod} and \eqref{matching} for a related argumentation, the expectation in the sum in \eqref{eq:ln112} can only be non-zero when
$
| \{ p_1,q_1,q_2,q_3,q_4\}  | \le 2.
$
Hence, by \eqref{eq:psisum} and since $\delta^{3}_n(2)\delta_n(3) \propto  d^{\frac{9}{2}}$ by the definition of $\delta_n$ in \eqref{eq:deltan}, we have $\Lambda_{n,1}^{(1,2)}=O(d^{-\frac{3}{2}})=o(1)$. Next, we have 
\begin{align*}
\Lambda_{n,1}^{(1,3)}
&=
\frac{16}{n^4\delta^{2}_n(2) \delta^{2}_n(3)}
\sum_{r=1}^d \sum_{q_1,\ldots,q_4=1}^{r-1}\sum_{p_1=1}^{q_1-1}\sum_{p_2=1}^{q_2-1}\sum_{ \mathbf{i}\in \mathcal{J}^2} \Psi_{\mathbf{i}}\E\left[\tilde I^{(p_1)}_{i_1,i_2}\tilde  I^{(q_1)}_{i_1,i_2}\tilde  I^{(p_2)}_{i_3,i_4} \tilde I^{(q_2)}_{i_3,i_4}\tilde  I^{(q_3)}_{i_5,i_6}\tilde  I^{(q_4)}_{i_7,i_8}\right].
\end{align*}	
By $2$-matching and $6$-wise independence, each summand in the previous sum can only be nonzero when 
$
| \{ p_1,p_2,q_1,q_2,q_3,q_4\}  | \le 3.
$
Hence, \eqref{eq:psisum} and  $\delta^{2}_n(2)\delta^{2}_n(3) \propto d^{5}$  implies $\Lambda_{n,1}^{(1,3)}=O(d^{-1})=o(1)$. Next,
\begin{align*}
\Lambda_{n,1}^{(2,3)}
&=
\frac{32}{n^4\delta_n(2) \delta^{3}_n(3)}
\sum_{r=1}^d  \sum \limits_{\substack{ q_\ell=1 \\ \ell=1,\ldots,4 }}^{r-1}\sum \limits_{\substack{ p_\ell=1 \\ \ell=1,\ldots,3 }}^{q_\ell-1}\sum_{ \mathbf{i}\in \mathcal{J}^2} 
\Psi_{\mathbf{i}}\E\left[\tilde I^{(p_1)}_{i_1,i_2}\tilde  I^{(q_1)}_{i_1,i_2}\tilde  I^{(p_2)}_{i_3,i_4} \tilde I^{(q_2)}_{i_3,i_4}I^{(p_3)}_{i_5,i_6}\tilde  I^{(q_3)}_{i_5,i_6}\tilde  I^{(q_4)}_{i_7,i_8}\right].
\end{align*}
By $2$-matching and $7$-wise independence, each summand in the previous sum can only be nonzero when 
$
| \{ p_1,p_2,q_1,q_2,q_3,q_4\}  | \le 3.
$
Hence, by  \eqref{eq:psisum} and since $\delta_n(2)\delta^{3}_n(3) \propto d^{\frac{11}{2}}$, we obtain that  $\Lambda_{n,1}^{(2,3)}=O(d^{-3/2})=o(1)$. Next,
\begin{align*}
\Lambda_{n,1}^{(1,1)}
&=
\frac{16}{n^4\delta^{4}_n(2)}
\sum_{r=1}^d \sum_{q_1,\ldots,q_4=1}^{r-1}\sum_{\mathbf{i}\in\mathcal{J}^2} 
\Psi_{\mathbf{i}}\E\left[ \tilde  I^{(q_1)}_{i_1,i_2}\tilde  I^{(q_2)}_{i_3,i_4}I^{(q_3)}_{i_5,i_6}\tilde  I^{(q_4)}_{i_7,i_8}\right].
\end{align*}
By $2$-matching  and $4$-wise independence,  each summand in the previous sum can only be nonzero when $|\{ q_1,q_2,q_3,q_4\} |\le 2.$ Hence, by  \eqref{eq:psisum} and since $\delta_{n}^4(2) \propto d^4$, we obtain that $\Lambda_{n,1}^{(1,1)}=O(d^{-1})=o(1)$. Next,
\begin{align*}
\Lambda_{n,1}^{(2,2)}
=
\frac{64}{n^4\delta^{2}_n(2)\delta^{2}_n(3)}
\sum_{r=1}^d \sum_{q_1,\ldots,q_4=1}^{r-1}\sum_{p_1=1}^{q_1-1}\sum_{p_3=1}^{q_3-1}\sum_{\mathbf{i}\in\mathcal{J}^2} \Psi_{\mathbf{i}}\E\left[\tilde I^{(p_1)}_{i_1,i_2}\tilde  I^{(q_1)}_{i_1,i_2}\tilde  I^{(q_2)}_{i_3,i_4}  	\tilde I^{(p_3)}_{i_5,i_6}\tilde  I^{(q_3)}_{i_5,i_6}\tilde  I^{(q_4)}_{i_7,i_8}\right].
\end{align*}
By $2$-matching  and $6$-wise independence, each summand in the previous sum can only be nonzero when $| \{ p_1,p_3,q_1,q_2,q_3,q_4\}  |\le 3.$ Hence, by  \eqref{eq:psisum} and since $\delta^{2}_n(2)\delta^{2}_n(3) \propto d^5$, we obtain that $\Lambda_{n,1}^{(2,2)}=O(d^{-1})=o(1)$. Finally,
\begin{align*}
\Lambda_{n,1}^{(3,3)}
=
\frac{16}{n^4\delta^{4}_n(3) }
\sum_{r=1}^d  \sum_{q_1,\ldots,q_4=1}^{r-1}\sum\limits_{\substack{p_\ell=1 \\\ell =1,\ldots,4 }}^{q_\ell-1}\sum_{\mathbf{i}\in\mathcal{J}^2} \Psi_{\mathbf{i}}
\E\left[\tilde I^{(p_1)}_{i_1,i_2}\tilde  I^{(q_1)}_{i_1,i_2}\tilde  I^{(p_2)}_{i_3,i_4} \tilde I^{(q_2)}_{i_3,i_4}
\tilde I^{(p_3)}_{i_5,i_6}\tilde  I^{(q_3)}_{i_5,i_6}\tilde  I^{(p_4)}_{i_7,i_8} \tilde I^{(q_4)}_{i_7,i_8}\right].
\end{align*}
By $2$-matching and $8$-wise independence and centering, each summand in the previous sum can only be nonzero when $|\{ p_1,p_2,p_3,p_4,q_1,q_2,q_3,q_4\}  | \le 4.$  Hence, by  \eqref{eq:psisum} and since $\delta^{4}_n(3) \propto d^6$, we obtain that $\Lambda_{n,1}^{(3,3)}=O(d^{-1})=o(1)$. As a summary, we have shown \eqref{eq:l1ll} and hence the first assertion in \eqref{eq:lconv-mult}.
	
For proving the second assertion in \eqref{eq:lconv-mult}, we make use of \eqref{eq:xidecomp} again and write	
\[
\Lambda_{n,2} 
=
\sum_{r \ne r'}^d \E \left[\chi_{n,r} \chi_{n,r'} \right] 
=
\sum_{\ell,\ell' =1}^3  \sum_{r \ne r'}^d \E\left[\chi^{(\ell )}_{n,r} \chi^{(\ell')}_{n,r'}\right]
=: 
\sum_{\ell,\ell' =1}^3\Lambda_{n,2}^{(\ell, \ell')}
\]
It is then to sufficient to show that
\begin{align}
\label{eq:l2lla}
&\lim_{n\to\infty} \Lambda_{n,2}^{(1,1)} = \lim_{n\to\infty} \Lambda_{n,2}^{(1,3)} = \lim_{n\to\infty} \Lambda_{n,2}^{(3,3)}=1,  \\
\label{eq:l2llb}
&\lim_{n\to\infty} \Lambda_{n,2}^{(1,2)} = \lim_{n\to\infty} \Lambda_{n,2}^{(2,2)} = \lim_{n\to\infty} \Lambda_{n,2}^{(2,3)}=0.
\end{align}

We start by proving \eqref{eq:l2lla}. First, $\Lambda_{n,2}^{(1,1)} / 2$ and $\Lambda_{n,2}^{(3,3)} / 2$ can be seen to be equal to $\Lambda_{n,2}(2)$ and $\Lambda_{n,2}(3)$ from \eqref{eq:ln2k}, respectively, both of which were shown to converge to $\frac{1}{2}$ in the proof of Lemma \ref{lem_cond_exp}. It remains to treat $\Lambda_{n,2}^{(1,3)}$, which can be written as 
\begin{align}
\label{eq:ln213}
\Lambda_{n,2}^{(1,3)}
&=
\frac{16}{n^4\delta^2_n(2)\delta^2_n(3)}
\sum_{r \ne r'}^d \sum_{q_1,q_2=1}^{r-1}\sum_{q_3,q_4=1}^{r'-1} \sum_{p_1=1}^{q_1-1}\sum_{p_2=1}^{q_2-1}
 A(p_1, q_1, p_2, q_2, q_3, q_4)
\end{align}
where
\[
A(p_1, q_1, p_2, q_2, q_3, q_4)
=
\sum_{\mathbf{i}\in \mathcal{J}^2}\Psi_{\mathbf{i}}\E\left[\tilde I^{(p_1)}_{i_1,i_2}\tilde  I^{(q_1)}_{i_1,i_2}\tilde  I^{(p_2)}_{i_3,i_4} \tilde I^{(q_2)}_{i_3,i_4}
\tilde  I^{(q_3)}_{i_5,i_6} \tilde I^{(q_4)}_{i_7,i_8}\right].
\]
The summands in \eqref{eq:ln213} can only be non-zero if $|\{p_1, q_1, p_2, q_2, q_3, q_4\}| \le 3$ and if each element of the previous set appears at least twice. The sum restricted to those summands for which $|\{p_1, q_1, p_2, q_2, q_3, q_4\}| \le 2$ can be bounded by
\[
\frac{16}{n^4\delta^2_n(2)\delta^2_n(3)}
O(d^{4}) \sum_{\mathbf{i}\in \mathcal{J}^2} |\Psi_{\mathbf{i}}|,
\]
which is of the order $O(d^{-1})=o(1)$ by \eqref{eq:psisum} and $\delta^2_n(2)\delta^2_n(3) \propto d^5$. Hence, it is sufficient to consider the sum in \eqref{eq:ln213} restricted to $|\{p_1, q_1, p_2, q_2, q_3, q_4\}| = 3$ where each element of the previous set appears exactly twice.  We further decompose that sum into two cases: either $p_1=p_2, q_1=q_2, q_3=q_4$, or any other combination of matchings applies (in which case, necessarily, $q_3 \ne q_4$). The resulting sum over those indices for which any other matching applies is negligible: indeed, the summands in \eqref{eq:ln213} must be of the form 
\[
A(p_1, q_1, p_2, q_2, q_3, q_4) = \sum_{\mathbf i \in \mathcal J} \Psi_{\mathbf i} \cdot \varphi_2(\mathbf i_{1:4}) \varphi_2(\mathbf i_{\{1,2,5,6\}}) \varphi_2(\mathbf i_{\{1,2,7,8\}}),
\]
with each summand being of the order
$O(n^{10-2|\mathbf i_{1:4}|- |\mathbf i_{5:8}|  - | \mathbf i_{\{1,2,5,6\}}| - | \mathbf i_{\{1,2,7,8\}} | })$ by Lemma~\ref{lem:2m}. Decomposing the sum over $\mathbf i \in \mathcal J^2$ into $\mathbf i \in \mathcal I_h \times \mathcal I_{h'}$ with $h,h' \in \{2,3,4\}$, it can be seen by a careful case-by-case analysis that 
$A(p_1, q_1, p_2, q_2, q_3, q_4) =o(n^4)$, uniformly.
As a consequence, we have shown that $\Lambda_{n,2}^{(1,3)}=\tilde \Lambda_{n,2}^{(1,3)} + o(1)$, where $\tilde \Lambda_{n,2}^{(1,3)}$ is defined just as in \eqref{eq:ln213}, but with the sum additionally restricted to $p_1=p_2, q_1=q_2, q_3=q_4$. It may be rewritten as
\[
\tilde \Lambda_{n,2}^{(1,3)}
=
\frac{16}{n^4\delta^2_n(2)\delta^2_n(3)}
\sum_{r \ne r'}^d \sum_{q_1=1}^{r-1} \sum_{p_1=1}^{q_1-1}\sum_{\substack{q_3=1 \\ q_3 \ne p_1, q_1}}^{r'-1}
\sum_{\mathbf{i}\in \mathcal{J}^2}
\varphi_2(\mathbf i_{1:4})^3\varphi_2(\mathbf i_{5:8})^2.
\]
Now, a further case by case analysis according to the number of equal indices among $\mathbf i_{1:4}$ and $\mathbf i_{5:8}$ shows that
\[
\sum_{\mathbf{i}\in \mathcal{J}^2}
\varphi_2(\mathbf i_{1:4})^3\varphi_2(\mathbf i_{5:8})^2
=
\sum_{\mathbf{i}\in \mathcal{I}_2^2}
\varphi_2(\mathbf i_{1:4})^3\varphi_2(\mathbf i_{5:8})^2 + o(n^4)
=
\frac{n^4}{4} \cdot \frac{1}{90^5} (1+o(1)),
\]
where the last equation follows from \eqref{cardinal} and Lemma~\ref{lem:2m}.
Moreover,
\[
\delta_n^2(2)\delta_n^2(3)= \frac{4 d^5}{90^5\cdot 12}(1+o(1))
\]
by the definition of $\delta_n$ in \eqref{eq:deltan}
and
\[
\sum_{r \ne r'}^d 
\sum_{q_1=1}^{r-1}  
\sum_{p_1=1}^{q_1-1}  
\sum_{\substack{q_3=1 \\ q_3 \ne p_1, q_1}}^{r'-1}
1
= 
\frac{d^5}{12}(1+o(1)).
\]
As a consequence of the last four  displays, we obtain that
\[
\tilde \Lambda_{n,2}^{(1,3)}
= 1+o(1),
\]
which implies the second assertion in \eqref{eq:l2lla}.

It remains to prove \eqref{eq:l2llb}.  Note that
\begin{align*}
\Lambda_{n,2}^{(1,2)}
&= 
\frac{32}{n^4\delta^3_n(2) \delta_n(3)}
\sum_{r < r'}^d \sum_{q_1,q_2=1}^{r-1}\sum_{q_3,q_4=1}^{r'-1}\sum_{p_3=1}^{q_3-1} 
\sum_{\mathbf{i}\in \mathcal{J}^2} \Psi_{\mathbf{i}}
\E\left[ \tilde  I^{(q_1)}_{i_1,i_2}\tilde  I^{(q_2)}_{i_3,i_4}\tilde  I^{(p_3)}_{i_5,i_6}\tilde  I^{(q_3)}_{i_5,i_6} \tilde  I^{(q_4)}_{i_7,i_8} \right], \\
\Lambda_{n,2}^{(2,2)}
&=
\frac{64}{n^4\delta^2_n(2)\delta^2_n(3)}
\sum_{r < r'}^d \sum_{q_1,q_2=1}^{r-1}\sum_{q_3,q_4=1}^{r'-1}\sum_{p_1=1}^{q_1-1} \sum_{p_3=1}^{q_3-1} 
\sum_{\mathbf{i}\in \mathcal{J}^2}\Psi_{\mathbf{i}}
\E\left[ \tilde  I^{(p_1)}_{i_1,i_2}\tilde  I^{(q_1)}_{i_1,i_2}\tilde  I^{(q_2)}_{i_3,i_4}\tilde  I^{(p_3)}_{i_5,i_6}\tilde  I^{(q_3)}_{i_5,i_6} \tilde  I^{(q_4)}_{i_7,i_8} \right], \\
\Lambda_{n,2}^{(2,3)}
&= 
\frac{32}{n^4\delta_n(2)\delta^3_n(3)}
\sum_{r < r'}^d \sum_{q_1,q_2=1}^{r-1}\sum_{q_3,q_4=1}^{r'-1}\sum_{p_1=1}^{q_1-1} \sum_{p_2=1}^{q_2-1}  \sum_{p_3=1}^{q_3-1} 
\sum_{\mathbf{i}\in \mathcal{J}^2}\Psi_{\mathbf{i}} \\
& \hspace{7.1cm} \times
\E\left[ \tilde  I^{(p_1)}_{i_1,i_2}\tilde  I^{(q_1)}_{i_1,i_2} \tilde  I^{(p_2)}_{i_3,i_4} \tilde  I^{(q_2)}_{i_3,i_4}\tilde  I^{(p_3)}_{i_5,i_6}\tilde  I^{(q_3)}_{i_5,i_6} \tilde  I^{(q_4)}_{i_7,i_8} \right], 
\end{align*}
We may use the same arguments that were used for $\Lambda_{n,1}^{(1,2)}$ and $\Lambda_{n,1}^{(2,3)}$ to show that $\Lambda_{n,2}^{(1,2)}$ and $\Lambda_{n,2}^{(2,3)}$ are of the order $O(d^{-\frac{1}{2}})$ (the only essential difference is the double sum over $r<r'$ instead of a simple sum over $r$, which yields an additional factor $d$). It remains to treat $\Lambda_{n,2}^{(2,2)}$. Note that its summands can only be non-zero if $|\{p_1, q_1, q_2, p_3, q_3, q_4\}| \le 3$ and if each element of the previous set appears at least twice. Hence, it is sufficient to show that
\[
\sum_{\mathbf{i}\in \mathcal{J}^2}\Psi_{\mathbf{i}}\E\left[ \tilde  I^{(p_1)}_{i_1,i_2}\tilde  I^{(q_1)}_{i_1,i_2}\tilde  I^{(q_2)}_{i_3,i_4}\tilde  I^{(p_3)}_{i_5,i_6}\tilde  I^{(q_3)}_{i_5,i_6} \tilde  I^{(q_4)}_{i_7,i_8} \right] = o(n^4),
\]
uniformly in those $p_1, q_1, q_2, p_3, q_3, q_4$. The latter can again be shown by a careful case-by-case analysis. 
\end{proof}

\begin{lemma} \label{lem:lya3}
Assume $9$-wise independence. Then, with $X_{n,r}$ from \eqref{eq:xnr-mult},
\[
\Theta_{n}:=
\E \left| \sum_{r=1}^{d}\E\left( X^4_{n,r} | \mathcal{F}_{n,r-1}  \right)  \right| 
=
\sum_{r=1}^{d}\E\left( X^4_{n,r} \right)  \xrightarrow[n\to +\infty]{} 0 .
\]
\end{lemma}

\begin{proof}
Recalling $X_{n,r}$ from \eqref{eq:xnr-mult} and using that $(a+b)^4 = a^4+b^4 +3a^3b + 3ab^3 + 6a^2b^2$, we may write 
\[
\Theta_{n}  =  \Theta_{n}^{(1)} + \Theta_{n}^{(2)} + 3 \Theta_{n}^{(3)} + 3 \Theta_{n}^{(4)} + 6 \Theta_{n}^{(5)} ,
\]
where
\begin{align*}
\Theta_{n}^{(1)} 
&=
\delta^{-4}_n(2) \sum_{r=1}^{d}\E \left[\left(  \sum_{q =1}^{r-1}  \tilde M_{n,\{ q,r \}} \right)^4 \right] \\
\Theta_{n}^{(2)}  
&=
\delta^{-4}_n(3)\sum_{r=1}^{d}\E \left[ \left(\sum_{q =1}^{r-1} \sum_{ p=1}^{q-1} \tilde M_{n,\{p,q,r\}} \right)^4 \right] \\
\Theta_{n}^{(3)} 
&= 
\delta^{-3}_n(2)\delta^{-1}_n(3) \sum_{r=1}^{d}\E \left[\left(  \sum_{q =1}^{r-1} \tilde M_{n,\{ q,r \}}\right)^3 \cdot \left(\sum_{q =1}^{r-1} \sum_{ p=1}^{q-1} \tilde M_{n,\{p,q,r\}} \right)   \right] \\
\Theta_{n}^{(4)} 
&=
 \delta^{-1}_n(2)\delta^{-3}_n(3)\sum_{r=1}^{d} \E \left[\left(\sum_{q =1}^{r-1} \sum_{ p=1}^{q-1} \tilde M_{n,\{p,q,r\}}\right)^3 \cdot \left(   \sum_{q =1}^{r-1} \tilde M_{n,\{ q,r \}}\right)   \right] \\
 \Theta_{n}^{(5)} 
&=
6\delta^{-2}_n(2) \delta^{-2}_n(3) \sum_{r=1}^{d} \E \left[\left(  \sum_{q =1}^{r-1} \tilde M_{n,\{ q,r \}}\right)^2 \cdot \left(\sum_{q =1}^{r-1} \sum_{ p=1}^{q-1} \tilde M_{n,\{p,q,r\}}\right)^2   \right].
\end{align*}
Lemma~\ref{lyapunovlem} implies that $\lim_{n \to \infty} \Theta_{n}^{(1)}= \lim_{n \to \infty} \Theta_{n}^{(2)}=0$; note that the latter requires 9-wise independence. It remains to consider the other three expressions. For the sake of brevity, we only consider the hardest one, which is $\Theta_{n}^{(4)}$. Recalling $\tilde M_{n,A}$ from \eqref{eq:mnnn}, we have
\begin{align*}
\Theta^{(4)}_{n}
&= 
\frac{16}{n^4\delta_n(2) \delta^{3}_n(3) }
\sum_{r=1}^d 
\sum_{\substack{q_j= 1\\ j=1, \dots, 4}}^{r-1}
\sum_{\substack{p_j= 1\\ j=1, \dots, 3}}^{r-1}
 \sum_{\mathbf{i} \in\mathcal{J}^2} \\
& \hspace{5cm} 
\E \Big[ \tilde I^{(p_1)}_{i_1,i_2}\tilde I^{(q_1)}_{i_1,i_2} \tilde I^{(p_2)}_{i_3,i_4}\tilde I^{(q_2)}_{i_3,i_4} \tilde I^{(p_3)}_{i_5,i_6}\tilde I^{(q_3)}_{i_5,i_6}\tilde  I^{(q_4)}_{i_7,i_8} \tilde I^{(r)}_{i_1,i_2}\tilde I^{(r)}_{i_3,i_4}\tilde I^{(r)}_{i_5,i_6}\tilde I^{(r)}_{i_7,i_8}\Big] \\
&=
\frac{16}{n^4\delta_n(2)\delta^3_n(3) } 
\sum_{r=1}^d 
\sum_{\substack{q_j= 1\\ j=1, \dots, 4}}^{r-1}
\sum_{\substack{p_j= 1\\ j=1, \dots, 3}}^{r-1}
\sum_{\mathbf{i} \in\mathcal{J}^2}\E \left[ \tilde I^{(p_1)}_{i_1,i_2}\tilde I^{(q_1)}_{i_1,i_2} \tilde I^{(p_2)}_{i_3,i_4}\tilde I^{(q_2)}_{i_3,i_4} \tilde I^{(p_3)}_{i_5,i_6}\tilde I^{(q_3)}_{i_5,i_6}\tilde I^{(q_4)}_{i_7,i_8}\right]\cdot \varphi_{4}\left(\mathbf{i} \right).
\end{align*}
The summands in the previous sum can only be non-zero if $\zp[b]{\{ p_1,p_2,p_3,q_1, q_2, q_3, q_4, q_4  \}}\leq 3.$ Moreover, $\sum_{\mathbf i \in \mathcal J} |\varphi_4(\mathbf i)| = O(n^4)$ by Lemma~\ref{lem:4m}. As a consequence, since $\delta_n(2) \delta^{3}_n(3) \propto d^{\frac{11}{2}}$,  
\[
\Theta^{(4)}_{n} = O\left(d^{-\frac{3}{2}} \right) =o(1). \qedhere
\]
\end{proof}

\bibliographystyle{chicago}
\bibliography{biblio}

\begin{thebibliography}{}

\bibitem[\protect\citeauthoryear{Cai and Jiang}{Cai and Jiang}{2011}]{CaiJia11}
Cai, T.~T. and T.~Jiang (2011).
\newblock Limiting laws of coherence of random matrices with applications to
  testing covariance structure and construction of compressed sensing matrices.
\newblock {\em Ann. Statist.\/}~{\em 39\/}(3), 1496--1525.

\bibitem[\protect\citeauthoryear{Canfield and McKay}{Canfield and
  McKay}{2005}]{CanfieldMckay}
Canfield, E. and B.~McKay (2005).
\newblock Asymptotic enumeration of dense 0-1 matrices with equal row sums and
  equal column sums.
\newblock {\em Electr. J. Comb.\/}~{\em 12}.

\bibitem[\protect\citeauthoryear{Chen, Zhang, and Zhong}{Chen
  et~al.}{2010}]{CheZhaZho10}
Chen, S.~X., L.-X. Zhang, and P.-S. Zhong (2010).
\newblock Tests for high-dimensional covariance matrices.
\newblock {\em J. Amer. Statist. Assoc.\/}~{\em 105\/}(490), 810--819.

\bibitem[\protect\citeauthoryear{Cormen, Leiserson, Rivest, and Stein}{Cormen
  et~al.}{2009}]{Cor09}
Cormen, T.~H., C.~E. Leiserson, R.~L. Rivest, and C.~Stein (2009).
\newblock {\em Introduction to algorithms}.
\newblock MIT press.

\bibitem[\protect\citeauthoryear{Deheuvels}{Deheuvels}{1979}]{Deh79}
Deheuvels, P. (1979).
\newblock La fonction de d\'ependance empirique et ses propri\'et\'es. {U}n
  test non param\'etrique d'ind\'ependance.
\newblock {\em Acad. Roy. Belg. Bull. Cl. Sci. (5)\/}~{\em 65\/}(6), 274--292.

\bibitem[\protect\citeauthoryear{Deheuvels}{Deheuvels}{1981a}]{Deh81a}
Deheuvels, P. (1981a).
\newblock An asymptotic decomposition for multivariate distribution-free tests
  of independence.
\newblock {\em J. Multivariate Anal.\/}~{\em 11\/}(1), 102--113.

\bibitem[\protect\citeauthoryear{Deheuvels}{Deheuvels}{1981b}]{Deh81b}
Deheuvels, P. (1981b).
\newblock A {K}olmogorov-{S}mirnov type test for independence and multivariate
  samples.
\newblock {\em Rev. Roumaine Math. Pures Appl.\/}~{\em 26\/}(2), 213--226.

\bibitem[\protect\citeauthoryear{Drton, Han, and Shi}{Drton
  et~al.}{2020}]{Drt20}
Drton, M., F.~Han, and H.~Shi (2020).
\newblock {High-dimensional consistent independence testing with maxima of rank
  correlations}.
\newblock {\em The Annals of Statistics\/}~{\em 48\/}(6), 3206 -- 3227.

\bibitem[\protect\citeauthoryear{Fang, Fang, and Kotz}{Fang
  et~al.}{2002}]{Fan02}
Fang, H.-B., K.-T. Fang, and S.~Kotz (2002).
\newblock The meta-elliptical distributions with given marginals.
\newblock {\em J. Multivariate Anal.\/}~{\em 82\/}(1), 1--16.

\bibitem[\protect\citeauthoryear{Geisser and Mantel}{Geisser and
  Mantel}{1962}]{GeiMan62}
Geisser, S. and N.~Mantel (1962).
\newblock Pairwise independence of jointly dependent variables.
\newblock {\em Ann. Math. Statist.\/}~{\em 33}, 290--291.

\bibitem[\protect\citeauthoryear{Genest, Ne\v{s}lehov\'{a}, R\'{e}millard, and
  Murphy}{Genest et~al.}{2019}]{Gen19}
Genest, C., J.~G. Ne\v{s}lehov\'{a}, B.~R\'{e}millard, and O.~A. Murphy (2019).
\newblock Testing for independence in arbitrary distributions.
\newblock {\em Biometrika\/}~{\em 106\/}(1), 47--68.

\bibitem[\protect\citeauthoryear{Genest, Quessy, and Remillard}{Genest
  et~al.}{2007}]{GenQueRem07}
Genest, C., J.-F. Quessy, and B.~Remillard (2007).
\newblock Asymptotic local efficiency of {C}ram\'{e}r-von {M}ises tests for
  multivariate independence.
\newblock {\em Ann. Statist.\/}~{\em 35\/}(1), 166--191.

\bibitem[\protect\citeauthoryear{Genest and R{\'e}millard}{Genest and
  R{\'e}millard}{2004}]{GenRem04}
Genest, C. and B.~R{\'e}millard (2004).
\newblock Tests of independence and randomness based on the empirical copula
  process.
\newblock {\em Test\/}~{\em 13\/}(2), 335--370.

\bibitem[\protect\citeauthoryear{Hall and Heyde}{Hall and
  Heyde}{1980}]{hall1980martingale}
Hall, P. and C.~Heyde (1980).
\newblock {\em Martingale Limit Theory and Its Application}.
\newblock Probability and mathematical statistics. Academic Press.

\bibitem[\protect\citeauthoryear{Han, Chen, and Liu}{Han
  et~al.}{2017}]{HanCheLiu17}
Han, F., S.~Chen, and H.~Liu (2017).
\newblock Distribution-free tests of independence in high dimensions.
\newblock {\em Biometrika\/}~{\em 104\/}(4), 813--828.

\bibitem[\protect\citeauthoryear{Han and Wu}{Han and Wu}{2020}]{HanWu20}
Han, Y. and W.~B. Wu (2020).
\newblock Test for high dimensional covariance matrices.
\newblock {\em Ann. Statist.\/}~{\em 48\/}(6), 3565--3588.

\bibitem[\protect\citeauthoryear{Jiang and Qi}{Jiang and Qi}{2015}]{JiaQi15}
Jiang, T. and Y.~Qi (2015).
\newblock Likelihood ratio tests for high-dimensional normal distributions.
\newblock {\em Scand. J. Stat.\/}~{\em 42\/}(4), 988--1009.

\bibitem[\protect\citeauthoryear{Kojadinovic and Holmes}{Kojadinovic and
  Holmes}{2009}]{KojHol09}
Kojadinovic, I. and M.~Holmes (2009).
\newblock Tests of independence among continuous random vectors based on
  {C}ram\'{e}r-von {M}ises functionals of the empirical copula process.
\newblock {\em J. Multivariate Anal.\/}~{\em 100\/}(6), 1137--1154.

\bibitem[\protect\citeauthoryear{Ledoit and Wolf}{Ledoit and
  Wolf}{2002}]{LedWol02}
Ledoit, O. and M.~Wolf (2002).
\newblock Some hypothesis tests for the covariance matrix when the dimension is
  large compared to the sample size.
\newblock {\em Ann. Statist.\/}~{\em 30\/}(4), 1081--1102.

\bibitem[\protect\citeauthoryear{Leung and Drton}{Leung and
  Drton}{2018}]{LeuDrt18}
Leung, D. and M.~Drton (2018).
\newblock Testing independence in high dimensions with sums of rank
  correlations.
\newblock {\em Ann. Statist.\/}~{\em 46\/}(1), 280--307.

\bibitem[\protect\citeauthoryear{Nelsen}{Nelsen}{2006}]{Nel06}
Nelsen, R.~B. (2006).
\newblock {\em An introduction to copulas\/} (Second ed.).
\newblock Springer Series in Statistics. New York: Springer.

\bibitem[\protect\citeauthoryear{Schott}{Schott}{2005}]{Sch05}
Schott, J.~R. (2005).
\newblock Testing for complete independence in high dimensions.
\newblock {\em Biometrika\/}~{\em 92\/}(4), 951--956.

\bibitem[\protect\citeauthoryear{Segers}{Segers}{2012}]{Seg12}
Segers, J. (2012).
\newblock Asymptotics of empirical copula processes under non-restrictive
  smoothness assumptions.
\newblock {\em Bernoulli\/}~{\em 18\/}(3), 764--782.

\bibitem[\protect\citeauthoryear{Sklar}{Sklar}{1959}]{Skl59}
Sklar, A. (1959).
\newblock Fonctions de r{\'{e}}partition {\`{a}} {$n$} dimensions et leurs
  marges.
\newblock {\em Publ.\ Inst.\ Statist.\ Univ.\ Paris\/}~{\em 8}, 229--231.

\bibitem[\protect\citeauthoryear{Stute}{Stute}{1984}]{Stu84}
Stute, W. (1984).
\newblock The oscillation behavior of empirical processes: the multivariate
  case.
\newblock {\em Ann. Probab.\/}~{\em 12\/}(2), 361--379.

\bibitem[\protect\citeauthoryear{Yao, Zhang, and Shao}{Yao
  et~al.}{2018}]{Yao18}
Yao, S., X.~Zhang, and X.~Shao (2018).
\newblock Testing mutual independence in high dimension via distance
  covariance.
\newblock {\em J. R. Stat. Soc. Ser. B. Stat. Methodol.\/}~{\em 80\/}(3),
  455--480.

\end{thebibliography}

\end{document}